\UseRawInputEncoding  

\documentclass[11pt,a4paper,reqno]{amsart}

\usepackage{setspace}
\usepackage{fullpage}
\usepackage{datetime}
\usepackage{xcolor}

\usepackage{enumitem}
\usepackage{amsmath}%
\usepackage{amsfonts}%
\usepackage{amssymb}%
\usepackage{graphicx}
\usepackage{mathrsfs}
\usepackage{hyperref}
\usepackage[margin=2.5cm]{geometry}

\usepackage[normalem]{ulem}

\usepackage{cite}

\usepackage[abbrev,msc-links,backrefs]{amsrefs}
\newtheorem{theorem}{Theorem}
\theoremstyle{plain}

\newtheorem{claim}[theorem]{Claim}

\newtheorem{conjecture}[theorem]{Conjecture}
\newtheorem{corollary}[theorem]{Corollary}
\newtheorem{definition}[theorem]{Definition}

\newtheorem{lemma}[theorem]{Lemma}

\newtheorem{question}[theorem]{Question}
\newtheorem{remark}[theorem]{Remark}
\numberwithin{theorem}{section}

\def\COMMENT#1{}
\let\COMMENT=\footnote
\title{Transversal Hamilton paths and cycles}

\author[Y.~Cheng]{Yangyang Cheng}
\address{(Y. Cheng) Mathematical Institute, University of Oxford, Oxford, UK}
\email{yangyang.cheng@maths.ox.ac.uk}

\author[W.~Sun]{Wanting Sun}
\address{(W. Sun) Data Science Institute, Shandong University, Jinan, 250100, China}
\email{wtsun2018@sina.com}

\author[G.~Wang]{Guanghui Wang}
\address{(G. Wang) School of Mathematics, Shandong University, Jinan, 250100, China}
\email{ghwang@sdu.edu.cn}

\author[L.~Wei]{Lan Wei}
\address{(L.~Wei) School of Mathematics, Shandong University, Jinan, 250100, China}
\email{lanwei@mail.sdu.edu.cn}

\thanks{Y. Cheng is  supported by the PhD studentship of ERC Advanced Grant (883810). W. Sun is  supported by the China Postdoctoral Science Foundation (2023M742092). G. Wang and L. Wei are supported by  the Natural Science Foundation of China (12231018) and Young Taishan
Scholars Program of Shandong Province.}

\begin{document}

\begin{abstract}
Given a collection $\mathcal{G} =\{G_1,G_2,\dots,G_m\}$ of graphs on the common vertex set $V$ of size $n$, an $m$-edge graph $H$
on the same vertex set $V$ is \textit{transversal} in $\mathcal{G}$ if there exists a bijection $\varphi :E(H)\rightarrow [m]$ such that 
$e \in E(G_{\varphi(e)})$ for all $e\in E(H)$. Denote $\delta(\mathcal{G}):=\operatorname*{min}\left\{\delta(G_i): i\in [m]\right\}$. In this paper, we first establish  a minimum degree  condition for the existence of transversal Hamilton paths in $\mathcal{G}$: if $n=m+1$ and $\delta(\mathcal{G})\geq \frac{n-1}{2}$, then $\mathcal{G}$ contains a transversal Hamilton path. This solves a problem proposed by [Li, Li and Li, J. Graph Theory, 2023]. As a continuation of the transversal version of Dirac's theorem [Joos and Kim, Bull. Lond. Math. Soc., 2020] and the  stability result for transversal Hamilton cycles [Cheng and Staden, arXiv:2403.09913v1], our second result characterizes all graph collections with minimum degree at least {$\frac{n}{2}-1$} and without transversal Hamilton cycles. We obtain an analogous result for transversal Hamilton paths. The proof is a combination of the stability result for transversal Hamilton paths or cycles, transversal blow-up lemma, along with some structural analysis. 
\end{abstract}
\maketitle
\section{Introduction}
Let $G=(V(G),E(G))$ be a graph with vertex set $V(G)$ and edge set $E(G)$. For a vertex $v$ of $G$, denote by $N_G(v)$ the neighborhood of $v$ 
and let $d_G(v)=|N_G(v)|$ be the degree of $v$. 
The subscripts are omitted when $G$ is clear from the context. Let $\delta(G)$ be the minimum degree of $G$. Denote by  $P_n,\,C_n,\,K_n$ and $K_{t,n-t}$ the path, cycle, complete graph and complete bipartite graph (with parts of size $t$ and $n-t$) on $n$ vertices, respectively. 

Let $\mathcal{G} =\{G_1,G_2,\dots,G_m\}$ be a collection of not necessarily distinct graphs with common vertex set $V$. We often think of each $G_i$ having the color $i$. Let $H$ be 
a graph on the vertex set $V$. We say $H$ is \textit{rainbow} in $\mathcal{G}$ if there exists an  injection $\varphi :E(H)\rightarrow [m]$ such that $e \in E(G_{\varphi(e)})$ for all $e\in E(H)$. In addition, if $|E(H)|=m$, then $\varphi$ is a bijection and we say $H$ is \textit{transversal} in $\mathcal{G}$. 
Let $\delta(\mathcal{G})=\min\{\delta(G_i):i\in [m]\}$ be the minimum degree of $\mathcal{G}$.  
Transversal often appears in infinitary combinatorics under several similar definitions (see, \cites{Aharoni1983,Erdos1968}), and it is also extensively studied in the context of Latin squares (see \cite{Wanless2011} for a survey). 

The following general question was proposed by Joos and Kim in \cite{2021jooskim}.
\begin{question}
    Let $H$ be a graph with $m$ edges, $\mathbf{G}$ be a family of graphs, and $\mathcal{G} =\{G_1,G_2,\dots,G_m\}$ be a collection of graphs with common vertex set $V$ such that $G_i\in \mathbf{G}$ for all $i\in [n]$. Which properties imposed on $\mathbf{G}$ guarantee a transversal copy of $H$?
\end{question}
By taking $G_1=G_2=\cdots=G_m$, we need to study properties for $\mathbf{G}$ such that each graph in $\mathbf{G}$ contains $H$ as a subgraph. However, this alone is not always sufficient. For example, Aharoni, DeVos, de la Maza, Montejano and \v{S}\'{a}mal 
\cite{2019Aharoni} proved that if $\mathcal{G}=\{G_1, G_2,G_3\}$  is a collection of graphs on a common vertex set of size $n$ and $|E(G_i)|>(\frac{26-2\sqrt{7}}{81})n^2$ for $i\in [3]$, then  $\mathcal{G}$ contains a rainbow triangle. Moreover,  the constant $\frac{26-2\sqrt{7}}{81}$ is optimal. However, Mantel's theorem states that any $n$-vertex graph with more than $\lfloor\frac{n^2}{4}\rfloor$ edges must contain a triangle.

In this paper, we investigate the  minimum degree condition for a graph collection to guarantee the existence of  transversal Hamilton paths or cycles. Hamiltonicity of graphs is one of the fundamental problems in extremal graph theory and structural graph theory. In 1952, Dirac \cite{1952dirac} proved that an $n$-vertex graph contains a Hamilton cycle if its minimum degree is no less than $\frac{n}{2}$. Ore \cite{1960Ore} relaxed this condition by  considering the sum of degrees of two non-adjacent vertices. In general, there are many other sufficient conditions that guarantee the Hamiltonicity of a graph, such as P{\'o}sa's condition \cite{1962PosaCon}, Bondy's condition \cite{1971BondyCon} and so on. For more problems and results about Hamiltonicity of graphs, we refer the reader to \cites{1980jacksonHam,2003AdvancesHam,2005RahmanHam,2014GouldHam,2017McdiarmidHam}. 

In 2020, Aharoni \cite{2019Aharoni} conjectured that Dirac's theorem can be extended to a transversal version.
\begin{conjecture}[\cite{2019Aharoni}]\label{Ahaconj}
Suppose $\mathcal{G} = \{G_1, \ldots, G_n\}$ is a collection of graphs with the same vertex set $V$ of size $n$. If $\delta(\mathcal{G})\geq \frac{n}{2}$, then $\mathcal{G}$ contains a transversal Hamilton cycle.
\end{conjecture}
Cheng, Wang and Zhao \cite{2021ChengWangZhao} solved  this conjecture asymptotically, and it was completely confirmed by Joos and Kim \cite{2021jooskim}. 
Bradshaw, Halasz and Stacho \cite{2022BradshawNum} extended the result of Conjecture \ref{Ahaconj} by showing any such graph collection has at least $(\frac{cn}{e})^{cn}$ transversal Hamilton cycles for some constant $c>\frac{1}{68}$. Anastos and Chakraborti \cite{2023Anastosrobust} improved $(\frac{cn}{e})^{cn}$ to $(Cn)^{2n}$ for some constant $C>0$.  Bradshaw \cite{2021Bradshaw} studied the Hamiltonicity in bipartite graph collections. Bowtell, Morris, Pehova and Staden \cite{2023universality} showed that $\mathcal{G}$ contains every Hamilton cycle pattern if $\delta(\mathcal{G})\geq(1+o(1))\frac{n}{2}$. 
The research of graph collections has also been generalized to random graph collections. Ferber, Han and Mao \cite{2022FerberRandom} provided a transversal version of the Dirac's theorem in random graph collections.

In addition to Dirac's theorem, many other classical results of extremal graph theory have been generalized. 
Cheng, Han, Wang and Wang \cite{2023chengspan} studied the minimum degree condition for the existence of transversal $K_t$-factors in (hyper)graph collections, which is an asymptotical 
version of the rainbow Hajnal-Szemerédi theorem \cite{Hajnal-Szemerédi}. Montgomery, Müyesser and Pehova \cite{2022Montgomery} gave asymptotically tight transversal versions of Hajnal-Szemerédi theorem \cite{Hajnal-Szemerédi} and K\"{u}hn-Osthus theorem \cite{K-Othm} on factors, and a transversal generalisation
of the theorem of Koml{\'o}s, S{\'a}rk{\'o}zy and Szemer{\'e}di \cite{2001komlosTree} on spanning trees.
Gupta, Hamann, M{\"u}yesser, Parczyk and Sgueglia \cite{2023GuptaPowerH} gave a general approach to  transversal versions of several classical Dirac-type results. Cheng and Staden \cite{2023chengBlowup} established the transversal blow-up lemma, which is  an effective tool for transversal embedding problems. Chakraborti, Im, Kim and Liu \cite{2023chakrabortiBandwidth} extended the bandwidth theorem \cite{2009bandwidth} to graph collections. More relevant developments on this topic, we refer the reader to \cites{2009Ahamatch,2017AhaEKR,2019Ahamatch,2021ChengHyp,2021KupavskiiMatch,2022LuMatch,2023LuMatch,2023TangHyp}.

Li, Li and  Li \cite{2022liluyi} established a sufficient condition for the existence of a rainbow Hamilton path 
in $\mathcal{G} = \{G_1,\ldots,G_n\}$. However, the edges of a rainbow Hamilton path only come from $n-1$ 
graphs in $\mathcal{G}$. Hence, they proposed the following problem: whether a collection of $n-1$ graphs is sufficient to 
guarantee the existence of transversal  Hamilton paths. 
Our first result solves the  above problem and prove the following result.  

\begin{theorem}\label{th1}
  Let $\mathcal{G} = \{G_1,\ldots,G_{n-1}\}$ be a collection of graphs on a common vertex set $V$ of size $n$. If $\delta(\mathcal{G})\geq \frac{n-1}{2}$, then $\mathcal{G}$ has a transversal Hamilton path.
\end{theorem}

{The proof of Theorem \ref{th1} is inspired by Joos and Kim \cite{2021jooskim},  which proved the transversal version of Dirac's theorem by constructing auxiliary digraphs.  The difficulty of our proof is that we can only find a rainbow cycle of length at least $n-2$ when  $\delta(\mathcal{G})\geq \frac{n-1}{2}$. If the longest rainbow cycle inside  $\mathcal{G}$ has length $n-2$, then we need to analyze the relation between two vertices outside the cycle by using a series of rotations and more structural analysis.}

For graphs $G$ and $H$, denote by  $G\cup H$ the disjoint union of $G$ and $H$. We use $G\vee H$ to denote the graph obtained from $G\cup H$ by adding an edge between each vertex of $G$ and each vertex of $H$. A $k$-\textit{matching} in a graph is a set consisting of $k$ edges with no shared vertices.  
For any vertex subsets $X,\,Y \subseteq V(G)$, let $G[X]$ be the induced subgraph of $G$ on $X$. Let $G[X,Y]$ be the subgraph of $G$ with vertex set $X\cup Y$ and edge set $\{xy:x\in X,\,y\in Y\ \text{and}\ xy\in E(G)\}$. Denote by  $G-xy$ the graph obtained from $G$ by deleting the edge $xy$, where $xy\in E(G)$. 
Let $\mathcal{G} =\{G_1,G_2,\dots,G_m\} $ be collection of graphs  on  a common vertex set. Then define $\mathcal{G}[X]:=\{G_i[X]:i\in [m]\}$ and $\mathcal{G}[X,Y]:=\{G_i[X,Y]:i\in [m]\}$. If $E(G)=\emptyset$, then we simply write $G= \emptyset$. Similarly, if there is no edge in $G_i$ for all $i \in [m]$, then we write $\mathcal{G}= \emptyset$.

Let $0<\kappa<1$ be a positive constant. Given two graph collections  
$\mathcal{G}=\{G_1,\ldots,G_m\}$ and $\mathcal{H}=\{H_1,\ldots,H_m\}$ defined on the vertex sets $V_1$ and $V_2$ respectively, 
where $|V_1|=|V_2|=n$ and $(1-\kappa)n\leq m \leq (1+\kappa)n$, 
we say $\mathcal{G}$ is {\textit{$\kappa$-close to $\mathcal{H}$}} if by adding or deleting at most $\kappa n^3$ edges of $\mathcal{G}$ we can obtain a copy of $\mathcal{H}$. A graph collection $\mathcal{G}'=\{G_1',\ldots,G_m'\}$ is a {\textit{spanning collection}} of $\mathcal{G}$ if $G_i'$ is a spanning subgraph of $G_i$ for each $i\in [m]$. For any vertex set $V$ of size $n$, we say $A\cup B$ is an \textit{equitable partition} of $V$ if $|A|+|B|=n$ and $||A|-|B|| \leq 1$. 


\begin{definition}[$\mathcal{H}_s^t$, half-split graph collection] 
Given integers $s,t\geq 0$, let $\mathcal{H}_s^t$ be the graph 
collection on a common vertex set of size $n$ obtained by taking $s$ copies of $K_{\lceil\frac{n}{2}\rceil}\cup K_{\lfloor\frac{n}{2}\rfloor}$ and $t$ copies of $K_{\lceil\frac{n}{2}\rceil,\lfloor\frac{n}{2}\rfloor}$,  where they are defined on the same equitable partition.

We say that a graph collection $\mathcal{G}=\{G_1,\ldots,G_m\}$ on a common vertex set $V$ of size $n$ is half-split if there is a subset $A\subseteq V$ with $|A|=\lceil \frac{n}{2}\rceil$ such that $\mathcal{G}[A]=\emptyset$ and all graphs in $\mathcal{G}[A,V\backslash A]$ are complete bipartite graphs.


\end{definition}

Cheng and Staden \cite{2024cheng} characterized the stability result for transversal Hamilton paths or cycles  when the minimum degree of the graph collection is slightly lower than $\frac{n}{2}$.

\begin{theorem}[\cite{2024cheng}]\label{th2}
   For all $\kappa>0$, there exist $\mu>0$ and $n_0$ such that the following holds for all integers $n\geq n_0$. Let $\mathcal{G}=\{G_1,\ldots,G_m\}$ be a collection of graphs with the common vertex set $V$ of size $n$ and $\delta(\mathcal{G})\geq (\frac{1}{2}-\mu)n$. 
   \begin{enumerate}
       \item[{\rm (i)}] If $m=n$ and $\mathcal{G}$ contains no transversal Hamilton cycles, then either $\mathcal{G}$ is $\kappa$-close to $\mathcal{H}_s^t$ for some $s\in [n]$ with $t=n-s$ is odd, or $\mathcal{G}$ is $\kappa$-close to a half-split graph collection.
       \item[{\rm (ii)}] If $m=n-1$ and $\mathcal{G}$ contains no transversal Hamilton paths, then either $\mathcal{G}$ is $\kappa$-close to $\mathcal{H}_{n-1}^0$, or $\mathcal{G}$ is $\kappa$-close to a half-split graph collection.
   \end{enumerate}
   
\end{theorem}

As a continuation of the transversal version of Dirac's theorem \cite{2021jooskim} and the  stability result for transversal Hamilton cycles (i.e.,  Theorem \ref{th2} (i)), we characterize all graph collections with minimum degree at least {$\frac{n}{2}-1$} and without transversal Hamilton cycles. Note that in a graph collection  $\mathcal{G}$, the order of colors does not affect the overall structure of $\mathcal{G}$. Thus, when we say $G_i$ (for example, $G_1$) is a graph in $\mathcal{G}$, it is usually arbitrary.
\begin{theorem}\label{th3}
    For sufficiently large $n$, let $\mathcal{C}$ be a set of $n$ colors, and let $\mathcal{G}=\{G_i:i\in \mathcal{C}\}$ be a collection of graphs with common vertex set $V$ of size $n$ and {$\delta(\mathcal{G})\geq \frac{n}{2}-1$}. Assume that $\mathcal{G}$ contains no transversal Hamilton cycles.
    \begin{enumerate}
        \item[{\rm (i)}] If $n$ is odd, then one of the following holds:
        \begin{enumerate}
            \item[{\rm (a)}] $\mathcal{G}$ is the half-split graph collection,
            \item[{\rm (b)}] there exists a fixed vertex $u$ with  $d_{G_i}(u)=n-1$ and an equitable partition $A\cup B$ of $V\backslash \{u\}$ such that $G_i[V\backslash \{u\}]=G_i[A]\cup G_{i}[B]= 2K_{\frac{n-1}{2}}$ for all $i\in \mathcal{C}$.
        \end{enumerate}
        \item[{\rm (ii)}] If $n$ is even, then one of the following holds:
        \begin{enumerate}
            \item[{\rm (a)}] $\mathcal{G}$ is a spanning collection of  $\mathcal{H}_{n-t}^t$ for some odd integer $t\in [n]$,
            \item[{\rm (b)}] there exists a partition $A\cup B$ of $V$ with $|A|=\frac{n}{2}-1$ such that either $G_i[B]=\emptyset$ for all but at most one $i\in \mathcal{C}$, or $E(G_i[B])\subseteq \{uv\}$ for fixed $u,v\in B$ and all $i\in \mathcal{C}$,
            \item[{\rm (c)}] there exists an equitable partition $A\cup B$ of $V$ such that  $\mathcal{G}[A,B]$ contains no rainbow $2$-matching, that is,  one of the following holds:
            \begin{itemize}
    \item $G_i=G_i[A]\cup G_i[B]=2K_{\frac{n}{2}}$ for all but at most one $i\in \mathcal{C}$,
    \item there is a vertex  $u\in A$ such that  $E(G_i[A,B])=E(G_i[\{u\},B])$ for all $i\in \mathcal{C}$,
    \item there are vertices $u\in A$ and $v\in B$ such that $E(G_1[A,B])\subseteq \{uw:w\in B\}\cup \{wv:w\in A\}$ and $E(G_i[A,B])\subseteq \{uv\}$ for all $i\in \mathcal{C}\backslash \{1\}$ (see Figure~\ref{Hc-even}~(a)),
    \item there are vertices $u,u'\in A$ and $v,v'\in B$ such that $E(G_{1}[A,B])=\{uv,u'v'\}$, $E(G_2[A,B])=\{uv',u'v\}$ and $G_j=G_j[A]\cup G_j[B]=2K_{\frac{n}{2}}$ for all $j\in \mathcal{C}\backslash [2]$ (see Figure~\ref{Hc-even}~(b)). 
\end{itemize}
\end{enumerate}          
     \end{enumerate}       
\begin{figure}[ht!]
    \centering
    \includegraphics[width=0.8\linewidth]{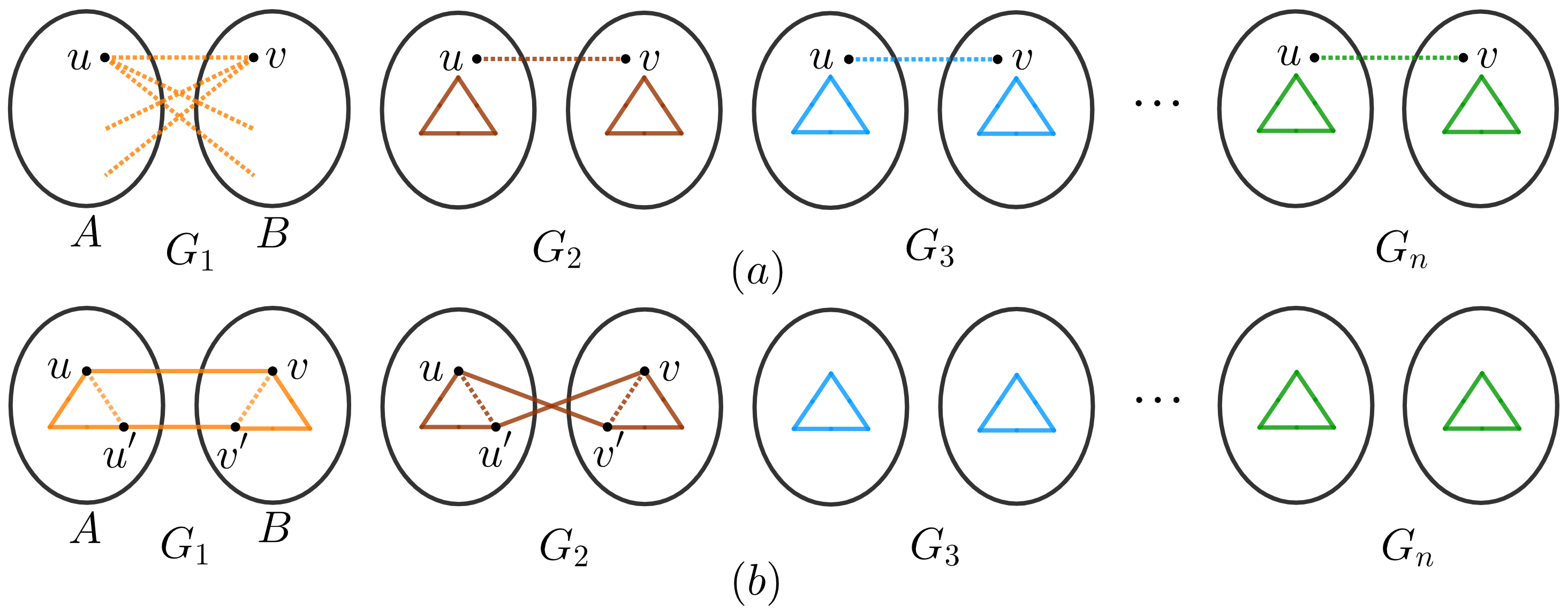}
    \caption{Extremal graph collections in Theorem \ref{th3}. Here the dotted line means that the edge may not exist, and the   triangle stands for a complete graph.}
    \label{Hc-even}
\end{figure}
    
\end{theorem}


\begin{remark}
    {\rm In fact, in Theorem \ref{th3} (ii) (b)-(c), based on the local structure of $\mathcal{G}$ and the minimum degree condition, the global structure of $\mathcal{G}$ can be deduced immediately. For example, $|A|=\frac{n}{2}-1$ and  $G_i[B]=\emptyset$ imply that $G_i[A,B]=K_{\frac{n}{2}-1,\frac{n}{2}+1}$ and $G_i[A]$ can be any graph.}
\end{remark}

For an $n$-vertex graph $G$ with $\delta(G)\geq \frac{n}{2} -1$, B{\"u}y{\"u}k{\c{c}}olak-G{\"o}z{\"u}pek-{\"O}zkan-Sibel-Shalom \cite{2019GraphOdd}  and Fu-Gao-Wang-Yang \cite{2024GraphEven} characterized the structure of $G$ without Hamilton cycles if $n\, (\geq 3)$ is odd and $n\, (\geq 22)$ is even, respectively. 
In fact, by applying Theorem \ref{th3} with $G_1=G_2=\dots=G_n=G$, we can deduce their results when $n$ is sufficiently large.

\begin{corollary}[\cites{2019GraphOdd,2024GraphEven}]
 For sufficiently large $n$, let $G$ be an $n$-vertex graph with vertex set $V$. Assume that $G$ contains no Hamilton cycles and  $\delta(G)\geq \frac{n}{2}-1$.
    \begin{enumerate}
        \item[{\rm (i)}] If $n$ is odd, then either $G=K_1\vee (K_{\frac{n-1}{2}}\cup K_{\frac{n-1}{2}})$ or $G$ is a spanning subgraph of  ${\frac{n+1}{2}K_1}\vee K_{\frac{n-1}{2}}$.
        \item[{\rm (ii)}] If $n$ is even, then $G$ is a spanning subgraph of $K_1\vee (K_{\frac{n}{2}-1}\cup K_{\frac{n}{2}})$ or $(K_2\cup (\frac{n}{2}-1)K_1)\vee K_{\frac{n}{2}-1}$.
            
            
            
            \end{enumerate}
            \end{corollary}

Recall that Theorem \ref{th1} gives a sufficient condition for the existence of transversal Hamilton paths in a collection of $n-1$ graphs. In contrast to the Hamilton cycle case, our next result characterizes all collections of $n-1$ graphs with minimum degree at least $\frac{n-3}{2}$ and containing no transversal Hamilton paths. 

\begin{theorem}\label{H-path}
     For sufficiently large $n$, let $\mathcal{C}$ be a set of $n-1$ colors, and let $\mathcal{G}=\{G_i:i\in \mathcal{C}\}$ be a collection of graphs with the common vertex set $V$ of size $n$. Assume that  {$\delta(\mathcal{G})\geq \frac{n-3}{2}$} and $\mathcal{G}$ contains no transversal Hamilton paths.
         \begin{enumerate}
             \item[{\rm (i)}] If $n$ is even, then one of the following holds:
        \begin{itemize}
            \item $\mathcal{G}$ is $\mathcal{H}_{n-1}^0$,
            \item there exists a partition $A\cup B$ of $V$ with $|A|=\frac{n}{2}+1$ such that $G_i[A]=\emptyset$ for all  $i\in \mathcal{C}$.
            
        \end{itemize}
 
         \item[{\rm (ii)}] If $n$ is odd, then one of the following holds:
         
     \begin{itemize}
            \item $\mathcal{G}$ is a spanning collection of $\mathcal{H}_{n-1}^0$,
            \item there exists a partition $A\cup B$ of $V$ with $|A|=\frac{n+3}{2}$ such that either $G_i[A]=\emptyset$ for all but at most one $i\in \mathcal{C}$, or $E(G_i[A])\subseteq \{uv\}$ for fixed vertices $u,v\in A$ and all $i\in \mathcal{C}$.
        \end{itemize}

     \end{enumerate}
\end{theorem}
Notice that given a collection $\mathcal{G}$ of $n-1$ graphs on a common vertex set of size $n$, the graph collection obtained from adding a complete
graph to $\mathcal{G}$ has a transversal Hamilton cycle if and only if $\mathcal{G}$ has a transversal Hamilton path. Thus, Theorem \ref{H-path} (i) is an immediate consequence of  Theorem \ref{th3} (ii). For Theorem \ref{H-path} (ii), it can be proved by using similar arguments as Theorem~\ref{th3}, so we omit its proof. Note that we just need to  find a transversal Hamilton path inside $\mathcal{G}$, that is, we don't need to connect its two endpoints, hence the proof of Theorem \ref{H-path}~(ii) is much simpler.

\section{Notation and Organization}
\subsection{Notation}
First of all, we state that all terminology and notation on graph theory not defined in this paper are the same as those used in the textbook \cite{graphtheory}.


 For a vertex subset $U,W\subseteq V(G)$ and a  vertex $v\in V(G)$, let $N_G(v, U) = N_G(v) \cap U$ and $d_G(v, U) = |N_G(v, U)|$. Let $N_G(U,W)=\cup_{u\in U}N_{G}(u,W)$. Denote by $S_r$ the star with $r$ vertices, where the vertex with degree $r-1$ is the center of $S_r$. 
Let $\mathcal{G}=\{G_i:i\in \mathcal{C}\}$ be a graph collection and $H$ be a rainbow subgraph inside $\mathcal{G}$. Denote by $col(H)$ the set of colors appearing in $H$.  Let $P_t=v_1v_2\ldots v_t$ be a rainbow path inside $\mathcal{G}$. For $1\leq i<j\leq t$, denote by $v_iPv_j:=v_iv_{i+1}\ldots v_j$ a rainbow subpath of $P$, and the color of each edge inherits its color on $P$. We say $P$ is \textit{maximal} in $\mathcal{G}$ if it is not a proper rainbow subpath of any other rainbow paths inside $\mathcal{G}$. Two rainbow paths $P$ and $Q$ are said to be \textit{disjoint} if $V(P)\cap V(Q)=\emptyset$ and $col(P)\cap col(Q)=\emptyset$. 


Let $D$ be a digraph. We denote its vertex set by $V(D)$ and its arc set by $A(D)$. For vertices $u$ and $v$ in $D$, the arc from $u$ to $v$ is denoted by $\overrightarrow{uv}$.
Given $v\in V(D)$, let $N_D^+ (v) $ and $N_D^-(v)$ be the out-neighborhood and in-neighborhood of $v$ in $D$, respectively. Denote by $d_D^+ (v)$ and $d_D^- (v) $ the out-degree and in-degree of $v,$ respectively.


For a positive integer $n$, we write $[n] := \{1, 2,\dots, n\}$ and $[a, b] := \{a, a + 1, \dots, b\}$ for two positive integers $a < b$. 
For any two constants $\alpha,\beta\in(0,1)$, we write $\alpha\ll\beta$ if there exists a
 function $f=f(\beta)$ such that the subsequent arguments hold for all $0<\alpha\leq f.$ 


\subsection{Organization}
In the remainder of this section, we give a sketch of the proof of Theorem \ref{th3}. In Section \ref{sec3}, we prove Theorem \ref{th1}, which establishes a minimum degree condition to guarantee the existence of transversal Hamilton paths. 
In Section \ref{sec4}, we give the proof of Theorem \ref{th3}. The proof is split into three steps: constructing short disjoint rainbow paths to cover bad vertices and bad colors; using the connecting tool (Claim \ref{conn}) to connect those rainbow paths;  applying transversal blow-up lemma to construct long rainbow paths to cover all left vertices and colors. In order to use the transversal blow-up lemma, we introduce Theorem \ref{lemma5.4} to balance the number of vertices in two parts, whose proof is presented in Section \ref{sec5}.  
Finally, some concluding remarks are given in Section \ref{sec6}.

\subsection{Sketch of the proof of Theorem \ref{th3}}
Assume $0<\frac{1}{n}\ll\delta\ll\eta\ll 1$. 
Let $\mathcal{G} =\{G_1,G_2,\dots,G_n\} $ be a collection of graphs with common vertex set $V$ of size $n$. Assume that  $\delta (\mathcal{G}) \geq \frac{n}{2}-1$ and $\mathcal{G}$ contains no transversal Hamilton cycles. Denote $\mathcal{C}_1:=\{i\in [n]:G_i$ is close to $K_{\lceil\frac{n}{2}\rceil}\cup K_{\lfloor\frac{n}{2}\rfloor} \}$, $\mathcal{C}_2:=\{i\in [n]:G_i$ is close to $K_{\lceil\frac{n}{2}\rceil,\lfloor\frac{n}{2}\rfloor}\}$ and $\mathcal{C}_{bad}:=[n]\backslash(\mathcal{C}_1\cup \mathcal{C}_2)$. By \cite{2024cheng}*{Lemma 5.3} (see also Lemma \ref{prop} in Section 4), we know that 
$|\mathcal{C}_1| +|\mathcal{C}_2| \geq (1-3\delta)n$. In view of Definition \ref{lemma3.1}, we can fix a characteristic partition $({A_i},{B_i},{C_i})$ for each 
graph $G_i$ with $i\in \mathcal{C}_1\cup \mathcal{C}_2$. By swapping the labels of ${A_i}$ and ${B_i}$, we have $|{A_1}\Delta {A_i}|,|{B_1}\Delta {B_i}|<\delta n$ for every $i\in (\mathcal{C}_1\cup \mathcal{C}_2)\backslash \{1\}$.

Based on the sizes of $|\mathcal{C}_1|$ and $|\mathcal{C}_2|$, we distinguish our proof into three cases: $|\mathcal{C}_1|< \eta n$, $|\mathcal{C}_2|< \eta n$,  $|\mathcal{C}_1|\geq \eta n$ and $|\mathcal{C}_2|\geq \eta n$. We will take the first case as an example to illustrate the general idea of our proof. 

Firstly, expand $A_1\cup B_1$ to an equitable partition $A\cup B$ of $V$. For a vertex $v\in A$, there are three bad cases: $v$ lies in almost all $B_i$ for $i\in \mathcal{C}_1\cup \mathcal{C}_2$; $v$ lies in almost all $C_i$ for $i\in \mathcal{C}_1\cup \mathcal{C}_2$; and $v$ lies in both $A_i$ and $B_i$ for many $i\in \mathcal{C}_1\cup \mathcal{C}_2$. Then move  vertices in the first case from $A$ to $B$, and do the same operation for such vertices in $B$. Denote the set of  vertices in the other two bad cases by $V_{bad}$. Now, the number of vertices in $A$ and $B$ may be unbalanced, but their difference is still small (no more than $9\delta n$).  In order to use the transversal blow-up lemma (see \cite{2023chengBlowup} or Claim \ref{lemma4.1} in Section 4) to find long rainbow paths inside $\mathcal{G}$, the following four steps are needed.

{\bf Step 1. Balance the number of vertices in $A$ and $B$.} If there exists a set of  disjoint rainbow paths in $\mathcal{G}[B]$ such that after deleting their vertices, the number of vertices in $A$ and $B$ are balanced, then we are  done. Otherwise, we will show that $\mathcal{G}$ must be a graph collection described in Theorem \ref{th3}.

{\bf Step 2. Deal with vertices in  $V_{bad}$.} By using colors in $\mathcal{C}_2$, we are to find a series of disjoint rainbow $P_3$ such that all centers of them are all vertices in $V_{bad}$ and the endpoints of them are unused vertices in $(A\cup B)\backslash V_{bad}$.

{\bf Step 3. Deal with colors in  $\mathcal{C}_{bad}$.}
Choose a maximal rainbow matching $M$ by using colors in $\mathcal{C}_{bad}$ and avoiding vertices used in Steps 1-2. It is routine to check that for each $j \in \mathcal{C}_{bad}\backslash col(M)$, the graph $G_{j}$ is close to $K_{\lfloor\frac{n}{2}\rfloor}\cup K_{\lceil\frac{n}{2}\rceil}$. 
 
{\bf Step 4. Deal with colors in $\mathcal{C}_{1}\cup (\mathcal{C}_{bad}\backslash col(M))$.}
We greedily select two rainbow paths with colors $\mathcal{C}_{1}\cup (\mathcal{C}_{bad}\backslash col(M))$ and avoiding the used vertices in the above three steps. The lengths of those two rainbow paths are determined by the parity of  $|\mathcal{C}_{1}\cup (\mathcal{C}_{bad}\backslash col(M))|$. 

Based on the minimum degree condition and the characteristic partition of extremal graphs, one may use colors in $\mathcal{C}_2$ to connect all rainbow paths obtained in the above four steps into a single short rainbow path, say $P$, such that $V_{bad}\subseteq V(P)$,  $\mathcal{C}_1\cup \mathcal{C}_{bad}\subseteq col(P)$, $|A\backslash V(P)|$ and $|B\backslash V(P)|$ are almost equal. By applying  
the transversal blow-up lemma and some structural analysis, 
the exact structure of graphs in $\mathcal{G}$ can be deduced.  \allowdisplaybreaks

\section{Proof of Theorem \ref{th1}}\label{sec3}
In this section, we prove Theorem \ref{th1}. The following lemma allows us to find a rainbow path of length at least $n-3$ in $\mathcal{G}$.

\begin{lemma}[\cite{2022liluyi}]\label{thm0}
     Suppose $\mathcal{G} = \{G_1,\ldots,G_n\}$ is a collection of graphs with common vertex set $V$ of size $n$, and $\delta(\mathcal{G})\geq \frac{n-1}{2}$. Then one of the following holds:
     \begin{itemize}
         \item $\mathcal{G}$ has a rainbow cycle of length at least $n-1$,
         \item $n$ is odd and $\mathcal{G}$ consists of $n$ copies of $K_1\vee 2K_{\frac{n-1}{2}}$.
     \end{itemize}
\end{lemma}
\begin{proof}[Proof of Theorem \ref{th1}]
We prove it by contradiction. Suppose that $\mathcal{G}$ does not contain transversal Hamilton paths. Let $V=\{x_{1},\ldots,x_{n}\}$. Firstly, we give the following claim. 
\begin{claim}\label{cliam1}
    $\mathcal{G}$ has a rainbow cycle of length at least $n-2$.
\end{claim}

\begin{proof}[Proof of Claim \ref{cliam1}]    
Suppose that $\mathcal{G}$ contains no rainbow cycle of length at least $n-2$. By Lemma \ref{thm0}, we know $\{G_1,\ldots,G_{n-1},K_n\}$ contains a rainbow cycle of length at least $n-1$. Thus, $\mathcal{G}$ contains a rainbow path of length at least $n-3$. Let $P=x_{1}x_{2}\ldots x_{n-2}$ be such a rainbow path, where $x_ix_{i+1}\in E(G_i)$ for $i\in [n-3]$. It is obvious that $x_1x_{n-2} \notin E(G_{n-1}) \cup E(G_{n-2})$. 

We claim that $|N_{G_{n-2}}(x_1) \cap \{x_{n-1},x_n\}|+|N_{G_{n-1}}(x_{n-2}) \cap \{x_{n-1},x_n\}|\leq 2$. Otherwise, there exists $x_j\in N_{G_{n-2}}(x_1)\cap N_{G_{n-1}}(x_{n-2})$ for some $j\in \{n-1,n\}$. It follows that $x_jx_1\ldots x_{n-2}x_{j}$ forms a rainbow cycle of length $n-1$ in $\mathcal{G}$, a contradiction. 
Let   
$$A=\{j\in[n-4]:x_{1}x_{j+1}\in E(G_{n-2})\}, \quad B=\{j\in[2, n-3]:x_{j}x_{n-2}\in E(G_{n-1})\}.$$ 
If $A\cap B \neq \emptyset$, then choose  $j_1\in A\cap B$ and  $x_{1}x_{j_1+1}\ldots x_{n-2}x_{j_1}x_{j_1-1}\ldots x_{1}$ is a rainbow cycle of length $n-2$ in $\mathcal{G}$, a contradiction.  Thus, $A\cap B = \emptyset$. Notice that $A\cup B\subseteq [n-3]$. Hence $|A|+|B|\leq n-3$. Recall that $x_1x_{n-2} \notin E(G_{n-1}) \cup E(G_{n-2})$. Then, $|A|\geq \frac{n-1}{2}-|N_{G_{n-2}}(x_1)\cap \{x_{n-1},x_n\}|$ and $|B|\geq \frac{n-1}{2}-|N_{G_{n-1}}(x_{n-2})\cap \{x_{n-1},x_n\}|$. Therefore, $|N_{G_{n-2}}(x_1) \cap \{x_{n-1},x_n\}|+|N_{G_{n-1}}(x_{n-2}) \cap \{x_{n-1},x_n\}|=2$ and  $|A|+|B|=n-3$. 
By symmetry, one of the following holds:
\begin{itemize}
    \item[{\rm (\textbf{A1})}] $N_{G_{n-2}}(x_1)\cap \{x_{n-1},x_n\}=\{x_{n-1}\}$ and $N_{G_{n-1}}(x_{n-2})\cap \{x_{n-1},x_n\}=\{x_{n}\}$,
    \item[{\rm (\textbf{A2})}] $N_{G_{n-2}}(x_1)\cap \{x_{n-1},x_n\}=\{x_{n-1}\}$ and $N_{G_{n-1}}(x_{n-2})\cap \{x_{n-1},x_n\}=\{x_{n-1}\}$,
    \item[{\rm (\textbf{A3})}] $N_{G_{n-2}}(x_1)\cap \{x_{n-1},x_n\}=\{x_{n-1},x_n\}$ and $N_{G_{n-1}}(x_{n-2})\cap \{x_{n-1},x_n\}=\emptyset$.
\end{itemize}

If (\textbf{A1}) holds, then $\mathcal{G}$ contains a transversal Hamilton path $x_{n-1}x_{1}\ldots x_{n-2}x_{n}$, a contradiction. If (\textbf{A2}) holds, then $\mathcal{G}$ contains a rainbow $(n-1)$-cycle $x_{n-1}x_{1}\ldots x_{n-2}x_{n-1}$, which is also a contradiction. Hence (\textbf{A3}) holds. 

It is easy to see that $x_{n-3}x_{n} \notin E(G_{n-3})\cup E(G_{n-1})$. Otherwise, $x_nx_{1}\ldots x_{n-3}x_n$ is a rainbow cycle of length $n-2$ in $\mathcal{G}$, a contradiction. For convenience, denote $x^1:=x_n$, $x^i:=x_{i-1}$ for $i\in [2,n-2]$, $x^{n-1}:=x_{n-1}$ and $x^{n}:=x_{n-2}$. Then  $P':=x^1x^2\ldots x^{n-2}$ is a rainbow path inside $\mathcal{G}$ with $n-3,n-1\notin col(P')$. Let
$$A'=\{j\in[n-4]:x^1x^{j+1}\in E(G_{n-1})\}, \quad B'=\{j\in[2, n-3]:x^{j}x^{n-2}\in E(G_{n-3})\}.$$
Notice that $x^1x^{n-1}\not\in E(G_{n-1})$, otherwise $x^{n-1}x^1P'x^{n-2}x^n$ forms a transversal Hamilton path, a contradiction. It follows from (\textbf{A3}) that $x^1x^n \notin E(G_{n-1})$ and $x^2x^{n-1}\in E(G_{n-2})$. Then $x^{n-1}x^{n-2} \notin E(G_{n-3})$. Together with $x_{n-3}x_{n} \notin E(G_{n-3})\cup E(G_{n-1})$, we have  $|A'|+|B'|\geqslant\frac{n-1}{2}+\frac{n-1}{2}-1=n-2$. Since $A'\cup B'\subseteq [n-3]$, one has $A'\cap B'\ne \emptyset$ and let  $j_2 \in A'\cap B'$. Hence, $x^1x^{j_2+1}P'x^{n-2}x^{j_2}P'x^1$ is a rainbow cycle of length $n-2$ in $\mathcal{G}$, a contradiction. 

This completes the proof of Claim \ref{cliam1}.
\end{proof}

Based on Claim \ref{cliam1}, we distinguish our proof into  the following two cases.

{\bf Case 1.} $\mathcal{G}$ has a rainbow cycle of length $n-1$.

Let $C=x_{1}x_{2}\ldots x_{n-1}x_{1}$ be a rainbow cycle inside $\mathcal{G}$ with $x_ix_{i+1}\in E(G_i)$ for $i\in [n-1]$, where we identify $x_n$ with $x_1$. Let $y$ be the vertex in $V\backslash V(C)$. Since $\mathcal{G}$ has no transversal Hamilton path, one has 
$yx_i,yx_{i+1}\not\in E(G_i)$ for $i \in [n-1]$.


We construct the following auxiliary digraph $D$ with  vertex set $V$ and arc set  
$$A(D)=  \bigcup_{i\in[n-1]}\{\overrightarrow{x_{i}z }:x_{i}z \in E(G_{i})\ \text{and}\ z\neq x_{i+1}\}. $$
By the definition of $D$, we have $d_{D}^{+}(y)=0$. Furthermore, $d_{D}^{-}(y)=0$. Otherwise, there exists $x_i \in N_{D}^{-}(y)$ for some $i\in[n-1]$. Then, $yx_i\ldots x_1x_{n-1}\ldots x_{i+1}$ is a transversal Hamilton path in $\mathcal{G}$, a contradiction. It follows from $\delta(\mathcal{G}) \geq\frac{n-1}{2}$ that $d_{D}^{+}(x_i)\geq \frac{n-1}{2}-1$ for all $i\in [n-1]$. Therefore, 
\begin{equation}\label{eq:0}
    |A(D)|\geq(n-1)(\frac{n-1}{2}-1)=(n-1)\frac{n-3}{2}.
\end{equation}

Without loss of generality, assume that $d_D^-(x_1)=\max\{d_D^-(x_i):i\in [n-1]\}$. Then we claim that $d_D^-(x_1)\geq\frac{n-3}{2}$. Otherwise, $d_D^-(x_1)\leq\frac{n-5}{2}$. This implies that 
$|A(D)|=\sum_{i\in [n-1]}d_D^-(x_i)\leq (n-1)\frac{n-5}{2}<(n-1)\frac{n-3}{2}$, which contradicts \eqref{eq:0}.
Let
$$I_1^1=\{i\in [2,n-2]:x_i\in N_{D}^{-}(x_{1})\}, \quad I_2^1=\{i\in [2,n-2]:x_{i+1}y\in E(G_1)\}.
$$
Recall that $x_1y,x_2y\notin E(G_1)$. 
Then $|I_1^1|+|I_2^1|\geq\frac{n-3}{2}+\frac{n-1}{2}=n-2$. But $I_1^1\cup I_2^1\subseteq [2,n-2]$. Hence $I_1^1\cap I_2^1\neq\emptyset$. Choose  $i_1\in I_1^1\cap I_2^1$. Therefore, $x_2\ldots x_{i_1}x_1x_{n-1}\ldots x_{i_1+1}y$ is a transversal Hamilton path in $\mathcal{G}$, a contradiction.


{\bf Case 2.} $\mathcal{G}$ has no  rainbow cycle of length $n-1$. 

By Claim \ref{cliam1}, there exists a rainbow cycle of length $n-2$ in $\mathcal{G}$, say $C=x_{1}x_{2}\ldots x_{n-2}x_{1}$ with $x_ix_{i+1}\in E(G_i)$ for $i \in [n-2]$, where we identify $x_{n-1}$ with $x_1$. Let $\{y,y'\}:=V\backslash V(C)$. Now, we utilize the following auxiliary digraph $D$ with  vertex set $V$ and arc set  
$$
A(D)=  \bigcup_{i\in[n-2]}\{ \overrightarrow{x_{i}z} :x_{i}z \in E(G_{i})\ \text{and}\ z\neq x_{i+1}\}. 
$$ 
Clearly, $d_{D}^{+}(y)=d_{D}^{+}(y')=0$ and $d_{D}^{+}(x_i)\geq \frac{n-1}{2}-1$ for each $i\in [n-3]$. Then, $|A(D)|\geq (n-2)\frac{n-3}{2}$.  Without loss of generality, assume that $d_D^-(x_1)=\max\{d_D^-(x_i):i\in [n-2]\}$. 
In what follows, we are going to prove that $d_{D}^-(x_{1}) \leq \frac{n-5}{2}$. Suppose that $d_{D}^-(x_{1}) \geq \frac{n-3}{2}$. We distinguish the proof into two subcases according to  $yy'$ is an edge in $G_{n-1}$ or not.

{\bf Subcase 2.1.} $yy' \in E(G_{n-1})$. 

Let 
$$
I_1^2=\{i\in [2,n-3]:x_{i+1}y\in E(G_{1})\}, \quad I_2^2=\{i\in [2,n-3]:x_{i} \in N_{D}^{-}(x_{1})\}.$$
Notice that $x_1y,x_2y\not\in E(G_1)$, otherwise either $y'yx_1x_{n-2}\ldots x_2$ or $y'yx_2\ldots x_{n-2}x_1$ is a transversal Hamilton path in $\mathcal{G}$, a contradiction. Recall that $d_D^-(x_1)\geq \frac{n-3}{2}$. Hence, $|I_1^2|+|I_2^2|\geq\frac{n-3}{2}+\frac{n-3}{2}=n-3$. 
Together with $I_1^2\cap I_2^2\subseteq[2,n-3]$, there exists an $i_2$ in $I_1^2\cap I_2^2$. Hence, $y'yx_{i_2+1}\ldots x_{n-2}x_1x_{i_2}\ldots x_2$ is a transversal Hamilton path in $\mathcal{G}$ (see Figure \ref{HP-2} (a)), which is a contradiction.
\begin{figure}[ht!]
    \centering
    \includegraphics[width=1\linewidth]{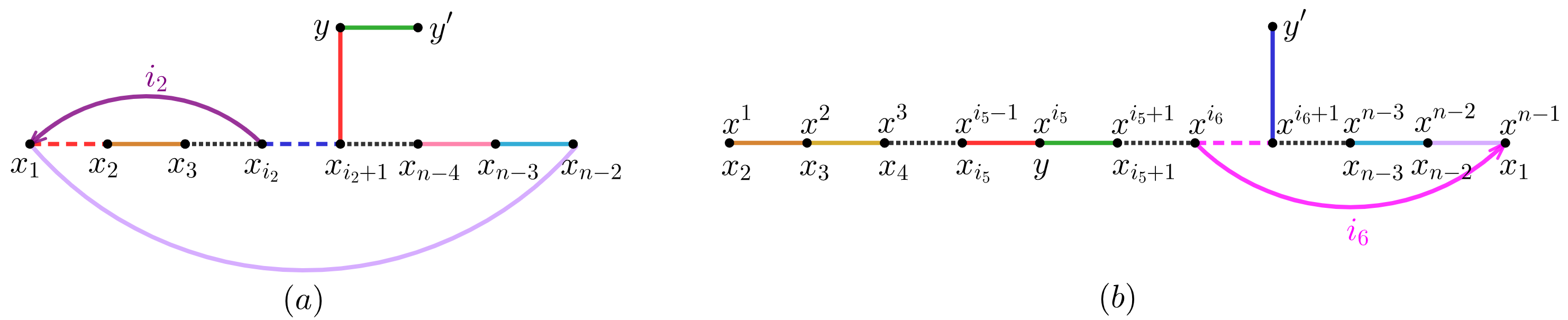}
    \caption{}
    \label{HP-2}
\end{figure}

{\bf Subcase 2.2.} $yy' \notin E(G_{n-1})$. 

Let
$$
I_1^3=\{i\in [n-2]:x_i\in N_{D}^{-}(y)\},\quad I_2^3=\{i\in [n-2]:x_{i+1}y\in E(G_{n-1})\}.
$$
It is routine to check that  $I_1^3\cap I_2^3=\emptyset$. Otherwise, let $i_3\in I_1^3\cap I_2^3$. Then $x_1\ldots x_{i_3}yx_{i_3+1}\ldots x_{n-2}x_1$ is a rainbow cycle of length $n-1$ in $\mathcal{G}$, which contradicts the assumption of Case 2. Note that $I_1^3\cup I_2^3\subseteq [n-2]$ and $|I_2^3| \geq \frac{n-1}{2}$. Then $|I_1^3| \leq \frac{n-3}{2}$, i.e., $d_{D}^{-}(y) \leq \frac{n-3}{2}$. Similarly, we have $d_{D}^{-}(y') \leq \frac{n-3}{2}$. Now, we give the following claim.


\begin{claim}\label{claim2}
    $yy' \notin E(G_{1})$.
\end{claim}
\begin{proof}[Proof of Claim \ref{claim2}]
Suppose to the contrary that $yy' \in E(G_{1})$. In view of Subcase 2.1, we can assume that $\{G_i[V\backslash \{y,y'\}]:i\in [2,n-1]\}$ contains no transversal Hamilton cycles. Hence $x_1x_2\notin E(G_{n-1})$.  Let
$$
I_1^4=\{i\in [2,n-3]:x_{i}\in N_{D}^{-}(x_{1})\} ,\quad I_2^4=\{i\in [2,n-3]:x_{i+1}x_2\in E(G_{n-1})\}.
$$
Then $|I_1^4|+|I_2^4|\geq \frac{n-3}{2}+\frac{n-1}{2}=n-2$. Since $I_1^4\cup I_2^4\subseteq [2,n-2]$, one has $I_1^4\cap I_2^4 \neq \emptyset$. Let $i_4\in I_1^4\cap I_2^4$. Then $x_1x_{i_4}x_{i_4-1}\ldots x_2x_{i_4+1}\ldots x_{n-2}x_1$ is a transversal Hamilton  cycle inside $\{G_i[V\backslash \{y,y'\}]:i\in [2,n-1]\}$, a contradiction.
\end{proof}

Define
$$
I_1^5=\{i\in [n-2]:x_{i}y\in E(G_1)\} ,\quad I_2^5=\{i\in [n-2]:x_{i+1}y\in E(G_{n-1})\}.
$$
By Claim \ref{claim2} and the fact that $yy'\not\in E(G_{n-1})$, we have $|I_1^5|+|I_2^5|\geq \frac{n-1}{2}+\frac{n-1}{2}=n-1$. Then  $I_1^5\cap I_2^5 \neq \emptyset$ and choose $i_5\in I_1^5\cap I_2^5$. Hence, $P'':=x_2\ldots x_{i_5}yx_{i_5+1}\ldots x_{n-2}x_1$ is a rainbow path inside $\mathcal{G}$. For convenience, we rewrite $P''= x^1\ldots x^{n-1}$ such that $x^1:=x_2$ and $x^{n-1}:=x_1$. 
Let 
$$
I_1^6=\{i\in [n-3]\backslash \{i_5-1\}:y'x^{i+1}\in E(G_{i_5})\} ,\quad I_2^6=\{i\in [n-3]\backslash \{i_5\}:x^{i}\in N_{D}^{-}(x^{n-1})\}.
$$
Clearly, $yy' \notin E(G_{i_5})$. Otherwise, $y'yx_{i_5+1}\ldots x_{n-2}x_1x_2\ldots x_{i_5}$ is a transversal  Hamilton path in $\mathcal{G}$, a contradiction. Furthermore, $y'x^{n-1},\,y'x^1\notin E(G_{i_5})$. Notice that $y=x^{i_5}$ and $d_D^+(y)=d_D^+(y')=0$. Hence  $|I_1^6|+|I_2^6|\geq  \frac{n-1}{2}+\frac{n-3}{2}=n-2$. Therefore, $I_1^6\cap I_2^6 \neq \emptyset$ and choose $i_6\in I_1^6\cap I_2^6$. Then $i_6\notin \{i_5-1,i_5\}$ and so $x^{i_6},x^{i_6+1}\neq y$. By the definition of $I_2^6$ and $D$, we have $x^{i_6}x^{i_6+1},x^{i_6}x^{n-1}\in E(G_{i_6+1})$ if $i_6\leq i_5-1$, and  $x^{i_6}x^{i_6+1},x^{i_6}x^{n-1}\in E(G_{i_6})$ if $i_6\geq i_5+1$. Hence, $P''-\{x^{i_6}x^{i_6+1}\}+\{y'x^{i_6+1}, x^{i_6}x^{n-1}\}$ forms a transversal Hamilton path in $\mathcal{G}$ (see Figure \ref{HP-2} (b)), which is also a contradiction. 

Therefore, we obtain  $d_{D}^{-}(x_{i})\leq d_{D}^{-}(x_1) \leq \frac{n-5}{2}$ for all $i \in [n-2].$ Recall that $d_{D}^{-}(y),\,d_{D}^{-}(y') \leq \frac{n-3}{2}$. Thus $|A(D)|\leq (n-2)\frac{n-5}{2}+(n-3)<(n-2)\frac{n-3}{2}$, a contradiction. This implies that there exists a transversal Hamilton path in $\mathcal{G}$.

This completes the proof of Theorem \ref{th1}.
\end{proof}

\section{Proof of Theorem \ref{th3}}\label{sec4}
\subsection{Preliminaries}
In this subsection, we give some preliminaries. Firstly, we define the extremality for a single graph.
\begin{definition}[\cite{2024cheng}]\label{lemma3.1}
Let 
$\epsilon>0$ and  $G$ be a graph on vertex set $V$ of size $n$. We call $G$ is $\epsilon$-extremal if there exists a partition ${A}\cup {B}\cup {C}$ of $V$ such that $|{A}|=|{B}|=(\frac{1}{2}-\epsilon)n$, and one of the following holds: 
\begin{enumerate}
    \item [{\rm (i)}]
$d_{G}(a,{A})\geq (\frac{1}{2}-2\epsilon)n$ for all $a\in {A}$ and $d_{G}(b,{B})\geq (\frac{1}{2}-2\epsilon)n$ for all $b\in {B}$;
\item [{\rm (ii)}] 
$d_{G}(a,{B})\geq (\frac{1}{2}-2\epsilon)n$ for all $a\in {A}$ and $d_{G}(b,{A})\geq (\frac{1}{2}-2\epsilon)n$ for all $b\in {B}$.
\end{enumerate}
For convenience, we say $G$ is $(\epsilon,K_{\lceil\frac{n}{2}\rceil}\cup K_{\lfloor\frac{n}{2}\rfloor})$-extremal if (i) holds; and $G$ is $(\epsilon,K_{\lceil\frac{n}{2}\rceil,\lfloor\frac{n}{2}\rfloor})$-extremal if (ii) holds. We always call $({A},{B},{C})$ a characteristic partition of $G$.
\end{definition}


Let $\mathcal{G}=\{G_1,\ldots,G_n\}$ be a graph collection on a common vertex set $V$ of size $n$ and $\delta(\mathcal{G})\geq \frac{n}{2}-1$. Let $0<\alpha,\beta,\delta<1$. We say $\mathcal{G}$ is $(\alpha,\beta)$-\textit{strongly stable} if $G_i$ is $\alpha$-extremal for at most $(1-\beta) n$ colors $i\in [n]$.  Assume that $G_1,\ldots,G_m$ are all $\alpha$-extremal graphs inside $\mathcal{G}$. By Definition \ref{lemma3.1}, we can fix  a characteristic partition $({A_i},{B_i},{C_i})$ of $G_i$ for each $i\in [m]$. We say $G_i$ and $G_j$ are $\delta$-\textit{crossing} if $G_i$ and $G_j$ are $\alpha$-extremal and  $|{A_i}\Delta {A_j}|\geq \delta n$,  $|{A_i}\Delta {B_j}|\geq \delta n$. For $i\in [m]$, define $I_i$ to be the set of $j\in [m]$ such that $G_i$ and $G_j$ are $\delta$-{crossing}. 
 Under the above characteristic  partitions, we say $\mathcal{G}$ is $(\alpha,\delta)$-\textit{weakly stable} if $\sum_{i\in [m]}|I_i|\geq 2\delta m^2$.

Cheng and Staden \cite{2024cheng}*{Lemma 5.3} proved the following result. 
\begin{lemma}[\cite{2024cheng}]\label{prop}
   Assume $0<\frac{1}{n}\ll\mu\ll\alpha\ll\beta,\epsilon\ll\delta\ll1$. Let $\mathcal{G}=\{G_1,\ldots,G_n\}$ be a graph collection on a common vertex set $V$ of size $n$ and $\delta(\mathcal{G})\geq (\frac{1}{2}-\mu)n$. If $\mathcal{G}$ is either $(\alpha,\beta)$-{strongly stable} or $(\epsilon,\delta)$-{weakly stable}, then $\mathcal{G}$ contains a transversal Hamilton cycle.
\end{lemma}

The following theorem characterizes the structure of almost balanced graph collections (i.e., almost all graphs in the graph collection have an almost balanced common vertex  bipartition such that their subgraphs induced by the bipartition are close to complete bipartite graphs) that do not contain transversal Hamilton cycles. For the sake of readability, we postpone its proof in Section \ref{sec5}. 
{{\begin{theorem}\label{lemma5.4}
    Let $0<\frac{1}{n}\ll \delta\ll 1$ and $0\leq\gamma \leq 3{\delta}$. Assume $\mathcal{C}$ is a set of $n$ colors and $\mathcal{G}=\{G_i:i\in \mathcal{C}\}$ is  collection of graphs on a common vertex set $V$ of size $n$ with $\delta(\mathcal{G})\geq \frac{n}{2}-1$. Let $\mathcal{C}'\cup \mathcal{C}''$ be a partition of $\mathcal{C}$ with $|\mathcal{C}''|\leq {\delta} n$, and let $A\cup B$ be a partition of $V$ with $|A|=\lceil\frac{n}{2}-1\rceil+\gamma n$.  Suppose that $G_i[B]=\emptyset$ for all $i\in \mathcal{C}'$ and $\mathcal{G}$ does not contain a transversal Hamilton cycle. Then $\mathcal{G}$ must be one of the graph collections described in Theorem \ref{th3}.
\end{theorem}}





\subsection{Proof of Theorem \ref{th3}}
In this subsection, we give the proof of Theorem \ref{th3}.
\begin{proof}[Proof of Theorem \ref{th3}] 
Suppose that $\mathcal{G}$ does not contain transversal Hamilton cycles. We are to prove that $\mathcal{G}$ must be a graph collection described in Theorem \ref{th3}. Let $0<\frac{1}{n}\ll\alpha\ll\beta,\epsilon\ll\delta\ll\eta\ll1$.  
By Lemma \ref{prop},  we know $\mathcal{G}$ is not $(\alpha,\beta)$-{strongly stable}. Thus, there are at least $(1-\beta)n$ graphs in $\mathcal{G}$ such that each of them is $\alpha$-extremal, and hence $\epsilon$-extremal. Therefore, there exists a constant $\beta'\,(<\beta)$ such that the number of $\epsilon$-extremal graphs in $\mathcal{G}$ is exactly $(1-\beta')n$. 
Without loss of generality, assume that for each $i\in [(1-\beta')n]$, $G_i$ is $\epsilon$-extremal. By Definition \ref{lemma3.1}, we can fix a characteristic partition $(\Tilde{A_i},\Tilde{B_i},\Tilde{C_i})$ for each $G_i$ with $i\in [(1-\beta')n]$. 

Under the above  characteristic partition, applying  Lemma \ref{prop} again yields that $\mathcal{G}$ is not 
$(\epsilon,\delta)$-{weakly stable}. Therefore, $\sum_{i\in [(1-\beta')n]}|I_i|< 2\delta (1-\beta')^2n^2$. 
Without loss of generality, assume that $|I_1|=\min\{|I_i|:i\in [(1-\beta')n]\}$. Then $|I_1|\leq 2\delta(1-\beta')n$. For every $i\in[(1-\beta')n]\backslash I_1$, based on the definition of $I_1$, we have either $|\Tilde{A_1}\Delta \Tilde{A_i}|<\delta n$ or $|\Tilde{A_1}\Delta \Tilde{B_i}|<\delta n$. By  swapping the labels of $\Tilde{A_i}$ and $\Tilde{B_i}$, we have $|\Tilde{A_1}\Delta \Tilde{A_i}|,|\Tilde{B_1}\Delta \Tilde{B_i}|<\delta n$ for every  $i\in[(1-\beta')n]\backslash I_1$. Denote $[m]:=[(1-\beta')n]\backslash I_1$ and $\mathcal{C}_{bad}:=[m+1,n]$. Then $m\geq (1-\beta'-2\delta)n\geq (1-3\delta)n$. By adding colors to $\mathcal{C}_{bad}$ if necessary we may assume $m=(1-3\delta)n$. Recall that $\alpha\ll \epsilon$. Hence for each $i\in [2,m]$, there exists a new partition $(A_i,B_i,C_i)$ of $G_i$ such that \allowdisplaybreaks
\begin{itemize}
    \item[{\rm \bf (B1)}] if $G_i$ is $(\epsilon,K_{\lceil\frac{n}{2}\rceil}\cup K_{\lfloor\frac{n}{2}\rfloor})$-extremal, then for 
$Y\in \{A,B\}$ we have  $\Tilde{Y_i}\subseteq Y_i$ and  $d_{G_i}(v,Y_i)\geq (\frac{1}{2}-3\sqrt{\delta})n$ for each vertex $v\in Y_i$,
\item[{\rm \bf (B2)}] if $G_i$ is $(\epsilon,K_{\lceil\frac{n}{2}\rceil,\lfloor\frac{n}{2}\rfloor})$-extremal, then for 
$Y\in \{A,B\}$ we have  $\Tilde{Y_i}\subseteq Y_i$ and  $d_{G_i}(v,Z_i)\geq (\frac{1}{2}-3\sqrt{\delta})n$ for each vertex $v\in Y_i$,
\item[{\rm \bf (B3)}] subject to {\bf (B1)-(B2)}, each vertex in $V$ lies in as many as possible $A_i\cup B_i$ for $i\in [2,m]$. 
\end{itemize}


Denote $(A_1,B_1,C_1):=(\Tilde{A}_1,\Tilde{B}_1,\Tilde{C}_1)$.  Expand $A_1\cup B_1$ to be an equitable partition $A\cup B$ of $V$. Let $\mathcal{C}_1\cup \mathcal{C}_2$ be a partition  of $[m]$ with 
\begin{align*}
    &\mathcal{C}_1=\{i\in[m]:\  G_i ~\textrm{is}~(\epsilon,K_{\lceil\frac{n}{2}\rceil}\cup K_{\lfloor\frac{n}{2}\rfloor})\textrm{-extremal}\},\\
    &\mathcal{C}_2=\{i\in[m]:\ G_i ~\textrm{is}~(\epsilon,K_{\lceil\frac{n}{2}\rceil,\lfloor\frac{n}{2}\rfloor})\textrm{-extremal}\}.
\end{align*}
Hence for $i\in[2,m]$ and $Y\in \{A,B\}$, we have $|Y_1\Delta Y_{i}|\leq |\Tilde{Y_1}\Delta \Tilde{Y_i}|+|\Tilde{C}_i|\leq  \delta n+2{\epsilon}n<2\delta n$.

In the following of our proof, we shall emphasize that if we say a set of vertices $V'\subseteq B$ is moved from $B$ to $A$, then the partition of $V$ becomes to $(A\cup V')\cup (B\backslash V')$. In order not to introduce too much  notation, we still denote the resulted partition by $A\cup B$. Moreover, unless otherwise specified, we assume $Y\in\{A,B\}$ and $\{Y,Z\}=\{A,B\}$.  

Now, we define the set of bad vertices. 
Let 
$
\hat{\mathcal{C}}:=\cup_{k\in [2]}\psi(\mathcal{C}_k),
$ 
where  $\psi(\mathcal{C}_k)=\mathcal{C}_k$ if $|\mathcal{C}_k|\geq \eta n$ and $\psi(\mathcal{C}_k)=\emptyset$ otherwise. Clearly, $|\hat{\mathcal{C}}|\geq (1-3\delta-\eta)n$. Define 
$$
X=\{x\in V: x\not\in A_i\cup B_i~\textrm{for at least}~ 3\sqrt{\delta}|\hat{\mathcal{C}}|~\textrm{colors}~i\in \hat{\mathcal{C}}\},
$$ 
$$
X_Y=\{x\in Y\backslash X:x\not\in Y_{i}~\textrm{for at least}~ 10\sqrt{\delta}|\hat{\mathcal{C}}|~\textrm{colors}~i\in \hat{\mathcal{C}}\}.
$$
Recall that for all 
$i\in \hat{\mathcal{C}}$, we have  $|C_{i}|\leq 2{\epsilon}n$ and $|Y_1\Delta Y_{i}|<2\delta n$. Hence $ 3\sqrt{\delta}|\hat{\mathcal{C}}||X|\leq 2{\epsilon}n|\hat{\mathcal{C}}|$ and $10\sqrt{\delta}|\hat{\mathcal{C}}||X_Y|\leq \sum_{i\in \hat{\mathcal{C}}}|Y\backslash Y_i|\leq (2\delta+2{\epsilon}) n|\hat{\mathcal{C}}|$. It follows that $|X|\leq \sqrt{\epsilon}n$ and $|X_Y|< \frac{1}{4}\sqrt{\delta} n$. 

For $k\in [2]$, define 
\begin{equation*}
    X_Y^k=
    \left\{
    \begin{array}{ll}
        \{x\in X_Y:x\in Y_{i}~\textrm{for at least}~ 3\sqrt{\delta}n~\textrm{colors}~i\in {\mathcal{C}_k}\}, &\textrm{if $|\mathcal{C}_k|\geq \eta n$,}\\[5pt]
        \emptyset, &\textrm{otherwise.}
    \end{array}
    \right.
\end{equation*}
Recall that $\{Y,Z\}= \{A,B\}$. For a vertex $x\in X_Y\backslash (\cup_{k\in [2]}X_Y^k)$, we know $x\in Y_{i}$  for at most $6\sqrt{\delta}n$ colors $i\in \hat{\mathcal{C}}$ and $x\in A_i\cup B_i$ for at least $(1-3\sqrt{\delta})|\hat{\mathcal{C}}|$ colors $i\in \hat{\mathcal{C}}$. Thus, $x\in Z_{i}$ for at least $(1-3\sqrt{\delta})|\hat{\mathcal{C}}|-6\sqrt{\delta}n\geq (1-10\sqrt{\delta})|\hat{\mathcal{C}}|$ colors $i\in \hat{\mathcal{C}}$. This implies that $x\in Y\backslash Y_{i}$ for at least $(1-10\sqrt{\delta})|\hat{\mathcal{C}}|$ colors $i\in \hat{\mathcal{C}}$. 
Hence  
$$
|X_Y\backslash (\cup_{k\in [2]}X_Y^k)|(1-10\sqrt{\delta})|\hat{\mathcal{C}}|\leq \sum_{i\in \hat{\mathcal{C}}}|Y\backslash Y_i|\leq \sum_{i\in \hat{\mathcal{C}}}(|Y\backslash Y_1|+|Y_1\Delta Y_i|)\leq (2\delta +2\epsilon)n|\hat{\mathcal{C}}|.
$$
It follows that $|X_Y\backslash (\cup_{k\in [2]}X_Y^k)|\leq 4\delta n$. Then move vertices in $X_A\backslash (\cup_{k\in [2]}X_A^k)$ to $B$ and vertices in $X_B\backslash (\cup_{k\in [2]}X_B^k)$ to $A$. Without loss of generality, assume that $|B|-|A|=r$, where  $0\leq r\leq 8{\delta}n+\sigma$, where $\sigma=0$ if $n$ is even and $\sigma=1$ if $n$ is odd. Define $V_{bad}:=\cup_{k\in [2]}(X^k_A\cup X^k_B)\cup X$. Obviously, $|V_{bad}|< (\sqrt{\epsilon}+\frac{1}{2}\sqrt{\delta})n$. Moreover, each vertex in $Y\backslash V_{bad}$ lies in $Y_i$ for at least $(1-10\sqrt{\delta})|\hat{\mathcal{C}}|$ colors $i\in \hat{\mathcal{C}}$.

Next, we give an application of the transversal blow-up lemma (see \cite{2023chengBlowup}) for embedding transversal Hamilton paths inside very dense bipartite graph collections, which can be proved by minor modifications to the proof of \cite{2024cheng}*{Lemma 6.1}, and we omit the proof here.
\begin{claim}\label{lemma4.1}
   Assume $W,Z\in \{A,B\}$. Let $W^*\subseteq W\backslash V_{bad}$ and $Z^*\subseteq Z\backslash V_{bad}$, where $|W^*|,|Z^*|\geq \eta n$, $W^*\cap Z^*=\emptyset$ and $|W^*|-|Z^*|=t\in \{0,1\}$. Let $T^*=Z^*$ if $t=0$ and $T^*=W^*$ if $t=1$. Let $\mathcal{C^*}\subseteq \mathcal{C}$ satisfy $|\mathcal{C^*}|=|W^*|+|Z^*|-1$, where $\mathcal{C^*}\subseteq \mathcal{C}_1$ if $W=Z$ and $\mathcal{C^*}\subseteq \mathcal{C}_2$ if $W\neq Z$. Let $W^-\subseteq W^*$ and $T^+\subseteq T^*$ with $|W^-|,|T^+|\geq \frac{\eta n}{8}$. Then there is a transversal Hamilton path in $\{G_i[W^*,Z^*]:i\in \mathcal{C}^*\}$ starting in $W^-$ and ending in $T^+$.
    \end{claim}

The subsequent result can be used to connect two disjoint  short rainbow paths into a single short rainbow path. 
\begin{claim}[Connecting tool]\label{conn}
    Assume $P=u_1u_2\ldots u_s$ and $Q=v_1v_2\ldots v_t$ are two disjoint rainbow paths inside $\mathcal{G}$ and $s+t\leq 5\eta n$.  
    \begin{enumerate}
        \item[{\rm (i)}] If $u_s\in A\backslash V_{bad}$, $v_1\in B\backslash V_{bad}$ and $|\mathcal{C}_2\backslash col(P\cup Q)|\geq 11\sqrt{\delta} n$, then there are three colors  $c_1,c_2,c_3\in\mathcal{C}_2\backslash col(P\cup Q)$ and two vertices $w_1\in B\backslash (V(P\cup Q)\cup V_{bad})$,  $w_1'\in A\backslash (V(P\cup Q)\cup V_{bad})$ such that $u_1Pu_sw_1w_1'v_1Qv_t$ is a rainbow path with colors $col(P\cup Q)\cup \{c_1,c_2,c_3\}$.
        \item[{\rm (ii)}]  If $u_s,v_1\in Y\backslash V_{bad}$ and $|\mathcal{C}_2\backslash col(P\cup Q)|\geq 11\sqrt{\delta}n$, then there are two colors  $c_1,c_2\in\mathcal{C}_2\backslash col(P\cup Q)$ and one vertex $w_1\in Z\backslash (V(P\cup Q)\cup V_{bad})$ such that $u_1Pu_sw_1v_1Qv_t$ is a rainbow path with colors $col(P\cup Q)\cup \{c_1,c_2\}$.
        \item[{\rm (iii)}] If $u_s,v_1\in Y\backslash V_{bad}$ and $|\mathcal{C}_1\backslash col(P\cup Q)|\geq 11\sqrt{\delta}n$, then there are two colors  $c_1,c_2\in \mathcal{C}_1\backslash col(P\cup Q)$ and one vertex $w_1\in Y\backslash (V(P\cup Q)\cup V_{bad})$ such that $u_1Pu_sw_1v_1Qv_t$ is a rainbow path with  colors $col(P\cup Q)\cup \{c_1,c_2\}$. (See Figure \ref{HC-conn}.)
 \end{enumerate}
\end{claim}
\begin{figure}[ht!]
    \centering
    \includegraphics[width=0.8\linewidth]{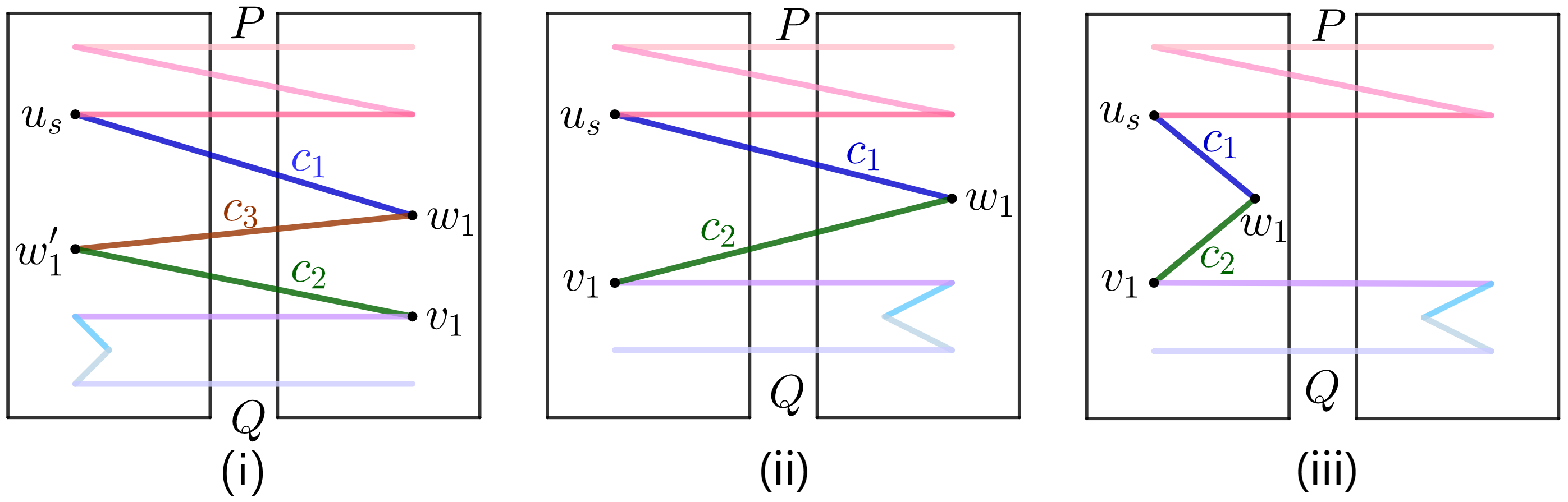}
    \caption{Connecting tool.}
    \label{HC-conn}
\end{figure}

\begin{proof}[Proof of Claim \ref{conn}]
We only give the proof of (i), the other two statements can be proved by a similar discussion, whose procedures are omitted.

Note that $u_s,v_1\not\in V_{bad}$. Then $u_s$ (resp. $v_1$) lies in $A_i$ (resp. $B_i$) for at least $(1-10\sqrt{\delta})|\hat{\mathcal{C}}|$ colors $i\in \hat{\mathcal{C}}$.  Since $|\mathcal{C}_2\backslash col(P\cup Q)|\geq 11\sqrt{\delta}n$, we obtain that $u_s$ (resp. $v_1$) lies in $A_i$ (resp. $B_i$) for at least $11\sqrt{\delta}n-10\sqrt{\delta}|\hat{\mathcal{C}}|\geq\sqrt{\delta}n$ colors $i\in \mathcal{C}_2$. Hence there are two colors $c_1,c_2\in \mathcal{C}_2\backslash col(P\cup Q)$ such that $u_s\in A_{c_1}$ and $v_1\in B_{c_2}$. It is routine to check that 
\begin{align*}
    &|N_{G_{c_1}}(u_s)\cap (B\backslash V_{bad})|\\
\geq& |N_{G_{c_1}}(u_s)\cap B_{c_1}|-|B_1\Delta B_{c_1}|-|X\cup X_A\cup X_B|\\
\geq& (\frac{1}{2}-3\sqrt{\delta})n-(2\delta+\sqrt{\epsilon}+\frac{1}{2}\sqrt{\delta})n\geq (\frac{1}{2}-5\sqrt{\delta})n.
\end{align*}
Hence there exists a vertex $w_1\in N_{G_{c_1}}(u_s)\cap (B\backslash V_{bad})$ 
that avoids $V(P\cup Q)$. Similarly, since $w_1\notin V_{bad}$, there is a color $c_3\in \mathcal{C}_2\backslash (col(P\cup Q)\cup \{c_1,c_2\})$ such that $w_1\in B_{c_3}$. On the other hand,
\begin{align*}
&|N_{G_{c_2}}(v_1)\cap N_{G_{c_3}}(w_1)\cap  (A\backslash V_{bad})|\\
\geq & |N_{G_{c_2}}(v_1)\cap (A\backslash V_{bad})|+|N_{G_{c_3}}(w_1)\cap  (A\backslash V_{bad})|-\frac{n}{2}\\
\geq & (\frac{1}{2}-10\sqrt{\delta})n.
\end{align*}
Hence, there is a vertex $w_1'\in N_{G_{c_2}}(v_1)\cap N_{G_{c_3}}(w_1)\cap  (A\backslash V_{bad})$ that avoids $V(P\cup Q)$. Therefore, $u_1Pu_sw_1w_1'v_1Qv_t$ is a rainbow path inside $\mathcal{G}$ with colors $col(P\cup Q)\cup \{c_1,c_2,c_3\}$.
\end{proof}

Based on the sizes of $|\mathcal{C}_1|$ and $|\mathcal{C}_2|$, our  proof is distinguished into the following three cases.   

{\bf Case 1}. $|\mathcal{C}_1|< \eta n$. 

In this case, each vertex in $Y\backslash (X\cup X_Y^2)$ belongs to $Y_i$ for at least $(1-10\sqrt{\delta})|\mathcal{C}_2|$ colors $i\in \mathcal{C}_2$. Recall that $|B|-|A|=r$ with  $0\leq r\leq 8{\delta}n+\sigma$, where $\sigma=0$ if $n$ is even and $\sigma=1$ if $n$ is odd. The proof is divided into four steps, as shown in the Figure \ref{HC-C1}.

\begin{figure}[ht!]
    \centering
    \includegraphics[width=130mm]{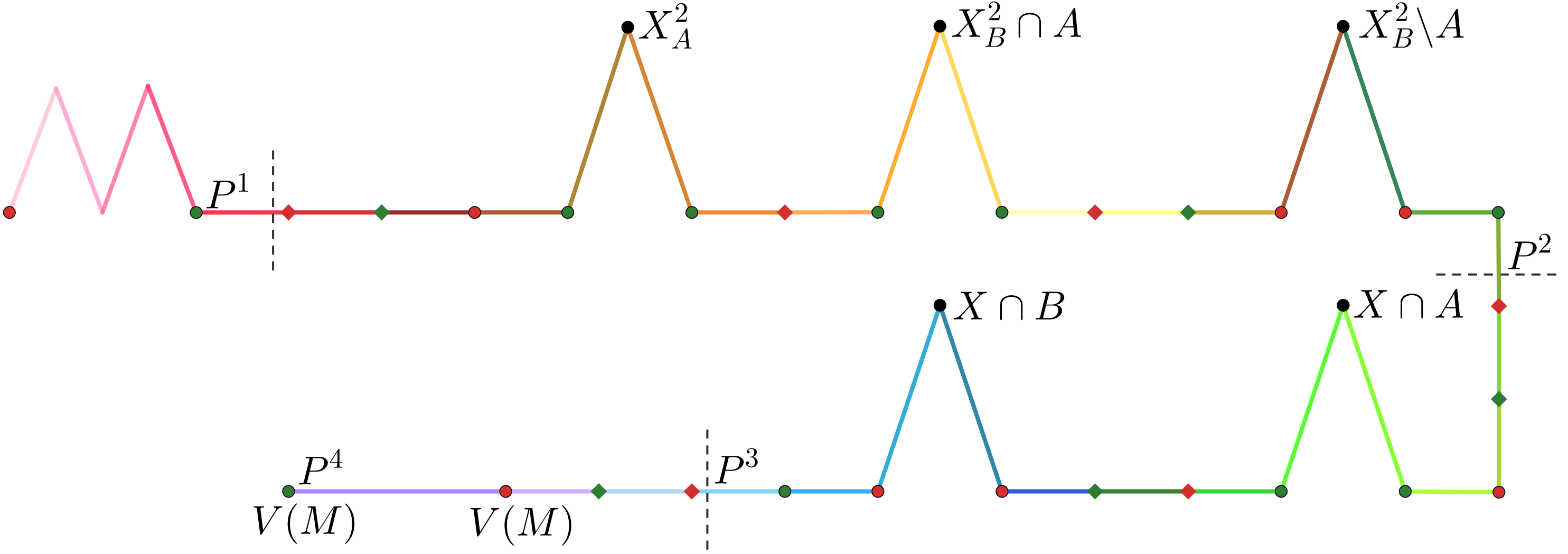}
    \caption{Step\,1 - Step\,3. The vertex with red (resp. green) lies in $A\backslash  V_{bad}$ (resp. $B\backslash  V_{bad}$), and the symbol  ``$\blacklozenge$'' denotes a vertex that is  used to connect rainbow paths.}
    \label{HC-C1}
\end{figure}

{\bf Step 1. Balance the number of vertices in $A$ and $B$.}

In this step, we claim that if  $\mathcal{G}$ is not a  graph collection described in Theorem \ref{th3}, then by moving vertices in $B\cap V_{bad}$ or deleting rainbow paths inside $\mathcal{G}[B]$, we can make the number of remaining vertices in $B$ and $A$ differ by $\sigma$. Denote $s:=|V_{bad}\cap B|$. If $|B|-|A|-\sigma\leq 2s$, then move $\frac{|B|-|A|-\sigma}{2}$ vertices in $V_{bad}\cap B$ to $A$, we can get our desired result. Hence, it suffices to consider $|B|-|A|-\sigma>2s$. We first  move all vertices in $V_{bad}\cap B$ to $A$. Then $B\cap V_{bad}=\emptyset$. Assume that $\{Q_1,\ldots,Q_t\}$ is a set of disjoint maximal  rainbow paths in $\mathcal{G}[B]$. If $|E(Q_1)|+\cdots+|E(Q_{t})|\geq |B|-|A|-\sigma$, then there must exist a set of disjoint rainbow paths, say $\{Q_1',\ldots,Q_{t'}'\}$, such that  $|E(Q_1')|+\cdots+|E(Q_{t'}')|= |B|-|A|-\sigma$. That is,  $|A|-t'=|B|-(|V(Q_1')|+\cdots+|V(Q_{t'}')|)-\sigma$. Using  Claim \ref{conn} (ii) to connect $Q_1',\ldots,Q_{t'}'$ into a single rainbow path $P^1$, whose endpoints are in different parts. (In fact, by using Claim \ref{conn} we only get a rainbow path with two endpoints in $B$. But we usually want to find a rainbow path with endpoints in different parts. Hence  we extend it by using an unused color of $\mathcal{C}_2$ and an unused vertex in $A$). Therefore, $|A\backslash V(P^1)|=|B\backslash V(P^1)|-\sigma$, as desired. 

In what follows, it suffices to consider $|E(Q_1)|+\cdots+|E(Q_{t})|< |B|-|A|-\sigma= r-2s-\sigma$. Now, we move all vertices of $Q_1,\ldots,Q_t$ to $A$.  
Let $\mathcal{C}'=\mathcal{C}\backslash col(\cup_{i\in [t]}Q_i)$. 
Hence \allowdisplaybreaks
\begin{align*}
    |A|
    &\leq \frac{n-r}{2}+s+2(r-2s-\sigma)\leq \frac{n}{2}+16{\delta}n,  \\
|B|&\geq \frac{n+r}{2}-s-2(r-2s-\sigma)\geq \frac{n}{2}-16{\delta}n,\\
|\mathcal{C}'|&=|\mathcal{C}|-(|E(Q_1)|+\cdots+|E(Q_{t})|)\geq (1-8\delta)n.
\end{align*}

By the maximality of $\{Q_1,\ldots,Q_t\}$, we know  $G_i[B]=\emptyset$ for all  $i\in \mathcal{C}'$. It follows from  $\delta(\mathcal{G})\geq \lceil\frac{n}{2}-1\rceil$ that $|A|\geq \lceil\frac{n}{2}-1\rceil$. 
Together with Theorem \ref{lemma5.4}, we obtain that  $\mathcal{G}$ is one of the graph collections described in Theorem \ref{th3},  as desired. 

In this step, if we only move vertices in $V_{bad}$ from $B$ to $A$, then let $P^1$ be a null graph (i.e., no vertices); otherwise, assume $P^1$ is a rainbow path obtained in the above, whose length is at most $24\delta n$ and endpoints are  $u_1\in A\backslash V_{bad},\,v_1\in B\backslash V_{bad}$. Hence $|A\backslash V(P^1)|=|B\backslash V(P^1)|-\sigma$. 

{\bf Step 2. Construct a series of disjoint rainbow $P_3$ such that all centers of them are exactly all vertices in $V_{bad}$.}

Recall that $|X_A^2\cup X_B^2|\leq |X_A\cup X_B|<\frac{1}{2}\sqrt{\delta}n$.  Let $v\in X_Y^2$. If $v\in Z$, then $v$ is moved from $B$ to $A$ in Step 1. Hence $v\in X_B$. Therefore, $v\in A_i\cup B_i$ for at least $(1-3\sqrt{\delta})|\mathcal{C}_2|$ colors $i\in \mathcal{C}_2$ and $v\not\in B_i$ for at least $10\sqrt{\delta}|\mathcal{C}_2|$ colors $i\in \mathcal{C}_2$. Hence, $v\in A_i$ for at least $7\sqrt{\delta}|\mathcal{C}_2|-24\delta n>4|X_A^2\cup X_B^2|$ colors $i\in \mathcal{C}_2\backslash col(P^1)$. If $v\in Y$ (i.e., $v$ has not been  moved in Step 1), then $v$ belongs to $Y_i$ for at least $(3\sqrt{\delta}-24\delta)n>4|X_A^2\cup X_B^2|$ colors $i\in \mathcal{C}_2\backslash col(P^1)$. Hence, using colors in $\mathcal{C}_2\backslash col(P^1)$, we can greedily choose 
$|X_A^2\cup X_B^2|$  disjoint rainbow $P_3$, such that all centers of them are exactly all vertices in $X_A^2\cup X_B^2$. Moreover, if $v\in (X_A^2\cup X_B^2)\cap Y$, then the rainbow $P_3$ with center in $v$ has endpoints in $Z\backslash (V_{bad}\cup V(P^1))$. 

Connect those rainbow $P_3$ with centers in $Y$ (if exists) into a single rainbow path with two endpoints in different parts by Claim \ref{conn} (ii). Then connect the two resulted rainbow paths and $P^1$ by Claim \ref{conn} (i)-(ii), we obtain  a rainbow path $P^2=u_1\ldots v_2$ with length at most $|E(P^1)|+4|X_A^2\cup X_B^2|+4\leq 24\delta n+2\sqrt{\delta}n+4$ and $v_2\in B\backslash V_{bad}$. 

For vertices in $X$, we consider the following claim. Notice that the position (in $A$ or $B$) of vertices in $X$ will not work in our proof, so if $v\in X$ is a vertex moved from $B$ to $A$, then we only need consider $v\in A$. 
\begin{claim}\label{claim5}
    Assume $X\backslash V(P^2)=\{x_1,\ldots,x_s\}$ with $s\leq \sqrt{\epsilon}n$. If 
$x_i\in X\cap Y$, then there exist  $c_i^1,c_i^2,c_i^3,c_i^4\in \mathcal{C}_2\backslash (col(P^2)\cup \{c_{\ell}^1,c_{\ell}^2,c_{\ell}^3,c_{\ell}^4:{\ell}\in  [i-1]\})$ and $x_i^1,\,x_i^2,\,x_i^3,x_i^4\in Z\backslash (V(P^2)\cup V_{bad}\cup  \{x_{\ell}^1,\,x_{\ell}^2,\,x_{\ell}^3,\,x_{\ell}^4:{\ell}\in  [i-1]\})$ such that $x_ix_i^k\in E(G_{c_i^k})$ for all $k\in[4]$.
\end{claim}
\begin{proof}[Proof of Claim \ref{claim5}]
    Suppose there exists an $i_0\in [s]$ such that  Claim \ref{claim5} holds for all $i\in [i_0-1]$ but does not hold for $i_0$. Without loss of generality, assume that $x_{i_0}\in X\cap A$. 
    Hence for all but at most three $j\in  \mathcal{C}_2\backslash (col(P^2)\cup \{c_{\ell}^1,c_{\ell}^2,c_{\ell}^3,c_{\ell}^4:{\ell}\in  [i_0-1]\})$ we have $d_{G_j}(x_{i_0},B)\leq 4(i_0-1)+3+|V(P^2)|$. 
    Thus, for such a color  $j$, \allowdisplaybreaks
    \begin{align*}
        d_{G_j}(x_{i_0},A_j)&\geq d_{G_j}(x_{i_0},A)-|A\backslash A_1|-|A_1\Delta A_j|-|X_A\backslash X_A^2|\\
        &\geq \delta(G_j)-d_{G_j}(x_{i_0},B)-|A\backslash A_1|-|A_1\Delta A_j|-|X_A\backslash X_A^2|
        \\
        &\geq \frac{n}{2}-1-(4(i_0-1)+3+|V(P^2)|)-2{\epsilon}n-2\delta n-4\delta n\\
        &\geq \frac{n}{2}-1-(4\sqrt{\epsilon}n+24\delta n+2\sqrt{\delta}n+4)-2{\epsilon}n-6\delta n\\
        &\geq (\frac{1}{2}-3\sqrt{\delta})n.
    \end{align*}   
Hence $x_{i_0}\in A_{j}\cup B_{j}$. Otherwise, there exists a new characteristic partition $(A_{j}, B_{j}',C_{j}')$ of $G_j$ with $B_{j}'=B_j\cup \{x_{i_0}\}$, which contradicts {\bf (B3)}. By the choice of $j$, we know $x_{i_0}\in A_{j}\cup B_{j}$ for all but at most three $j\in  \mathcal{C}_2\backslash (col(P^2)\cup \{c_{\ell}^1,c_{\ell}^2,c_{\ell}^3,c_{\ell}^4:{\ell}\in  [i_0-1]\})$. It is routine to check that 
     \begin{align*}
         &| \mathcal{C}_2\backslash (col(P^2)\cup \{c_{\ell}^1,c_{\ell}^2,c_{\ell}^3,c_{\ell}^4:{\ell}\in  [i_0-1]\})|-3\\
         \geq &|\mathcal{C}_2|-24\delta n-2\sqrt{\delta}n-4-4\sqrt{\epsilon}n-3\\
         \geq& |\mathcal{C}_2|-3\sqrt{\delta}n.
     \end{align*}
     Therefore, $x_{i_0}\in C_j$ for at most $3\sqrt{\delta}n$ colors  $j\in \mathcal{C}_2$, which implies $x_{i_0}\not\in X$, a contradiction. 
\end{proof}
By Claim \ref{claim5}, for each $x_i\in X\cap Y$, there exists a  rainbow path $Q_{x_i}=x_i^1x_ix_i^2$ with $col(Q_{x_i})=\{c_i^1,c_i^2\}\subseteq \mathcal{C}_2$ and $x_i^1,x_i^2\in Z\backslash V_{bad}$. 
Furthermore, each color or endpoint in $Q_{x_i}$ can be replaced by two colors or vertices not in $P^2$. Connect rainbow paths in $\{Q_{x_i}:x_i\in X\cap Y\}$ by Claim~\ref{conn}~(i)-(ii) into a single rainbow path with endpoints in different parts (if $|X\cap Y|\geq 1$). Then using Claim \ref{conn} (i)-(ii) to connect $P^2$ and the resulted two rainbow paths into a single rainbow path $P^3$ with endpoints $u_1\in A\backslash V_{bad}$ and $v_3\in B\backslash V_{bad}$, whose length is at most $|E(P^2)|+4|X|+4\leq 24\delta n+2\sqrt{\delta}n+4\sqrt{\epsilon}n+8$.

{\bf Step 3. Select a rainbow matching inside $\{G_i[A\backslash V(P^3),B\backslash V(P^3)]:i\in \mathcal{C}_{bad}\backslash col(P^3)\}$.}

Choose a maximum rainbow matching, say $\Tilde{M}$, in $\{G_i[A\backslash V(P^3),B\backslash V(P^3)]:i\in \mathcal{C}_{bad}\backslash col(P^3)\}$. If $\Tilde{M}$ contains an edge with color $j$ such that  $G_j[A\backslash V(P^3\cup \Tilde{M}),B\backslash V(P^3\cup \Tilde{M})]$ contains no $2$-matching, then delete all such edges from $\Tilde{M}$ and denote the resulted rainbow matching by $M$.  Connect $P^3$ and all rainbow edges in $M$ by Claim \ref{conn} (i), we obtain a rainbow path $P^4$ with endpoints $u_1\in A\backslash V_{bad}$ and $v_4\in B\backslash V_{bad}$, whose length is at most $|E(P^3)|+12\delta n\leq 2\sqrt{\delta}n+4\sqrt{\epsilon}n+36\delta n+8$. 


Denote $\mathcal{C}_{bad}':=\mathcal{C}_{bad}\backslash col(P^4)$ and let  $j\in \mathcal{C}_{bad}'$. Hence in $G_j$, all but at most one vertex in  $Y\backslash V(P^3\cup \Tilde{M})$ is adjacent to at least $\frac{n}{2}-1-|V(P^3\cup \Tilde{M})|-1\geq (\frac{1}{2}-3\sqrt{\delta})n$ vertices in $Y$. Therefore, $G_j$ is $(3\sqrt{\delta},K_{\lfloor\frac{n}{2}\rfloor}\cup K_{\lceil\frac{n}{2}\rceil})$-extremal for all $j\in \mathcal{C}_{bad}'$. 

{\bf Step 4. Construct two disjoint rainbow paths inside $\mathcal{G}[A\backslash V(P^4)]$ and $\mathcal{G}[B\backslash V(P^4)]$ by using unused colors in $\mathcal{C}_1\cup \mathcal{C}_{bad}'$}.

It is straightforward to check that 
$$
 |B\backslash V(P^4)|-|A\backslash V(P^4)|=|B\backslash V(P^1)|-|A\backslash V(P^1)|=\sigma.
$$
By the construction of $P^4$, we know that if $w\in V(P^4)$ is a vertex not in $(V(P^1)\cap B)\cup V_{bad}$, then it is possible to avoid $w$ when construct $P^4$; if $i\in col(P^4)$ is a color not in $col(P^1-A)\cup col(M)$, then it is possible to avoid $i$ when construct $P^4$. 
That is, we can assume that $w\not\in V(P^4)$ unless $w\in (V(P^1)\cap B)\cup V_{bad}$ and $i\not\in col(P^4)$ unless $i\in col(P^1-A)\cup col(M)$.  

Furthermore, in the process of choosing rainbow paths inside $\mathcal{G}[B]$ (in Step 1), one may  use colors in $\mathcal{C}_1\cup \mathcal{C}_{bad}'$ before other colors and avoid vertices in $V_{bad}\cup V(M)$. This implies that either $\mathcal{C}_1\cup \mathcal{C}_{bad}'\subseteq col(P^1)$ or $col(P^1-A)\subseteq \mathcal{C}_1\cup \mathcal{C}_{bad}'$. Denote $\Tilde{\mathcal{C}}:=(\mathcal{C}_1\cup \mathcal{C}_{bad}')\backslash col(P^1)$. 

\begin{claim}\label{parity}
Let $Q=x_1\ldots x_s$ be a rainbow path inside $\mathcal{G}$ such that $s\leq 3\sqrt{\delta}n$, $V_{bad}\subseteq V(Q)$, $(\mathcal{C}_1\cup \mathcal{C}_{bad})\backslash \Tilde{\mathcal{C}}\subseteq col(Q)$, $x_1\in A\backslash V_{bad}$ and $x_s\in B\backslash V_{bad}$. 
If $|\Tilde{\mathcal{C}}|-||B\backslash V(Q)|-|A\backslash V(Q)||$ is a nonnegative even integer, then $\mathcal{G}$ has a transversal Hamilton cycle.
\end{claim}
\begin{proof}[Proof of Claim \ref{parity}]
    Using colors in $\Tilde{\mathcal{C}}$, we can greedily choose two disjoint  rainbow paths $P^5$ and $P^6$ in $\mathcal{G}[A\backslash V(Q)]$ and $\mathcal{G}[B\backslash V(Q)]$ with lengths  $\frac{|\Tilde{\mathcal{C}}|-(|B\backslash V(Q)|-|A\backslash V(Q)|)}{2}$ and  $\frac{|\Tilde{\mathcal{C}}|+(|B\backslash V(Q)|-|A\backslash V(Q)|)}{2}$ respectively, whose endpoints are $u_5,v_5\in A\backslash V_{bad}$ and  $u_6,v_6\in B\backslash V_{bad}$. (Note that if a rainbow path has length $0$, then it contains one vertex instead of the null graph.) Applying Claim \ref{conn} (i)-(ii) to connect $Q,P^5$ and $P^6$ in turn, we get a rainbow path $P$ with endpoints $x_1\in A\backslash V_{bad}$ and $v_6\in B\backslash V_{bad}$. Clearly, $|V(P)|\leq 3\sqrt{\delta}n+3\delta n+\eta n+4$ and $|A\backslash V(P)|=|B\backslash V(P)|$. 

Notice that $x_1,v_6\not\in V_{bad}$. Hence there exist $i_1,i_2\in \mathcal{C}_2\backslash col(P)$ such that $x_1\in A_{i_1}$ and $v_6\in B_{i_2}$. Let  $W^*=A\backslash V(P),\,T^*=B\backslash V(P)$, $\mathcal{C}^*=\mathcal{C}\backslash (col(P)\cup \{i_1,i_2\})$, $W^-=N_{G_{i_2}}(v_6)\cap W^*$ and $T^+=N_{G_{i_1}}(x_1)\cap Z^*$. It is routine to check that $|W^*|=|T^*|$, $|\mathcal{C}^*|=|W^*|+|T^*|-1$, 
\begin{align*}
    |W^-|&\geq d_{G_{i_2}}(v_6,A)-|V(P)|\geq d_{G_{i_2}}(v_6,A_{i_2})-|A_1\Delta A_{i_2}|-|X_A\backslash X_A^2|-|V(P)|\\
    &\geq (\frac{1}{2}-3\sqrt{\delta})n-2\delta n-4{\delta}n-|V(P)|> (\frac{1}{2}-2\eta)n.
\end{align*}
Similarly, $|T^+|> \eta n$.  Applying Claim \ref{lemma4.1} yields that $\{G_i[W^*,T^*]:i\in \mathcal{C}^*\}$ contains a transversal path $P'$ starting at $v'\in W^-$ and ending at $u'\in T^+$. Thus, $x_1Pv_6v'P'u'x_1$ is a transversal Hamilton cycle inside $\mathcal{G}$ (see Figure \ref{HC-4.7}).  
\end{proof}

\begin{figure}[ht!]
    \centering
    \includegraphics[width=0.35\linewidth]{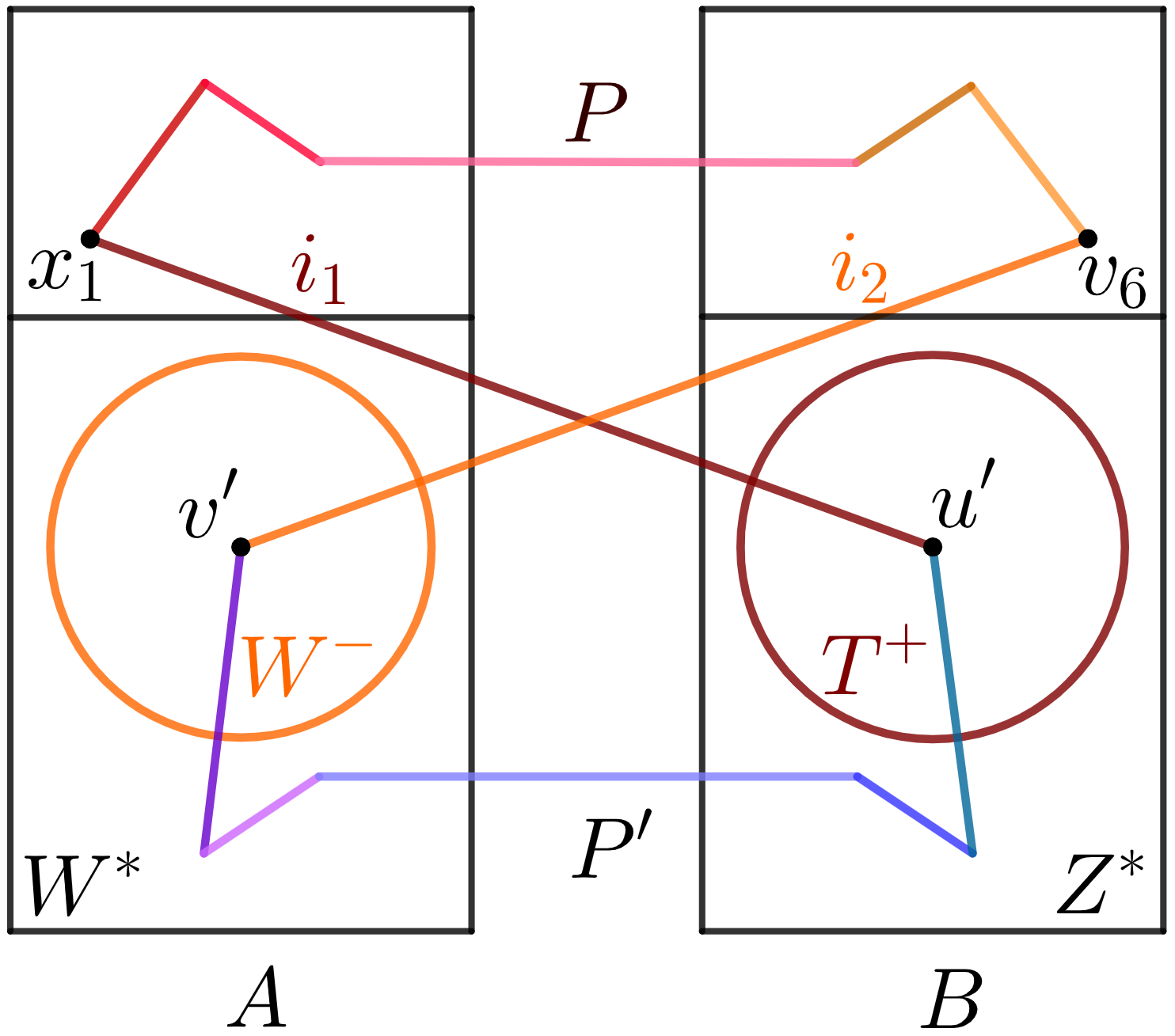}
    \caption{}
    \label{HC-4.7}
\end{figure}

If $n$ is even and $|\Tilde{\mathcal{C}}|$ is even, or $n$ is odd and $|\Tilde{\mathcal{C}}|$ is odd, then in view of Claim \ref{parity}, there exists a transversal Hamilton cycle inside $\mathcal{G}$, a contradiction. Hence, it suffices to consider $n$ is even and $|\Tilde{\mathcal{C}}|$ is odd, or $n$ is odd and $|\Tilde{\mathcal{C}}|$ is even. Therefore, either $|\Tilde{\mathcal{C}}|-(|B\backslash V(P^4)|-|A\backslash V(P^4)|)$ is a positive odd integer or $|\Tilde{\mathcal{C}}|=0$ and $n$ is odd. 




\begin{claim}\label{claim4.7}
\begin{enumerate}
   \item[{\rm (i)}] If $|\Tilde{\mathcal{C}}|\geq 1$, then $E(G_{i}[A])=\emptyset$ for all $i\in \mathcal{C}\backslash \Tilde{\mathcal{C}}$,
    \item[{\rm (ii)}]  $E(G_{i}[A,B])=\emptyset$ for all $i\in \Tilde{\mathcal{C}}$.
\end{enumerate}
\end{claim}
\begin{proof}[Proof of Claim \ref{claim4.7}]
(i)\ Since $|\Tilde{C}|\geq 1$, one has $col(P^1-A)\subseteq \mathcal{C}_1\cup \mathcal{C}_{bad}'$. Suppose that there exists an $i_0\in \mathcal{C}\backslash \Tilde{\mathcal{C}}$ such that $E(G_{i_0}[A])\ne \emptyset$. Choose $w_1w_2\in E(G_{i_0}[A])$. Clearly, $w_1,w_2\not\in V(P^1)\cap B$. Recall that $w\not\in V(P^4)$ unless $w\in (V(P^1)\cap B)\cup V_{bad}$ and $i_0\not\in col(P^4)$ unless $i_0\in col(P^1-A)\cup col(M)$. 
If $i_0\in col(P^1-A)\cup col(M)$, then let $\Tilde{P}^4$ be a rainbow path obtained by deleting the edge with color $i_0$ from $P^4$ and  connecting its two endpoints by Claim~\ref{conn}~(i) (if $i_0\in col(M)$) or Claim \ref{conn} (ii) (if $i_0\in col(P^1-A)$). It is routine to check that $|(\mathcal{C}_1\cup \mathcal{C}_{bad}')\backslash col(\Tilde{P}^4)|-(|B\backslash V(\Tilde{P}^4)|-|A\backslash V(\Tilde{P}^4)|)$ is a positive odd integer. Let $P^4:=\Tilde{P}^4$ if $i_0\in col(P^1-A)\cup col(M)$. Therefore, it suffices to consider $i_0\not\in col(P^4)$. 

$\bullet$ $w_1,w_2\not\in V_{bad}$. Then $w_1,w_2\not\in V(P^4)$. 
Connect $w_2$ with the endpoint $u_1$ of $P^4$ by Claim \ref{conn} (ii). Denote by $Q_0$ the resulted rainbow path  (see Figure \ref{HC-Q0}).

$\bullet$ Exactly one of $w_1,w_2$ lies in $V_{bad}$. Assume, without loss of generality, that $w_1\not\in V_{bad}$ and $w_2\in  V_{bad}$. Hence $w_1\not\in V(P^4)$ and $w_2\in V_{bad}\cap A$. Thus, there exists a rainbow subpath $u_1P^4w_2^1w_2^2w_2$ of $P^4$ with $w_2^1\in A\backslash V_{bad}$, $w_2^2\in B\backslash V_{bad}$. Replace $w_2^1w_2^2w_2$ with $w_1w_2$ in $P^4$, and connect $w_1$ with the vertex in $N_{P^4}(w_2^1)\backslash \{w_2^2\}$ (if exists) by Claim \ref{conn} (i). Let $Q_1$ be the resulted rainbow path (see Figure~\ref{HC-Q0}). 

$\bullet$ $w_1,w_2\in V_{bad}$. By rearranging, one may assume $w_1$ and $w_2$ is connected by a rainbow path $w_1w_1^1w_2^1w_2^2w_2$ in  $P^4$, where $w_1^1,w_2^2\in B\backslash V_{bad}$ and $w_2^1\in A\backslash V_{bad}$. Replace $w_1w_1^1w_2^1w_2^2w_2$ with $w_1w_2$ in $P^4$. Denote by $Q_2$ the resulted rainbow path (see Figure \ref{HC-Q0}). 
\begin{figure}[ht!]
    \centering
    \includegraphics[width=1\linewidth]{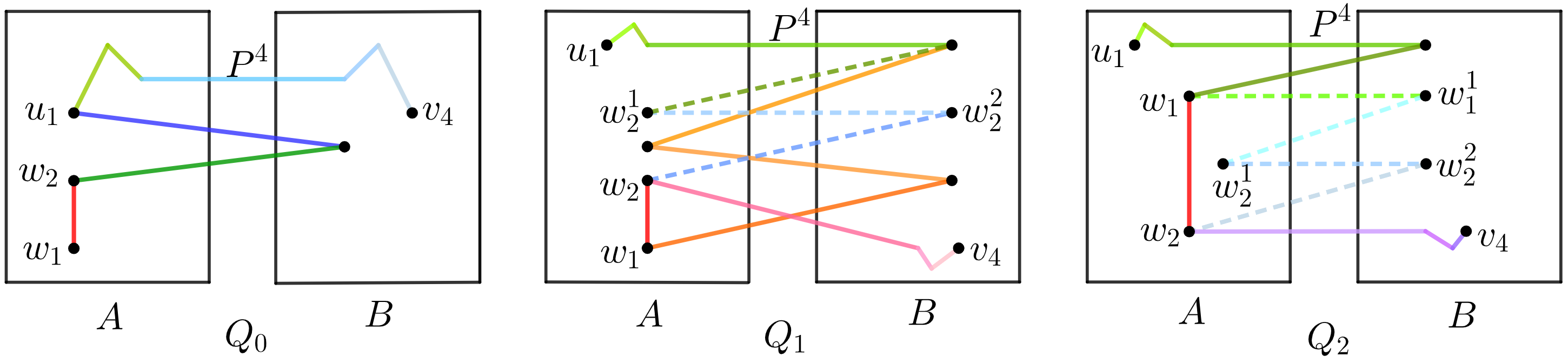}
    \caption{}
    \label{HC-Q0}
\end{figure}

In each of the above cases, we get a rainbow path $Q_i$ ($i\in \{0,1,2\}$) such that $|\Tilde{\mathcal{C}}|-||B\backslash V(Q_i)|-|A\backslash V(Q_i)||$ is a nonnegative even integer.   It follows from Claim \ref{parity} that $\mathcal{G}$ contains  a transversal Hamilton cycle, a contradiction.



(ii)\ Suppose that $E(G_{i_1}[A,B])\neq \emptyset$ for some $i_1\in \Tilde{\mathcal{C}}$. Choose $z_1z_2\in E(G_{i_1}[A,B])$ with $z_1\in A$ and $z_2\in B$. Notice that $col(P^1-A)\subseteq \mathcal{C}_1\cup \mathcal{C}_{bad}'$. If $z_i\in V_{bad}\cap Y$ for $i\in [2]$, then $P^4$ can be written as $u_1P^4z_i^0z_i^1z_iz_i^2P^4v_4$, where $z_i^0\in Y\backslash V_{bad}$ and $z_i^1,z_i^2\in Z\backslash V_{bad}$; if $z_2\in V(P^1)\cap B$, then $P^4$ can be written as $u_1P^4z_2z_2'P^4v_4$, where $z_2z_2'\in E(P^1)\cap E(G_{i_2})$ for some $i_2\in \mathcal{C}$. Clearly,  $i_2\in \mathcal{C}_2$ if $z_2'\in A$ and $i_2\in (\mathcal{C}_1\cup \mathcal{C}_{bad}')\backslash \Tilde{\mathcal{C}}$ if $z_2'\in B$. Notice that $V(P^1)\cap B\cap V_{bad}=\emptyset$. Hence $z_2'\not\in V_{bad}$.

$\bullet$ $z_1,z_2\not\in (V(P^1)\cap B)\cup V_{bad}$. Then $z_1,z_2\not\in V(P^4)$. Connect $z_2$ with the endpoint $u_1$ of $P^4$ by Claim \ref{conn} (i). Let  $Q_3$ be the resulted rainbow path and $\Tilde{\mathcal{C}}_3=\Tilde{\mathcal{C}}\backslash \{i_1\}$. 


$\bullet$ Exactly one of $z_1,\,z_2$ lies in $(V(P^1)\cap B)\cup V_{bad}$. Without loss of generality, assume that $z_1\not\in (V(P^1)\cap B)\cup V_{bad}$ and $z_2\in  (V(P^1)\cap B)\cup V_{bad}$ (the another case is more easier since $z_1\not\in V(P^1)\cap B$). Hence $z_1\not\in V(P^4)$.  If $z_2\in  V(P^1)\cap B$,  then connect $u_1P^1z_2z_1$ with $z_2'P^4v_4$ into a single rainbow path by Claim~\ref{conn} (i)-(ii). Let  $Q_4$ be the resulted rainbow path,  let $\Tilde{\mathcal{C}}_4=\Tilde{\mathcal{C}}\backslash \{i_1\}$ if $z_2'\in A$ and $\Tilde{\mathcal{C}}_4=(\Tilde{\mathcal{C}}\backslash \{i_1\})\cup \{i_2\}$ if $z_2'\in B$. If $z_2\in V_{bad}$, then connect $u_1P^4z_2^1$ with $z_1z_2z_2^2P^4v_4$ into a single rainbow path by Claim~\ref{conn} (ii). Let  $Q_5$ be the resulted rainbow path and $\Tilde{\mathcal{C}}_5=\Tilde{\mathcal{C}}\backslash \{i_1\}$. 

$\bullet$ $z_1,z_2\in (V(P^1)\cap B)\cup V_{bad}$. Then $z_1\in V_{bad}$. If $z_2\in V_{bad}$, by rearranging, one may assume $z_1$ and $z_2$ is connected by a rainbow path $z_1z_1^2z_1^3z_2^3z_2^1z_2$ in  $P^4$, where $z_1^2,z_2^3\in B\backslash V_{bad}$ and $z_1^3,z_2^1\in A\backslash V_{bad}$. Replace such a rainbow path with $z_1z_2$ in $P^4$. Let  $Q_6$ be the resulted rainbow path and $\Tilde{\mathcal{C}}_6=\Tilde{\mathcal{C}}\backslash \{i_1\}$. 
If $z_2\in  V(P^1)\cap B$, then delete vertices $z_1^0,z_1^1,z_1,z_1^2$ from $P^4$ and connect the resulted two components by Claim \ref{conn} (i). Denote the resulted rainbow path by $\Tilde{P}^4$. Then delete the edge $z_2z_2'$ from $\Tilde{P}^4$. Applying Claim \ref{conn} (i)-(ii) again to connect $u_1\Tilde{P}^4z_2z_1z_1^1$ with $z_2'\Tilde{P}^4v_4$ into a single rainbow path $Q_7$. Let  $\Tilde{\mathcal{C}}_7=\Tilde{\mathcal{C}}\backslash \{i_1\}$ if $z_2'\in A$ and $\Tilde{\mathcal{C}}_7=(\Tilde{\mathcal{C}}\backslash \{i_1\})\cup \{i_2\}$ if $z_2'\in B$. 

In each of the above cases, we get a rainbow path $Q_i$ ($i\in [3,7]$) such that $|\Tilde{\mathcal{C}_i}|-||B\backslash V(Q_i)|-|A\backslash V(Q_i)||$ is a nonnegative even integer. Applying Claim \ref{parity} yields that $\mathcal{G}$ contains  a transversal Hamilton cycle, a contradiction. Hence $E(G_{i}[A,B])=\emptyset$ for all $i\in \Tilde{\mathcal{C}}$.
\end{proof}




If $|\Tilde{\mathcal{C}}|=0$, then $n$ is odd. By a similar discussion as the proof of Claim \ref{claim4.7}, we can show that $E(G_i[B])=E(G_i[V(P^1)\cap B])$ for all $i\in \mathcal{C}$. Then $|A|\geq \frac{n-1}{2}$. Recall that $|B|\geq |A|$. Hence $|A|=\frac{n-1}{2}$ and $|B|=\frac{n+1}{2}$. This implies that $P^1$ is a null graph. Thus, $G_i[B]=\emptyset $ for all $i\in \mathcal{C}$. Therefore, $\mathcal{G}$ is the half-split graph collection and it does not contain transversal Hamilton cycles, as desired. 

If $|\Tilde{\mathcal{C}}|\geq 1$, then based on Claim \ref{claim4.7} and the fact that $\delta(\mathcal{G})\geq \frac{n}{2}-1$, we know $n$ is even, $|A|=\frac{n}{2}$ and $|B|=\frac{n}{2}$, and  $G_i[A]=\emptyset$ for all but at most $\eta n+3\delta n$ colors in ${\mathcal{C}}$. 
By Theorem \ref{lemma5.4}, we know $\mathcal{G}$ must be one of the collections described in Theorem \ref{th3}, as desired.


{\bf Case 2. $|\mathcal{C}_2|< \eta n$.} 

In this case, each vertex in $Y\backslash (X\cup X_Y^1)$ belongs to $Y_i$ for at least  $(1-10\sqrt{\delta})|\mathcal{C}_1|$ colors $i\in \mathcal{C}_1$. 
By a similar discussion as Claim \ref{claim5}, we get the following claim.
\begin{claim}\label{claim6}
    Assume $X=\{x_1,\ldots,x_s\}$ with $s\leq \sqrt{\epsilon}n$. If 
$x_i\in X\cap Y$, then there exist  $c_i^1,c_i^2,c_i^3,c_i^4\in \mathcal{C}_1\backslash \{c_{\ell}^1,c_{\ell}^2,c_{\ell}^3,c_{\ell}^4:{\ell}\in  [i-1]\}$ and $x_i^1,\,x_i^2,\,x_i^3,\,x_i^4\in Y\backslash (V_{bad}\cup  \{x_{\ell}^1,\,x_{\ell}^2,\,x_{\ell}^3,\,x_{\ell}^4:{\ell}\in  [i-1]\})$ such that $x_ix_i^k\in E(G_{c_i^k})$ for all $k\in [4]$.
\end{claim}


By Claim \ref{claim6} and a similar discussion as Steps 1-3 in Case 1, the following hold: 
\begin{enumerate}
    \item[{\rm \bf (C1)}] In $\{G_i[A]\cup G_i[B]:i\in \mathcal{C}_1\}$, there are  $|V_{bad}|$ disjoint rainbow $P_3$ with centers in $V_{bad}$ and endpoints in $(A\cup B)\backslash V_{bad}$. Let $\mathbf{P}$ be a set consisting of those rainbow $P_3$.
    \item[{\rm \bf (C2)}] For colors in  $\mathcal{C}_{bad}$, there exists a maximal rainbow matching, say $M$,  inside $\{G_i[A\backslash V(\mathbf{P})]\cup G_i[B\backslash V(\mathbf{P})]:i\in\mathcal{C}_{bad}\}$ such that for each edge in $M[Y]$ with color $j$ we have $G_j[Y\backslash V(\mathbf{P}\cup {M})]$ contains a $2$-matching. 
    Denote $\mathcal{C}_{bad}':=\mathcal{C}_{bad}\backslash col(M)$. Then $G_{j}$ is $(3\sqrt{\delta},K_{\lfloor\frac{n}{2}\rfloor,\lceil\frac{n}{2}\rceil})$-extremal for all $j\in \mathcal{C}_{bad}'$.

    
    \item[{\rm \bf (C3)}] Let $q=|\mathcal{C}_2\cup \mathcal{C}_{bad}'|$. In the graph collection $\{G_i[A\backslash V(\mathbf{P}\cup M),B\backslash V(\mathbf{P}\cup M)]:i\in \mathcal{C}_2\cup \mathcal{C}_{bad}'\}$, there exists a transversal  matching $M'=\{u_1v_1,u_2v_2,\ldots,u_qv_q\}$. 
    \item[{\rm \bf (C4)}] 
    We can assume a vertex $v\notin V(\mathbf{P}\cup M\cup M')$ unless $v\in V_{bad}$; and a color $c\notin col(\mathbf{P})$ if  $c\in \mathcal{C}_1$. 
\end{enumerate}

Next, we give the following claim. 
\begin{claim}\label{path}
    Assume that $Q_1=z_1\ldots z_s$  and $Q_2=z_1'\ldots z_t'$ are  two disjoint rainbow paths inside $\mathcal{G}$ such that $z_1,z_1'\in A\backslash V_{bad}$, $z_s,z_t'\in B\backslash V_{bad}$ and $s+t<4\eta n$. If $V_{bad}\subseteq V(Q_1\cup Q_2)$ and $\mathcal{C}_2\cup \mathcal{C}_{bad}\subseteq col(Q_1\cup Q_2)$, then $\mathcal{G}$ contains a transversal Hamilton cycle. 
\end{claim}
\begin{proof}[Proof of Claim \ref{path}]
Note that 
$z_1,z_1'\in A\backslash V_{bad}$ and $z_s,z_t'\in B\backslash V_{bad}$. There exist $i_1,i_2,i_3,i_4\in \mathcal{C}_1\backslash col(Q_1\cup Q_2)$ such that $z_1\in A_{i_1}$, $z_1'\in A_{i_2}$, $z_s\in B_{i_3}$ and $z_t'\in B_{i_4}$. Split $\mathcal{C}_1\backslash (col(Q_1\cup Q_2)\cup \{i_1,i_2,i_3,i_4\})$ into two parts $\mathcal{C}_a\cup \mathcal{C}_b$, where $|\mathcal{C}_a|=|A\backslash V(Q_1\cup Q_2)|-1$ and $|\mathcal{C}_b|=|B\backslash V(Q_1\cup Q_2)|-1$. Let $I_1\cup I_2$ and $J_1\cup J_2$ be equitable partitions of $A\backslash V(Q_1\cup Q_2)$ and $B\backslash V(Q_1\cup Q_2)$ with $|I_1|\geq |I_2|$ and $|J_1|\geq |J_2|$, respectively. Define
$$
\mathcal{G}_a=\{G_i[I_1,I_2]:i\in \mathcal{C}_a\}\ \text{and}\ \mathcal{G}_b=\{G_i[J_1,J_2]:i\in \mathcal{C}_b\}.
$$
Notice that 
$$
|Y\backslash V(Q_1\cup Q_2)|\geq \frac{n-r}{2}-|V(Q_1\cup Q_2)|>\frac{n}{2}-5\eta n,
$$
and for $j\in [2]$ we have 
\begin{align*}
|N_{G_{i_1}}(z_1,I_j)|&\geq d_{G_{i_1}}(z_1,A\backslash V(Q_1\cup Q_2))-|I_{3-j}|\\
&\geq d_{G_{i_1}}(z_1,A_{i_1})-|A_1\Delta A_{i_1}|-|X_A\backslash X_A^1|-|V(Q_1\cup Q_2)|-|I_{3-j}|\\
&\geq (\frac{1}{2}-3\sqrt{\delta})n-(2\delta+4\delta+8\eta)n-\lceil\frac{n+r}{4}\rceil >\frac{n}{4}-5\eta n.
\end{align*}
Similarly, we have  $|N_{G_{i_2}}(z_1',I_j)|,|N_{G_{i_3}}(z_s,J_j)|,|N_{G_{i_4}}(z_t',J_j)|>\frac{n}{4}-5\eta n$ for each $j\in  [2]$. 
By Claim~\ref{lemma4.1}, there exists a transversal path $Q_a$ inside $\mathcal{G}_a$ with two endpoints $x_a\in N_{G_{i_1}}(z_1,I_1)$ and $y_a\in N_{G_{i_2}}(z_1',I)$; and a transversal path $Q_b$ inside $\mathcal{G}_b$ with two endpoints $x_b\in N_{G_{i_3}}(z_s,J_1)$ and $y_b\in N_{G_{i_4}}(z_t',J)$ (here $y_a\in I_2$ (resp. $y_b\in J_2$) if and only if $|I_1|=|I_2|$ (resp. $|J_1|=|J_2|$) is even). Thus,  $z_1Q_1z_sx_bQ_by_bz_t'Q_2z_1'y_aQ_ax_az_1$ 
is a transversal Hamilton cycle inside $\mathcal{G}$ (see Figure~\ref{HC-4.10}).
\end{proof}
\begin{figure}[htbp]
\centering
\begin{minipage}[t]{0.4\textwidth}
\centering
\includegraphics[scale=1.5]{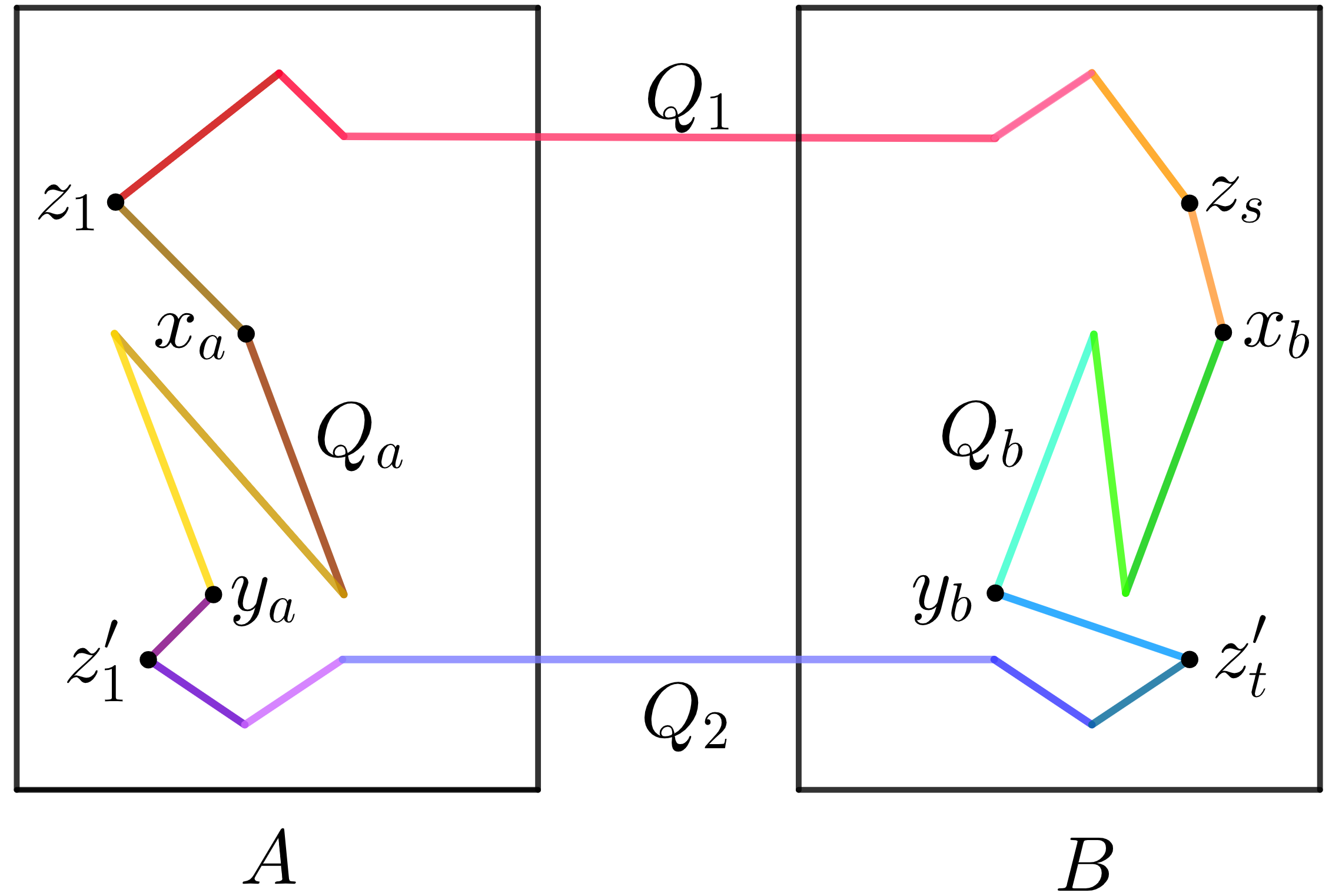}
    \caption{}
    \label{HC-4.10}
\end{minipage}
\begin{minipage}[t]{0.5\textwidth}
\centering
\includegraphics[scale=1.5]{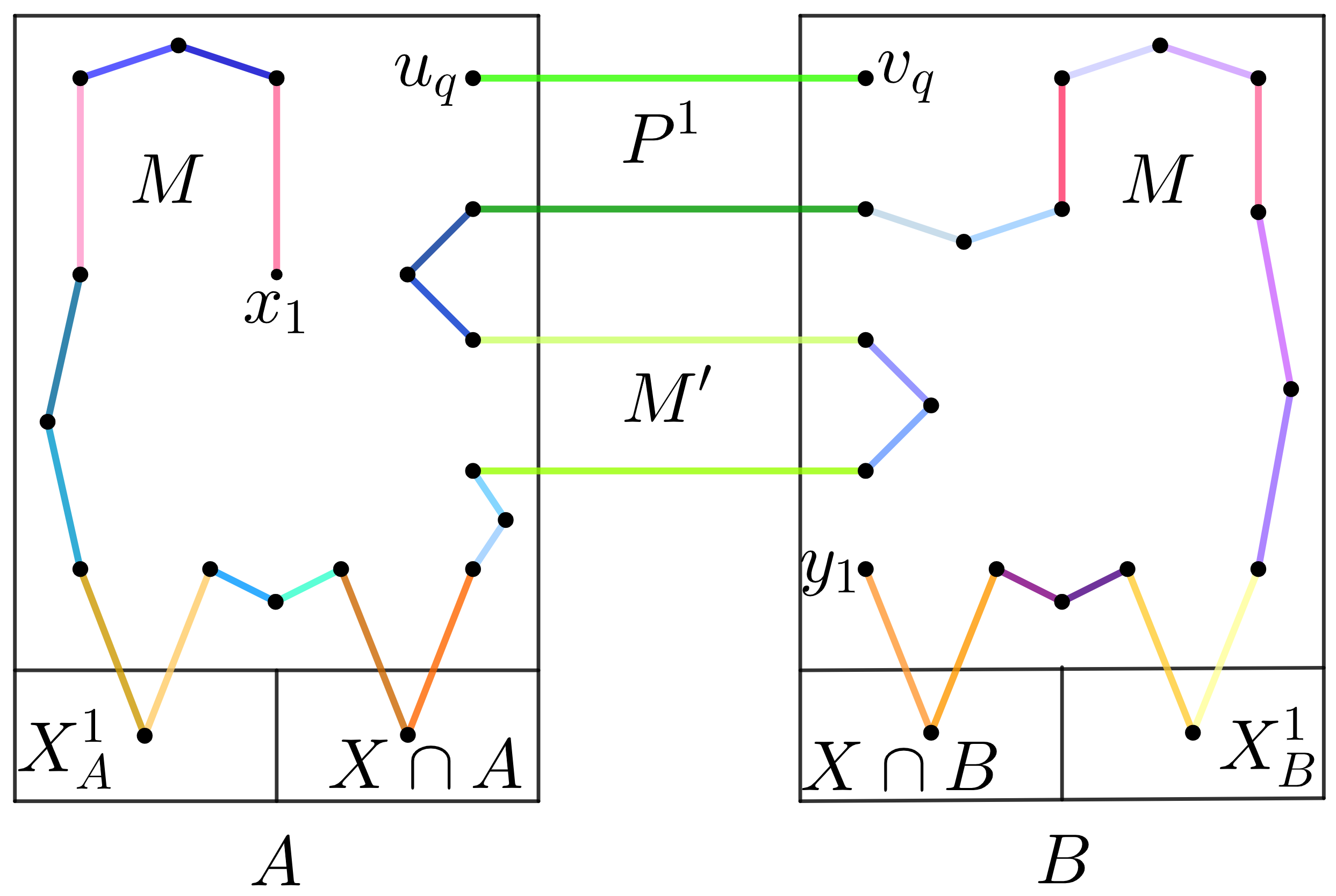}
    \caption{}
    \label{HC-C2}
\end{minipage}
\end{figure}


Now, we proceed by considering the value of $q$.

{\bf Subcase 2.1.} $q$ is even  and $q\geq 2$. 

Using Claim \ref{conn} (iii) to connect all rainbow paths in $\mathbf{P},M$ and $M'-u_qv_q$ into a single rainbow path $P$, whose 
length is at most $(4\sqrt{\epsilon}+2\sqrt{\delta}+9\delta +3\eta )n$ and  endpoints are $x\in A\backslash V_{bad}$, $y\in B\backslash V_{bad}$ (see Figure \ref{HC-C2}). Thus, $P$ and $u_qv_q$ are two disjoint rainbow paths satisfying all conditions in Claim  \ref{path}. Thus,  $\mathcal{G}$ contains a transversal Hamilton cycle, a contradiction. 

{\bf Subcase 2.2.}  $q$ is odd.  

Suppose that $E(G_{i_1}[A,B])\neq \emptyset$ for some $i_1\in \mathcal{C}_1\cup col(M)$. Choose $w_1w_2\in E(G_{i_1}[A,B])$ with $w_1\in A$ and $w_2\in B$. Based on {\bf (C1)}, the edge  $w_1w_2$ can be extended to a rainbow path of length at most $3$, say $Q_1$, which has endpoints in $A\backslash V_{bad}$ and $B\backslash V_{bad}$ respectively. Next, by using Claim \ref{conn} (iii) and  avoiding vertices and colors in $Q_1$, one may connect all rainbow $P_3$ in $\mathbf{P}$ with centers not in $\{w_1,w_2\}$, all rainbow edges in $M$ except the possible edge with color $i_1$ and all rainbow edges in $M'$ into a single rainbow path $P^1$, whose length is at most $(4\sqrt{\epsilon}+2\sqrt{\delta}+9\delta +3\eta )n$ and endpoints are $x_1\in A\backslash V_{bad}$, $y_1\in B\backslash V_{bad}$. 
Then $P^1$ and $Q_1$ satisfy all conditions in Claim \ref{path}. Therefore, $\mathcal{G}$ contains a transversal Hamilton cycle, a contradiction. Hence $E(G_{i}[A,B])=\emptyset$ for all $i\in \mathcal{C}_1\cup col(M)$. Recall that $\delta(\mathcal{G})\geq \frac{n}{2}-1$. Then $n$ is even and $G_i[A]=G_i[B]=K_{\frac{n}{2}}$ for all $i\in \mathcal{C}_1\cup col(M)$.


If $q=1$, then it is easy to see that $\mathcal{G}$ contains no transversal Hamilton cycle, as desired. Next, we consider  $q\geq 3$. Suppose that there exists an $i_2\in \mathcal{C}_2\cup \mathcal{C}_{bad}'$ such that $E(G_{i_2}[A])\cup E(G_{i_2}[B])\neq \emptyset$. Without loss of generality, assume that $z_1z_2\in E(G_{i_2}[A])$ and $u_1v_1$ is the  edge in $M'$ with color $i_2$. 
Based on {\bf (C1)}, we can extend $z_1z_2$ to a rainbow path $Q_2$ with two endpoints in $A\backslash V_{bad}$ and length at most $3$. 
Using Claim \ref{conn} (iii) to connect $Q_2$, all rainbow $P_3$ in $\mathbf{P}$ with centers not in $\{z_1,z_2\}$, all rainbow edges in $M$ and all rainbow edges in $M'-\{u_1v_1,u_qv_q\}$ into a single rainbow path $P^2$, whose length is at most $(4\sqrt{\epsilon}+2\sqrt{\delta}+9\delta +3\eta )n+6$ and  endpoints are $x_2\in A\backslash V_{bad}$, $y_2\in B\backslash V_{bad}$.
Clearly, ${P}^2$ and $u_qv_q$ are two rainbow paths satisfying all conditions in Claim \ref{path}. Therefore, $\mathcal{G}$ contains a transversal Hamilton cycle, a contradiction. Thus, for each $i\in \mathcal{C}_2\cup \mathcal{C}_{bad}'$, $E(G_{i}[A])\cup E(G_{i}[B])=\emptyset$ and so $G_i$ is a subgraph of $K_{\frac{n}{2},\frac{n}{2}}$. It follows that $\mathcal{G}$ is a spanning collection of $\mathcal{H}_{n-q}^q$ with odd $q$.

We claim that $\mathcal{H}_{n-q}^q$ contains no transversal Hamilton cycles if $n$ is even and $q$ is odd.  Suppose not, let $C$ be a transversal Hamilton cycle and give it an arbitrary direction to make it a directed cycle. We say an edge of $C$ is $1$-type if it comes from some $G_i$ with $i\in \mathcal{C}_1\cup col(M)$ and $2$-type if it comes from some $G_i$ with $i\in \mathcal{C}_2\cup \mathcal{C}_{bad}'$. Note that the number of $2$-type edges directed from $A$ to $B$ is equal to the number of $2$-type edges directed from $B$ to $A$. Hence, the total number of $2$-type edges in $C$ is even, which implies that $q$ is even, a contradiction.

{\bf Subcase 2.3.}  $q=0$. 

Suppose that $\{G_i[A,B]:i\in \mathcal{C}\}$ contains a rainbow $2$-matching. Let $\{w_1w_2,z_1z_2\}$ be such a rainbow $2$-matching with colors $c_1,c_2$. 
Based on {\bf (C1)}, we can extend $w_1w_2$ and $z_1z_2$ to two disjoint rainbow paths  $Q_3$ and $Q_4$ with length at most $3$, each of which has one endpoint in $A\backslash V_{bad}$ and the other in $B\backslash V_{bad}$. By applying Claim \ref{conn} (iii) and  avoiding vertices and colors in $Q_4$, one may connect  $Q_3$, all rainbow $P_3$ in $\mathbf{P}$ with centers outside $\{w_1,w_2,z_1,z_2\}$ and  all rainbow edges in $M$ with colors outside $\{c_1,c_2\}$ into a single rainbow path $P^3$, whose length is at most  $(4\sqrt{\epsilon}+2\sqrt{\delta}+9\delta +3\eta )n+6$ and  endpoints are $x_3\in A\backslash V_{bad}$, $y_3\in B\backslash V_{bad}$. Clearly, $P^3$ and $Q_4$ are two rainbow paths satisfying all conditions in Claim \ref{path}. Therefore, $\mathcal{G}$ contains a transversal Hamilton cycle, a contradiction. 
Hence $\{G_i[A,B]:i\in \mathcal{C}\}$ contains no rainbow $2$-matching. Finally, we finish the proof of this case by proving the following claim. 

\begin{claim}\label{matching}
    If $\mathcal{G}[A,B]$ has no rainbow $2$-matching, then $\mathcal{G}$ does not contain transversal Hamilton cycles. Furthermore, $\mathcal{G}$ must be a graph collection  described in Theorem \ref{th3} (i)(b) or Theorem~\ref{th3}~(ii)~(c).
\end{claim}
\begin{proof}[Proof of Claim \ref{matching}]
Notice that each transversal Hamilton cycle contains at least $2$ disjoint rainbow edges in  $\mathcal{G}[A,B]$. Hence $\mathcal{G}$ contains no transversal Hamilton cycles. Next, we characterize the structure of $\mathcal{G}$.

Assume that each vertex of $A$ has at least one neighbor in $B$ for all $G_i\in \mathcal{G}$. Since $\mathcal{G}[A,B]$ contains no rainbow $2$-matching,  there exists a fixed vertex 
$w\in B$ such that $N_{G_i}(v,B)=\{w\}$ for all $i\in \mathcal{C}$ and all $v\in A$. 
Therefore, $|A|\geq \lceil\frac{n}{2}-1\rceil$.  If $n$ is odd,  then $d_{G_i}(w)=n-1$ and  
$G_i[V\backslash \{w\}]=G_i[A]\cup G_i[B\backslash \{w\}]=2K_{\frac{n-1}{2}}$  for all $i\in \mathcal{C}$,  as desired. If $n$ is even, then there exists an equitable partition $A'\cup B'$ of $V$ and a vertex  $w'\in A'$ such that $G_{i}[A']=K_{\frac{n}{2}}$ and  $E(G_i[A',B'])=E(G_i[\{w'\},B'])$ for all $i\in \mathcal{C}$, as desired. 

In what follows, it suffices to consider that there is a vertex $v\in A$ and a color $i\in \mathcal{C}$ such that $N_{G_i}(v,B)=\emptyset$. 
Hence $|B|\geq |A|\geq \lceil\frac{n}{2}-1\rceil+1$. Therefore, $n$ is even and $|A|=|B|=\frac{n}{2}$.
  
$\bullet$ There exists a color $j_1\in \mathcal{C}$ and a vertex $v_1\in A$   such that $|N_{G_{j_1}}(v_1)\cap B|\geq 2$. 
Since $\mathcal{G}$ contains no rainbow $2$-matching, we know  $G_i[A\backslash \{v_1\},B]=\emptyset$ for all $i\in \mathcal{C}\backslash \{j_1\}$. This implies that $E(G_i[A,B])=E(G_i[\{v_1\},B])$ for all $i\in \mathcal{C}\backslash \{j_1\}$. If $G_i[\{v_1\},B]=\emptyset$ for all $i\in \mathcal{C}\backslash \{j_1\}$, then $G_i[A]=G_i[B]=K_{\frac{n}{2}}$ for all $i\in \mathcal{C}\backslash \{j_1\}$, as desired. 
If there exists some  $j_2\in \mathcal{C}\backslash \{j_1\}$ such that $G_{j_2}[\{v_1\},B]\neq \emptyset$, then let $v_1w\in E(G_{j_2}[\{v_1\},B])$. Therefore, $N_{G_{j_1}}(v,B)\subseteq \{w\}$ for each $v\in A\backslash \{v_1\}$.

If $N_{G_{j_1}}(v,B)=\emptyset$ for all $v\in A\backslash \{v_1\}$, then $E(G_i[A,B])=E(G_i[\{v_1\},B])$ and $E(G_i[A])=K_{\frac{n}{2}}$ for all $i\in \mathcal{C}$, as desired. If there exists $v_2\in A\backslash \{v_1\}$ such that  $N_{G_{j_1}}(v_2,B)= \{w\}$, then $E(G_i[\{v_1\},B])\subseteq \{v_1w\}$ for all $i\in \mathcal{C}\backslash \{j_1\}$. That is, $E(G_{j_1}[A,B])\subseteq \{v_1v:v\in B\}\cup \{vw:v\in A\}$, $E(G_i[A,B])\subseteq \{v_1w\}$ and  $G_i[A]=G_i[B]=K_{\frac{n}{2}}$ for all $i\in \mathcal{C}\backslash \{j_1\}$, as desired. 


$\bullet$ $|N_{G_i}(v,B)|\leq 1$ for all $i\in \mathcal{C}$ and all $v\in A$. By symmetry, one may assume that $|N_{G_i}(v,A)|\leq 1$ for all $i\in \mathcal{C}$ and all $v\in B$. Let $|E(G_{j_3}[A,B])|=\max\{|E(G_{j}[A,B])|:j\in \mathcal{C}\}$. If $|E(G_{j_3}[A,B])|\geq 3$, then $G_i=G_i[A]\cup G_i[B]=2K_{\frac{n}{2}}$ for all $i\in \mathcal{C}\backslash \{j_3\}$, as desired. 

If $|E(G_{j_3}[A,B])|=2$, then assume $E(G_{j_3}[A,B])=\{v_1w_1,v_2w_2\}$, where $v_1,w_1,v_2,w_2$ are pairwise distinct. This implies $G_{j_3}[A]\in \{K_{\frac{n}{2}}-v_1v_2,K_{\frac{n}{2}}\}$ and  $G_{j_3}[B]\in \{K_{\frac{n}{2}}-w_1w_2,K_{\frac{n}{2}}\}$. 
Recall that $\mathcal{G}$ contains no rainbow $2$-matching. Then either $E(G_{j_4}[A,B])=\{v_1w_2,v_2w_1\}$ for a unique $j_4\in \mathcal{C}\backslash \{j_3\}$ and $G_i[A,B]=\emptyset$ for all $i\in \mathcal{C}\backslash \{j_3,j_4\}$, or $E(G_i[A,B])\subseteq \{v_1w_2\}$ for all $i\in \mathcal{C}\backslash \{j_3\}$. The former case implies $E(G_{j_4}[A,B])=\{v_1w_2,v_2w_1\}$, $G_{j_4}[A]\in \{K_{\frac{n}{2}}-v_1v_2,K_{\frac{n}{2}}\}$, $G_{j_4}[B]\in \{K_{\frac{n}{2}}-w_1w_2,K_{\frac{n}{2}}\}$ and $G_i=G_i[A]\cup G_i[B]= 2K_{\frac{n}{2}}$ for all $i\in \mathcal{C}\backslash \{j_3,j_4\}$ (see Figure \ref{Hc-even} (a)), as desired. The latter case implies $E(G_i[A,B])\subseteq \{v_1w_2\}$ and $G_i[A]=G_i[B]=K_{\frac{n}{2}}$ for all $i\in \mathcal{C}\backslash \{j_3\}$ (see Figure \ref{Hc-even} (b)), as desired. 

If $|E(G_{j_3}[A,B])|\leqslant 1$, then $G_i[A]=G_i[B]=K_{\frac{n}{2}}$ for all $i\in \mathcal{C}$ and  all edges in $\{G_i[A,B]:i\in \mathcal{C}\}$ must have a common vertex, as desired.

This completes the proof of Claim \ref{matching}.
\end{proof}

{\bf Case 3.} $|\mathcal{C}_1|\geq \eta n$ and $|\mathcal{C}_2|\geq \eta n$.

In this case, each vertex in $Y\backslash V_{bad}$ lies in $Y_{i}$ for at least $(1-10\sqrt{\delta})|{\mathcal{C}_1}\cup \mathcal{C}_2|$ colors $i\in \mathcal{C}_1\cup {\mathcal{C}_2}$. 
Denote $\cup_{k\in[2]}(X_A^k\cup X_B^k)=\{v_1,v_2,\ldots,v_{t}\}$. Notice that $t<\frac{1}{2}\sqrt{\delta}n$. 
By the definition of $X_Y^k$, we obtain that for each $i\in [t]$, there exist $c_{4i-3},c_{4i-2},c_{4i-1},c_{4i}\in (\mathcal{C}_1\cup \mathcal{C}_2)\backslash \{c_1,c_2,\ldots,c_{4i-5},c_{4i-4}\}$ and $v_i^1,v_i^2,v_i^3,v_i^4\in (A\cup B)\backslash (V_{bad}\cup \{v_{\ell}^1,v_{\ell}^2,v_{\ell}^3,v_{\ell}^4:\ell\in [i-1]\})$ such that 
\begin{itemize}
    \item if $v_i\in X_Y^1$, then $c_{4i-4+j}\in \mathcal{C}_1$, $v_i^j\in Y\cap Y_{c_{4i-4+j}}$ and $v_iv_i^j\in E(G_{c_{4i-4+j}})$ for all $j\in [4]$,
    \item if $v_i\in X_Y^2$, then $c_{4i-4+j}\in \mathcal{C}_2$, $v_i^j\in Z\cap Z_{c_{4i-4+j}}$ and $v_iv_i^j\in E(G_{c_{4i-4+j}})$ for all $j\in [4]$.
\end{itemize}
Therefore, for each $i\in [t]$, there exists a rainbow path  $v_i^1v_iv_i^2$ with colors $c_{4i-3}$ and $c_{4i-2}$. Connect those rainbow paths by using colors in $\mathcal{C}_2$ and Claim \ref{conn} (i)-(ii), one may get a rainbow path $P^1$ with endpoints  $u_1\in A\backslash V_{bad}$ and $w_1\in B\backslash V_{bad}$, whose length is at most $4t\leq 2\sqrt{\delta}n$. 


By a similar discussion as Claim \ref{claim5}, we obtain the following claim. 
\begin{claim}\label{claim4.8}
    Assume $X=\{x_1,\ldots,x_s\}$ with $s\leq \sqrt{\epsilon}n$. Then for each $i\in [s]$, there exist $c_i^1,c_i^2,c_i^3,c_i^4\in \mathcal{C}_2\backslash (col(P^1)\cup \{c_{\ell}^1,c_{\ell}^2,c_{\ell}^3,c_{\ell}^4:{\ell}\in [i-1]\})$ and $x_i^1,\,x_i^2,\,x_i^3,\,x_i^4\in V\backslash (V(P^1)\cup V_{bad}\cup  \{x_{\ell}^1,\,x_{\ell}^2,\,x_{\ell}^3,\,x_{\ell}^4:{\ell}\in [i-1]\})$  such that: if $x_i\in Y$, then $x_ix_i^j\in E(G_{c_i^j})$, and either $x_i^j\in Y$ and $c_i^j\in \mathcal{C}_1$ for all $j\in [4]$, or $x_i^j\in Z$ and $c_i^j\in \mathcal{C}_2$ for all $j\in [4]$.
\end{claim}
By Claim \ref{claim4.8}, there exists a rainbow path $x_i^1x_ix_i^2$ with colors $c_i^1,c_i^2$ for each $i\in [s]$. Connect all those rainbow paths and $P^1$ by Claim \ref{conn} (i)-(ii), we obtain a rainbow path $P^2$ with endpoints  $u_1\in A\backslash V_{bad}$ and $w_2\in B\backslash V_{bad}$, whose length is at most $|E(P^1)|+4\sqrt{\epsilon}n+2\leq (2\sqrt{\delta}+4\sqrt{\epsilon})n+2$.

Choose a maximal rainbow matching, say $M$,  inside $\{G_i[A\backslash V(P^2)]\cup G_i[B\backslash V(P^2)]:i\in\mathcal{C}_{bad}\}$ such that for each edge in $M[Y]$ with color $j$ we have $G_j[Y\backslash V(P^2)]$ contains a $2$-matching.  
Denote $\mathcal{C}_{bad}':=\mathcal{C}_{bad}\backslash col(M)$. Note that $G_{j}$ is $(3\sqrt{\delta},K_{\lfloor\frac{n}{2}\rfloor,\lceil\frac{n}{2}\rceil})$-extremal for all $j\in \mathcal{C}_{bad}'$.  Hence, $\{G_i[A,B]:i\in \mathcal{C}_{bad}'\}$ contains a transversal matching, say $M'$. Connect $P^2$ and all rainbow edges in $M\cup M'$ by Claim \ref{conn} (i)-(ii), we obtain a rainbow path $P^3$ with endpoints  $u_1\in A\backslash V_{bad}$ and $w_3\in B\backslash V_{bad}$, whose length is at most $|E(P^2)|+12\delta n\leq (2\sqrt{\delta}+4\sqrt{\epsilon}+12\delta)n+2$. 

Since $u_1,w_3\not\in V_{bad}$, there exist $i_0,i_1\in \mathcal{C}_1\backslash col(P^3)$ such that $u_1\in A_{i_0}$ and $w_3\in B_{i_1}$. By the construction of $P^3$, one may assume that $v\not\in V(P^3)$ unless $v\in V_{bad}$, and $c\not\in col(P^3)$ unless $c\in \mathcal{C}_{bad}$. 
In what follows, we proceed by considering the parity of $|\mathcal{C}_2\backslash col(P^3)|$.

{\bf Subcase 3.1.}  $|\mathcal{C}_2\backslash col(P^3)|$ is odd. 

Split $A\backslash V(P^3)$ into $I_1$ and $I_2$ such that $|I_1|=\frac{|\mathcal{C}_2\backslash col(P^3)|+1}{2}$, and split $B\backslash V(P^3)$ into $J_1$ and $J_2$ such that $|J_1|=\frac{|\mathcal{C}_2\backslash col(P^3)|+1}{2}$. Let $\mathcal{G}_1=\{G_i[I_1,J_1]:i\in \mathcal{C}_2\backslash col(P^3)\}$. 
By Claim \ref{lemma4.1}, there exists  a transversal path $Q$ inside $\mathcal{G}_1$ with endpoints $x\in I_1$ and $y\in J_1$. Since $x,y\not\in V_{bad}$, there exist $j_0,j_1\in \mathcal{C}_1$ such that $x\in A_{j_0}$ and $y\in B_{j_1}$. 

Suppose that $|I_2|=|A\backslash V(P^3)|-|I_1|=k_1$ and $|J_2|=|B\backslash V(P^3)|-|J_1|=k_2$. We further split $I_2$ into $I_2'$ and $I_2''$ 
such that $|I_2'|=\lceil\frac{k_1}{2}\rceil$, split $J_2$ into $J_2'$ and $J_2''$ 
such that $|J_2'|=\lceil\frac{k_2}{2}\rceil$. Then split the color
set $\mathcal{C}_1\backslash (col(P^3)\cup \{j_0,j_1\})$ into two subsets $\mathcal{C}_a$ and $\mathcal{C}_b$ such that $|\mathcal{C}_a|=k_1-1$ and $|\mathcal{C}_b|=k_2-1$. Let $\mathcal{G}_a=\{G_i[I_2',I_2'']:i\in \mathcal{C}_a\}$ and $\mathcal{G}_b=\{G_i[J_2',J_2'']:i\in \mathcal{C}_b\}$.  Applying Claim \ref{lemma4.1} again yields that there is a transversal path $Q_a$ inside $\mathcal{G}_a$ with endpoints $x_a\in N_{G_{i_0}}(u_1,I_2'),\,y_a\in N_{G_{j_0}}(x,I_2)$; and a transversal path $Q_b$ inside $\mathcal{G}_b$ with endpoints $x_b\in N_{G_{i_1}}(w_3,J_2'),\,y_b\in N_{G_{j_1}}(y,J_2)$. Thus, 
$u_1P^3w_3x_bQ_by_byQxy_aQ_ax_au_1$ is a transversal Hamilton cycle in $\mathcal{G}$ (see Figure \ref{HC-C3}), a contradiction.
\begin{figure}[ht!]
    \centering
    \includegraphics[width=0.45\linewidth]{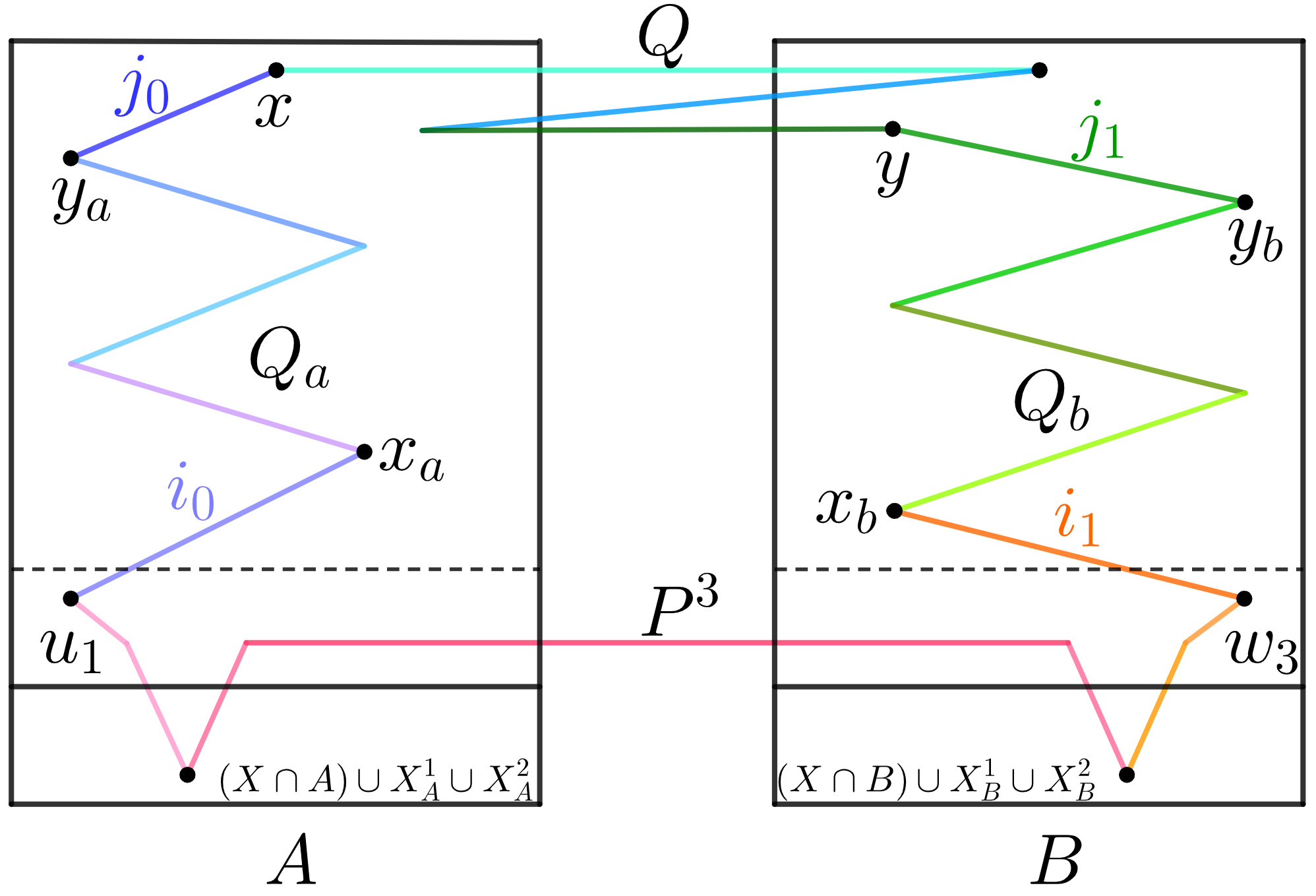}
    \caption{}
    \label{HC-C3}
\end{figure}

{\bf Subcase 3.2.} $|\mathcal{C}_2\backslash col(P^3)|$ is even. 

In this subcase, we first prove the following two claims. 
\begin{claim}\label{claim4.9}
    Assume that $Q_1=z_1\ldots z_k$  and $Q_2=z_1'\ldots z_t'$ are two disjoint rainbow paths inside $\mathcal{G}$ such that $z_1,z_1'\in A\backslash V_{bad}$, $z_k,z_t'\in B\backslash V_{bad}$ and each of them has length at most $3\sqrt{\delta} n$. If $V_{bad}\subseteq V(Q_1\cup Q_2)$, $\mathcal{C}_{bad}\subseteq col(Q_1\cup Q_2)$ and $|\mathcal{C}_2\backslash col(Q_1\cup Q_2)|$ is even, then $\mathcal{G}$ contains a transversal Hamilton cycle. 
\end{claim}
\begin{proof}[Proof of Claim \ref{claim4.9}]
Note that 
$z_1,z_1'\in A\backslash V_{bad}$ and $z_k,z_t'\in B\backslash V_{bad}$. There exist  $l_1,l_2,l_3\in \mathcal{C}_1\backslash col(Q_1\cup Q_2)$ and $l_4\in \mathcal{C}_2\backslash col(Q_1\cup Q_2)$ such that $z_1\in A_{l_1}$, $z_k\in B_{l_2}$, $z_1'\in A_{l_4}$ and $z_t'\in B_{l_3}$. 

Split $A\backslash V(Q_1\cup Q_2)$ into $I_1$ and $I_2$ such that $|I_1|=\frac{|\mathcal{C}_2\backslash col(Q_1\cup Q_2)|}{2}$, and split $B\backslash V(Q_1\cup Q_2)$ into $J_1$ and $J_2$ such that $|J_1|=\frac{|\mathcal{C}_2\backslash col(Q_1\cup Q_2)|}{2}$. Let $\mathcal{G}_1=\{G_i[I_1,J_1]:i\in \mathcal{C}_2\backslash (col(Q_1\cup Q_2)\cup \{l_4\})\}$. By Claim \ref{lemma4.1}, there exists a transversal path $Q$ inside $\mathcal{G}_1$ with endpoints $z_1''\in N_{G_{l_4}}(z_1')\cap J_1$ and $y\in I_1$. Since $y\not\in V_{bad}$, there exists $l_5\in \mathcal{C}_1\backslash (col(Q_1\cup Q_2)\cup \{l_1,l_2,l_3\})$ such that $y\in A_{l_5}$. 

Assume that $|I_2|=k_1$ and $|J_2|=k_2$. 
We further split $I_2$ into $I_2'$ and $I_2''$ 
such that $|I_2'|=\lceil\frac{k_1}{2}\rceil$, and split $J_2$ into $J_2'$ and $J_2''$ 
such that $|J_2'|=\lceil\frac{k_2}{2}\rceil$. Split the color
set $\mathcal{C}_1\backslash (col(Q_1\cup Q_2)\cup \{l_1,l_2,l_3,l_5\})$ into two subsets $\mathcal{C}_a$ and $\mathcal{C}_b$ such that $|\mathcal{C}_a|=k_1-1$ and $|\mathcal{C}_b|=k_2-1$. Let $\mathcal{G}_a=\{G_i[I_2',I_2'']:i\in \mathcal{C}_a\}$ and $\mathcal{G}_b=\{G_i[J_2',J_2'']:i\in \mathcal{C}_b\}$. 
By Claim \ref{lemma4.1}, there exists a transversal path $Q_a$ inside $\mathcal{G}_a$ with endpoints $x_a\in N_{G_{l_1}}(z_1)\cap I_2',y_a\in N_{G_{l_5}}(y)\cap I_2$, and a transversal path $Q_b$ inside $\mathcal{G}_b$ with endpoints $x_b\in N_{G_{l_2}}(z_k)\cap J_2',y_b\in N_{G_{l_3}}(z_t')\cap J_2$. Hence, $z_1Q_1z_kx_bQ_by_bz_t'Q_2z_1'z_1''Qyy_aQ_ax_az_1$ is a transversal Hamilton cycle in $\mathcal{G}$. 
\end{proof}

\begin{claim}\label{claim4.14}
 $E(G_{i}[A,B])=\emptyset$ for all $i\in \mathcal{C}_1\cup col(M)$.
\end{claim}
\begin{proof}[Proof of Claim \ref{claim4.14}]
Suppose that there exists an edge $y_1y_2\in E(G_{i_2}[A,B])$ for some $i_2\in \mathcal{C}_1\cup col(M)$, where $y_1\in A$ and $y_2\in B$. If $i_2\in col(M)$, then assume $uv\in E(M[A])\cap E(G_{i_2})$. Connect $P^2$, all rainbow edges in $M-uv$ and all rainbow edges in $M'$ in turn by Claim \ref{conn} (i)-(ii), we get a rainbow path $\Tilde{P}^3$ with endpoints $\Tilde{u}_1\in A\backslash V_{bad}$ and $\Tilde{w}_3\in B\backslash V_{bad}$. Therefore, $P^3$ can be written as 
$uvv'\Tilde{u}_1\Tilde{P}^3\Tilde{w}_3$, where the colors of  $vv'$ and $v'\Tilde{u}_1$ are in $\mathcal{C}_2$. Since $|\mathcal{C}_2\backslash col(P^3)|$ is even, one has $|\mathcal{C}_2\backslash col(\Tilde{P}^3)|$ is even. Let $P^3:=\Tilde{P}^3$ if $i_2\in col(M)$. Hence we can assume  $i_2\not\in col(P^3)$.  
Recall that $y_i\not\in V(P^3)$ unless it is not in $V_{bad}$ for each $i\in [2]$. Moreover, if $y_i\in V_{bad}$ for some $i\in [2]$, then $P^3$ can be written as $u_1P^3 y_i^1y_iy_i^2P^3w_3$. Assume $y_i^1y_i$ and $y_iy_i^2$ have colors 
$j_{2i-1}$ and $j_{2i}$ in $P^3$, respectively. 

If $y_1,y_2\not\in V_{bad}$, then $P^3$ and $y_1y_2$ are two disjoint rainbow paths satisfying all conditions in Claim \ref{claim4.9}. Hence,  $\mathcal{G}$ contains a transversal Hamilton cycle, a contradiction. Therefore, at least one of $y_1$ and $y_2$ lies in $V_{bad}$. Without loss of generality, assume that $y_2\in V_{bad}$. Observe that either $j_3,j_4\in \mathcal{C}_1$ and $y_2^1,\,y_2^2\in B$, or $j_3,j_4\in \mathcal{C}_2$ and $y_2^1,\,y_2^2\in A$ holds.

Assume $y_1\not\in V_{bad}$. Delete $y_2^1y_2$ from $P^3$, we get two rainbow paths $P'=u_1P^3y_2^1$ and $P''=y_1y_2y_2^2P^3w_3$. Recall that $u_1,y_1\in A\backslash V_{bad}$, $w_3\in B\backslash V_{bad}$ and $y_2^1\notin V_{bad}$. It is routine to checck  that $P'$ can be extended to a rainbow path with one endpoint $u_1$ and the other endpoint $\hat{y}\in B\backslash V_{bad}$, where $\hat{y}=y_2^1$ if $y_2^1\in B$. Therefore, there are two rainbow paths satisfying all conditions in Claim \ref{claim4.9}. Therefore,  $\mathcal{G}$ contains a transversal Hamilton cycle, a contradiction.  

Finally, we consider $y_1,y_2\in V_{bad}$. Then by rearranging, one may assume that $y_1^2P^3y_2^1$ is a rainbow path with length at most $3$. Replace  $y_1P^3y_2$ by the edge $y_1y_2$ with color $i_2$. It is routine to check that $|\mathcal{C}_2\backslash col(u_1P^3y_1y_2P^3w_3)|$ is odd (see Figure \ref{Hc-y12} for all possible cases). Thus, this can be attributed to the Subcase 3.1.  Hence $\mathcal{G}$ contains a transversal Hamilton cycle, a contradiction. 
\end{proof}

\begin{figure}
    \centering \includegraphics[width=0.9\linewidth]{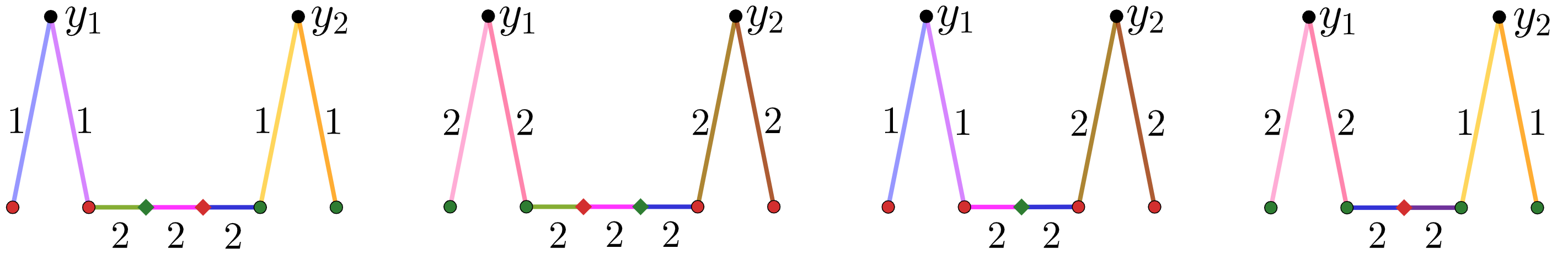}
    \caption{All possible cases for  $y_1,y_2\in V_{bad}$. We use  ``1'' (resp. ``2'') to denote that the color of this edge lies in $\mathcal{C}_1$ (resp. $\mathcal{C}_2$).}
    \label{Hc-y12}
\end{figure}


By Claim \ref{claim4.14}, one has   $E(G_{i}[A,B])=\emptyset$ for all $i\in \mathcal{C}_1\cup col(M)$. Therefore, $n$ is even and $G_i[A]=G_i[B]=K_{\frac{n}{2}}$ for all $i\in \mathcal{C}_1\cup col(M)$. By a similar argument as Claim \ref{claim4.14}, we can show that $E(G_{i}[A]\cup G_i[B])=\emptyset$ for all $i\in \mathcal{C}_2\cup \mathcal{C}_{bad}'$. It follows that $\mathcal{G}$ is a spanning collection of $\mathcal{H}_{n-t}^t$, where $t=|\mathcal{C}_2\cup \mathcal{C}_{bad}'|$.

We claim that $t$ is odd. Otherwise, $t$ is even. Recall that $G_i[A]=G_i[B]=K_{\frac{n}{2}}$ for all $i\in \mathcal{C}_1\cup col(M)$. Then there exist two rainbow paths, say $Q_1:=z_1^1\ldots z_{s_1}^1$ and $Q_2:=z_1^2\ldots z_{s_2}^2$, such that $V(Q_1\cup Q_2)=V_{bad}$ and $col(Q_1\cup Q_2)\subseteq \mathcal{C}_1\cup col(M)$. For $i\in [2]$, it is easy to see that $Q_i$ can be extended to a rainbow path $Q_i':=x^iz_1^i\ldots z_{s_i}^iy^iz^i$, where $x^i,y^i\in A\backslash V_{bad}$, $z_i\in B\backslash V_{bad}$, and the colors of $x^iz_1^i,\,z_{s_i}^iy^i,\,y^iz^i$ are unused colors in $\mathcal{C}_1,\,\mathcal{C}_1,\,\mathcal{C}_2$, respectively. Since $t$ is even, one has $|(\mathcal{C}_2\cup \mathcal{C}_{bad}')\backslash col(Q_1'\cup Q_2')|$ is even. By a similar argument as Claim \ref{claim4.9}, we know $\mathcal{G}$ contains a transversal Hamilton cycle, a contradiction. Therefore, $\mathcal{G}$ is a spanning collection of $\mathcal{H}_{n-t}^t$, where $n$ is even and $t$ is odd.

This completes the proof of Theorem \ref{th3}.
\end{proof}

\section{Proof of Theorem \ref{lemma5.4}}\label{sec5}

In this section, we prove Theorem \ref{lemma5.4}, which characterizes the structure of almost balanced graph collections that do not contain transversal Hamilton cycles. Based on the parity of $n$ and the size of the vertex partition, we split the proof of Theorem \ref{lemma5.4} into four lemmas (see, Lemmas \ref{lemma4.4}, \ref{lemma5.2}, \ref{lemma5.3} and \ref{lemma5.7}).  

\begin{lemma}\label{lemma4.4}
For every $\delta$ with $0<\delta\ll 1$, there exists $n_0\in \mathbb{N}$ satisfying the
following for every odd integer  $n\geq n_0$.  
Let $\mathcal{C}$ be a set of $n$ colors and $\mathcal{G}=\{G_i:i\in \mathcal{C}\}$ be a collection of graphs with common vertex set $V$ of size $n$. Let $A\cup B$ be a partition of $V$ with  $|A|=\frac{n-1}{2}$, and let $\mathcal{C}'\cup \mathcal{C}''$ be a partition of $\mathcal{C}$ with  $|\mathcal{C}''|\leq \delta n$.  Assume that $\delta(\mathcal{G})\geq \frac{n-1}{2}$ and $G_i[B]=\emptyset$ for all $i\in \mathcal{C}'$. If  $\mathcal{G}$ does not contain transversal Hamilton cycles, then $\mathcal{G}$ is the half-split graph collection. 
\end{lemma}
\begin{proof}
In order to show that $\mathcal{G}$ is the half-split graph collection, it suffices to prove $G_i[B]=\emptyset$ for all $i\in \mathcal{C}''$. By adding colors to $\mathcal{C}''$ if necessary we may assume $|\mathcal{C}''|= \delta n>0$ and  $\mathcal{C}''=\{i_1,i_2,\ldots,i_{\delta n}\}$. 
Suppose that $G_{i_1}[B]\neq \emptyset$ and let $uv\in E(G_{i_1}[B])$. Choose a maximal rainbow matching inside $\{G_{i_j}[A,B\backslash \{u,v\}]: j\in [2,\delta n]\}$, say $M=\{e_{i_j}:2\leq j\leq s\}$ with $s\leq \delta n$ and $e_{i_j}\in E(G_{i_j})$, such that $G_{i_j}[A\backslash V(M),B\backslash (V(M)\cup \{u,v\})]\neq \emptyset$ for all $j\in [2,s]$. Notice that $G_{i_j}$ is $(2\delta n,K_{\lfloor\frac{n}{2}\rfloor}\cup K_{\lceil\frac{n}{2}\rceil})$-extremal for all $j\in [s+1,\delta n]$. We greedily select two rainbow paths $P_A$ and $P_B$ inside $\{G_{i_j}[A\backslash V(M)]\cup G_{i_j}[B\backslash (V(M)\cup \{u,v\})]:j\in [s+1,\delta n]\}$ with lengths   $\lfloor\frac{\delta n-s}{2}\rfloor$ and $\lceil\frac{\delta n-s}{2}\rceil$ respectively, where $V(P_A)\subseteq A$ and $V(P_B)\subseteq B$. 

We first consider $\delta n-s$ is even. It is easy to see that $G_i[A,B]= K_{\frac{n-1}{2},\frac{n+1}{2}}$ for  all $i\in \mathcal{C}'$. 
Using colors in $\mathcal{C}'$ and edges between $A$ and $B$, one may connect $uv$, all rainbow edges in $M$, $P_A$ and $P_B$ into a single rainbow path, say $P^1:=u\ldots v_1$, where $v_1\in B$.  Therefore, $ |A\backslash V(P^1)|=|B\backslash V(P^1)|+1$. Observe that $\{G_i[A\backslash V(P^1),B\backslash V(P^1-\{u,v_1\})]:i\in \mathcal{C}\backslash col(P^1)\}$ is a collection of complete bipartite graphs, therefore it contains a transversal path that connects $u$ and $v_1$. Hence $\mathcal{G}$ has a transversal Hamilton cycle, a contradiction. 

Now, assume that $\delta n-s$ is odd. 
Recall that $|A|=\frac{n-1}{2}$ and $\delta(\mathcal{G})\geq \frac{n-1}{2}$. Hence there exists an edge $w_1w_2\in E(G_{i_1}[A\backslash V(M\cup P_A\cup P_B),B])$ with $w_1\in A$ and $w_2\in B$. Notice that one may avoid $w_2$ when choosing  $M$ and $P_B$.  Using colors in $\mathcal{C}'$ and edges between $A$ and $B$, we can connect $w_1w_2$, all rainbow edges in $M$, $P_A$ and $P_B$ into a single rainbow path $P^2$, whose endpoints are in different parts.  Therefore, $ |A\backslash V(P^1)|=|B\backslash V(P^1)|$. Similarly, $\mathcal{G}$ has a transversal Hamilton cycle, a contradiction. 
\end{proof}

\begin{lemma}\label{lemma5.2}
For every $\delta$ with $0<\delta\ll 1$, there exists $n_0\in \mathbb{N}$ satisfying the
following for every even integer  $n\geq n_0$. Let $\mathcal{C}$ be a set of $n$ colors, and $\mathcal{G}=\{G_i:i\in \mathcal{C}\}$ be a collection of graphs with common vertex set $V$ of size $n$. Let $A\cup B$ be a partition of $V$ with  $|A|=\frac{n}{2}-1$, and let $\mathcal{C}'\cup \mathcal{C}''$ be a partition of $\mathcal{C}$ with  $|\mathcal{C}''|\leq \delta n$.  Assume that $\delta(\mathcal{G})\geq \frac{n}{2}-1$ and $G_i[B]=\emptyset$ for all $i\in \mathcal{C}'$. If  $\mathcal{G}$ does not contain transversal Hamilton cycles, then either $G_i[B]=\emptyset$ for all but at most one $i\in \mathcal{C}$, or  $E(G_i[B])\subseteq \{uv\}$ for two fixed vertices $u,v\in B$ and all $i\in \mathcal{C}$.
\end{lemma}
\begin{proof}
In order to prove this lemma, it suffices to show that $\{G_i[B]:i\in \mathcal{C}''\}$ contains no rainbow $P_3$ or $2P_2$. 
By adding colors to $\mathcal{C}''$ if necessary we may assume $|\mathcal{C}''|= \delta n>0$ and  $\mathcal{C}''=\{i_1,i_2,\ldots,i_{\delta n}\}$. 
Suppose that $\{G_i[B]:i\in \mathcal{C}''\}$ contains a rainbow path $u_1u_2u_3$ with $u_1u_2\in E(G_{i_1})$ and $u_2u_3\in E(G_{i_2})$. (The proof of the case that $\{G_i[B]:i\in \mathcal{C}''\}$ contains  a rainbow $2P_2$ is similar).

Choose a maximal  rainbow matching inside $\{G_{i_j}[A,B\backslash \{u_1,u_2,u_3\}]: j\in [3,\delta n]\}$, say $M=\{e_{i_j}:3\leq j\leq s\}$ with $s\leq \delta n$ and $e_{i_j}\in E(G_{i_j})$, such that $G_{i_j}[A\backslash V(M),B\backslash (V(M)\cup \{u_1,u_2,u_3\})]\neq \emptyset$ for all $j\in [3,s]$. Notice that $G_{i_j}$ is $(2\delta n,K_{\lfloor\frac{n}{2}\rfloor}\cup K_{\lceil\frac{n}{2}\rceil})$-extremal for all $j\in [s+1,\delta n]$. We greedily select two rainbow paths $P_A$ and $P_B$ inside $\{G_{i_j}[A\backslash V(M)]\cup G_{i_j}[B\backslash (V(M)\cup \{u_1,u_2,u_3\})]:j\in [s+1,\delta n]\}$ with lengths   $\lfloor\frac{\delta n-s}{2}\rfloor$ and $\lceil\frac{\delta n-s}{2}\rceil$ respectively, where $V(P_A)\subseteq A$ and $V(P_B)\subseteq B$. 

If $\delta n-s$ is even, then by a similar discussion as Lemma \ref{lemma4.4}, it is easy to find a transversal Hamilton cycle inside $\mathcal{G}$, a contradiction. 

Now, assume that $\delta n-s$ is odd. 
Recall that $|A|=\frac{n}{2}-1$ and $\delta(\mathcal{G})\geq \frac{n}{2}-1$. Then there exists an edge $w_1w_2\in E(G_{i_1}[A\backslash V(M\cup P_A\cup P_B),B])$ with $w_1\in A$ and $w_2\in B$.   Moreover, $G_i[A,B]= K_{\frac{n}{2}-1,\frac{n}{2}+1}$ for  all $i\in \mathcal{C}'$. Notice that one may avoid $w_2$ when choosing  $M$ and $P_B$.  
If $w_2\in \{u_2,u_3\}$, then without loss of generality, assume that $w_2=u_2$. Using colors in $\mathcal{C}'$ and edges between $A$ and $B$, we can  connect $w_1u_2u_3$, all rainbow edges in $M$, $P_A$ and $P_B$ in turn to get a single rainbow path $P^1$, whose endpoints are in different parts.  Therefore, $ |A\backslash V(P^1)|=|B\backslash V(P^1)|$. If $w_2\not\in \{u_2,u_3\}$, then using colors in $\mathcal{C}'$ and edges between $A$ and $B$, one may connect $u_2u_3$, $w_1w_2$, all rainbow edges in $M$, $P_A$ and $P_B$ in turn to get a single rainbow path $P^2$, whose endpoints are in $B$.  Therefore, $ |A\backslash V(P^2)|=|B\backslash V(P^2)|+1$. 
In each of the above two cases, it is routine to check that $\mathcal{G}$ has a transversal Hamilton cycle, a contradiction. 
\end{proof}
\begin{lemma}\label{claim4.2}
     Let $\mathcal{C}$ be a set of colors, and $\mathcal{G}=\{G_i[Y,B]:i\in \mathcal{C}\}$ be a collection of bipartite graphs with common bipartition $(Y,B)$  such that $7|Y|<|B|\leq\frac{3}{5}|\mathcal{C}|$. If $\sum_{i\in \mathcal{C}}|E(G_i[Y,B])|\geq t|B||\mathcal{C}|$ for some integer $t$ with $1\leq t\leq |Y|$, then $\mathcal{G}$ contains $t$ disjoint rainbow star $S_5$,  each of them has center in $Y$ and other vertices in $B$.
\end{lemma}
\begin{proof}
    We use induction to prove this result.  
    If $t=1$, then $\sum_{i\in \mathcal{C}}|E(G_i[Y,B])|\geq |B||\mathcal{C}|$. Suppose that $\mathcal{G}$ does not contain any rainbow $S_5$ with center in $Y$ and other vertices in $B$. Then each vertex $v\in Y$ satisfies either $N_{G_i}(v,B)\neq \emptyset$ for at most three $i\in \mathcal{C}$, or there exist three vertices $w,w',w''\in B$ such that $N_{G_i}(v)\subseteq\{w,w',w''\}$ for all $i\in \mathcal{C}$. It follows that 
    $$
    \sum_{i\in \mathcal{C}}|E(G_i[Y,B])|=\sum_{v\in Y}\sum_{i\in \mathcal{C}}d_{G_i}(v,B)\leq \sum_{v\in Y}\max\{3|B|,3|\mathcal{C}|\}= 3|Y||\mathcal{C}|<|B||\mathcal{C}|,
    $$
    a contradiction. That is to say, our result holds for $t=1$.

    Assume the result holds for $t$, and we prove it for $t+1$. Now, we give the following claim.
    \begin{claim}\label{fact4}
        There exists a vertex $w\in Y$ such that $d_{G_i}(w,B)\geq 4t+4$ for at least $4t+4$ colors $i\in \mathcal{\mathcal{C}}$.
    \end{claim}
    \begin{proof}[Proof of Claim \ref{fact4}]
        Suppose not, then for each vertex $v\in Y$, we have $d_{G_i}(v,B)\geq 4t+4$ for at most $4t+3$ colors $i\in \mathcal{\mathcal{C}}$. It implies that
        \allowdisplaybreaks
        \begin{align*}
            \sum_{i\in \mathcal{C}}|E(G_i[Y,B])|&= \sum_{v\in Y} \sum_{i\in \mathcal{C}}d_{G_i}(v,B)\\
            &=\sum_{v\in Y}\Big(\sum_{\substack {i\in \mathcal{C}\\d_{G_i}(v,B)\geq 4t+4}}d_{G_i}(v,B)+\sum_{\substack{i\in \mathcal{C}\\d_{G_i}(v,B)\leq 4t+3}}d_{G_i}(v,B)\Big)\\
            &\leq \sum_{v\in Y}\Big(\sum_{\substack{i\in \mathcal{C}\\d_{G_i}(v,B)\geq 4t+4}}|B|+\sum_{\substack{i\in\mathcal{C}\\d_{G_i}(v,B)\leq 4t+3}}(4t+3)\Big)\\
            &\leq |Y|((4t+3)|B|+|\mathcal{C}|(4t+3))\\
            &\leq (4t+3)|Y|(|B|+|\mathcal{C}|)\\
            &<7(t+1)|Y||\mathcal{C}|\\
            &<(t+1)|B||\mathcal{C}|,
        \end{align*}
        a contradiction. 
    \end{proof}
    By Claim \ref{fact4}, there exists a vertex $w\in Y$ such that $d_{G_i}(w,B)\geq 4t+4$ for at least $4t+4$ colors $i\in \mathcal{\mathcal{C}}$. 
    Let $\mathcal{G}'=\{G_i[Y\backslash \{w\},B]:i\in \mathcal{C}\}$. Then 
    $$
    \sum_{i\in \mathcal{C}}|E(G_i[Y\backslash \{w\},B])|\geq (t+1)|B||\mathcal{C}|-|B||\mathcal{C}|=t|B||\mathcal{C}|.
    $$
    By induction, there are $t$ disjoint rainbow $S_5$ inside $\mathcal{G}'$, each of which has center in $Y\backslash \{w\}$ and other vertices in $B$. Let $\mathbf{S}$ be a set consisting of those disjoint rainbow $S_5$. 
    Recall that $d_{G_i}(w,B)\geq 4t+4$ for at least $4t+4$ colors $i\in \mathcal{C}$. Hence there exist four colors $c^1,c^2,c^3,c^4\in \mathcal{C}\backslash col(\mathbf{S})$ such that $d_{G_{c^j}}(w)\geq 4t+4$ for all $j\in [4]$. Thus, there are four vertices $w^1,w^2,w^3,w^4\in B\backslash V(\mathbf{S})$ such that $ww^j\in E(G_{c^j})$ for all $j\in [4]$. Therefore, $\{G_{c^j}[\{w,w^1,w^2,w^3,w^4\}]:j\in [4]\}$ contains a rainbow $S_5$ with center $w$, which is disjoint with each rainbow star in $\mathbf{S}$, as desired.
\end{proof}

\begin{lemma}\label{Y}
Assume $0<\frac{1}{n}\ll\delta\ll1$ and $0\leq \gamma\leq3{\delta}$. Let $\mathcal{C}$ be a set of $n$ colors, and $\mathcal{G}=\{G_i:i\in \mathcal{C}\}$ be a collection of graphs with common vertex set $V$ of size $n$ such that  $\delta(\mathcal{G})\geq \frac{n}{2}-1$. Let $A\cup B$ be a partition of $V$ with  $|A|=\lfloor\frac{n+1}{2}\rfloor+\gamma n$, and let $\mathcal{C}'\cup \mathcal{C}''$ be a partition of $\mathcal{C}$ with  $|\mathcal{C}''|\leq \delta n$.  Define $Y=\{v\in A: d_{G_i}(v,B)\leq (1-\delta^{\frac{1}{4}})|B|\ \textrm{for at least}\ \delta^{\frac{1}{4}}|\mathcal{C}'|\ \textrm{colors}\  i\in \mathcal{C}'\}$.  Assume that $|Y|>\gamma n$, $\mathcal{G}$ does not contain transversal Hamilton cycles and $G_i[B]=\emptyset$ for all $i\in \mathcal{C}'$. 
\begin{enumerate}
    \item[{\rm (i)}] If $n$ is odd, then $\mathcal{G}$ is the half-split graph collection.
    \item[{\rm (ii)}] If $n$ is even, then there exists a new partition $A'\cup B'$ of $V$ with $|A'|=\frac{n}{2}-1$ such that either $G_i[B']=\emptyset$ for all but at most one $i\in \mathcal{C}$, or  $E(G_i[B'])\subseteq \{v_1v_2\}$ for fixed vertices $v_1,v_2\in B'$ and all $i\in \mathcal{C}$.
\end{enumerate}
\end{lemma}
\begin{proof}
    It is routine to check that 
\begin{align}\notag
    \lceil\frac{n}{2}-1\rceil|B||\mathcal{C}'|&\leq \sum_{i\in \mathcal{C}'}|E({G_i}[A,B])|\\\notag
    &\leq |Y|(1-\delta^{\frac{1}{4}})|B|\delta^{\frac{1}{4}}|\mathcal{C}'|+|Y||B|(1-\delta^{\frac{1}{4}})|\mathcal{C}'|+(|A|-|Y|)|B||\mathcal{C}'|\\\label{eq:Y}
    &=(|A|-\delta^{\frac{1}{2}}|Y|)|B||\mathcal{C}'|.
\end{align}
It implies that $|Y|\leq \frac{\gamma n+1}{\sqrt{\delta}}\leq 4\sqrt{\delta}n$. Notice that Claim \ref{lemma4.1} and Claim \ref{conn} (i)-(ii) hold by setting  $V_{bad}=Y$ and $\mathcal{C}_2=\mathcal{C}'$. 

Since $\delta(\mathcal{G})\geq \frac{n}{2}-1$, we have $|E(G_i[A,B])|\geq |B|\lceil\frac{n}{2}-1\rceil$ for each $i\in \mathcal{C}'$, i.e., there are at most $|B|(\lfloor\frac{n+1}{2}\rfloor+\gamma n)-|B|\lceil\frac{n}{2}-1\rceil=|B|(\gamma n+1)$ non-edges between $A$ and $B$ in $G_i$. Hence $\sum_{i\in \mathcal{C}'}|E(G_i[Y,B])|\geq (|Y|-\gamma n-1)|B||\mathcal{C}'|$. 
By Lemma \ref{claim4.2}, there exist  $|Y|-\gamma n-1$ disjoint rainbow $P_3$ inside $\{G_i[Y,B]:i\in \mathcal{C}'\}$ with centers in $Y$. Let $\mathbf{P}_1$ be a set consisting of those rainbow $P_3$ and $Y_1$ be a subset of $Y$ consisting of the centers of them. Denote $Y_2:=Y\backslash Y_1$. Then $|Y_2|= \gamma n+1$. 

Recall that each vertex in $Y$ is adjacent to at least $\lceil\frac{n}{2}-1\rceil-(1-\delta^{\frac{1}{4}})|B|>4|Y|$ vertices in $A$ for at least $\delta^{\frac{1}{4}}|\mathcal{C}'|>\delta^{\frac{1}{4}}(1-\delta)n>4|Y|$ colors $i\in\mathcal{C}'$. Then there exist  $\gamma n+1$ disjoint rainbow $P_3$ inside $\{G_i[Y_2,A\backslash Y]:i\in \mathcal{C}'\backslash col(\mathbf{P}_1)\}$ with centers in $Y_2$, and let $\mathbf{P}_2$ be a set consisting of those rainbow $P_3$.  For each $y\in Y$, denote the rainbow $P_3$ in $\mathbf{P}_1\cup \mathbf{P}_2$ with center $y$ by $P_y:=y^1yy^2$ with colors $c_y^1$ and $c_y^2$. Let $\mathbf{P}=\mathbf{P}_1\cup \mathbf{P}_2=\{P_y:y\in Y\}$. 

It is routine to check that there exists a rainbow matching inside $\{G_i[A\backslash V(\mathbf{P}),B\backslash V(\mathbf{P})]: i\in \mathcal{C}''\}$, say $M$, such that $G_j[A\backslash V(\mathbf{P}\cup M),B\backslash V(\mathbf{P}\cup M)]$ contains a $3$-matching for each $j\in col(M)$. 
Furthermore, $G_{j}$ is $(13\sqrt{\delta},K_{\lfloor\frac{n}{2}\rfloor}\cup K_{\lceil\frac{n}{2}\rceil})$-extremal for all $j\in \mathcal{C}''\backslash col(M)$. 
 We proceed by considering the following three cases. 

{\bf Case 1.} $n$ is odd and $|\mathcal{C}''\backslash col(M)|$ is odd, or $n$ is even  and $|\mathcal{C}''\backslash col(M)|\geq 2$ is even.

Greedily choose two disjoint  rainbow paths $P_A$ and $P_B$ inside $\{G_i[A\backslash V(\mathbf{P}\cup M)]\cup G_i[B\backslash V(\mathbf{P}\cup M)]:i\in \mathcal{C}''\backslash col(M)\}$ with lengths 
$\lfloor\frac{|\mathcal{C}''\backslash col(M)|-1}{2}\rfloor$ and $\lceil\frac{|\mathcal{C}''\backslash col(M)|+1}{2}\rceil$ respectively, such that $V(P_A)\subseteq A$ and $V(P_B)\subseteq B$. 
In view of Claim \ref{conn} (i)-(ii), by using colors in $\mathcal{C}'$, one may connect all rainbow paths  in $\mathbf{P}_2$ and  $\mathbf{P}_1$ in turn to get a single rainbow path $P^1$ with $2|Y_1|+3|Y_2|$ vertices in $A$ and $2|Y_1|+|Y_2|$ vertices in $B$, whose endpoints are in different parts. 
Therefore, 
$|B\backslash V(P^1)|-|A\backslash V(P^1)|=2-\sigma$. Applying Claim \ref{conn} (i) again to connect $P^1$, all rainbow edges in $M$, $P_A$ and $P_B$ in turn, we get a single rainbow path $P^2$ with endpoints are in different parts. Clearly, $|A\backslash V(P^2)|=|B\backslash V(P^2)|$. 
Together with Claim \ref{lemma4.1}, we know $\mathcal{G}$ contains a transversal Hamilton cycle, a contradiction.   

{\bf Case 2.} $n$ is odd and  $|\mathcal{C}''\backslash col(M)|$ is even, or $n$ is even and $|\mathcal{C}''\backslash col(M)|$ is odd. 

Greedily choose two disjoint  rainbow paths $P_A$ and $P_B$ inside $\{G_i[A\backslash V(\mathbf{P}\cup M)]\cup G_i[B\backslash V(\mathbf{P}\cup M)]:i\in \mathcal{C}''\backslash col(M)\}$ with lengths $\lfloor\frac{|\mathcal{C}''\backslash col(M)|}{2}\rfloor$ and $\lceil\frac{|\mathcal{C}''\backslash col(M)|}{2}\rceil$ 
 respectively, such that $V(P_A)\subseteq A$ and $V(P_B)\subseteq B$. Suppose that there exists an $i_1\in \mathcal{C}'\cup col(M)$ such that 
$G_{i_1}[Y_2\cup B]\neq \emptyset$.  Choose $y_1y_2\in E(G_{i_1}[Y_2\cup B])$. By Claim \ref{conn} (i)-(ii), we connect all rainbow $P_3$ in $\mathbf{P}_2\backslash \{P_{y_1},P_{y_2}\}$, $\mathbf{P}_1$,  all rainbow edges in $M$ except the possible edge with color $i_1$ and $P_A,P_B$ in turn to get a single rainbow path $P^3$, whose endpoints are in different parts. Similar to Case 1, we have $|A\backslash V(P^3)|-|B\backslash V(P^3)|=2|\{y_1,y_2\}\cap Y_2|-1$. 

If $y_1,y_2\in Y_2$, then let $Q:=y_1^1y_1y_2y_2^2$ with $col(Q)=\{c_{y_1}^1,i_1,c_{y_2}^2\}$. If $y_1\in Y_2$ and $y_2\in B$, then one may assume $y_2\notin V(P^3)$ and let $Q:=y_1^1y_1y_2$ with $col(Q)=\{c_{y_1}^1,i_1\}$. If $y_1,y_2\in B$, then  we may assume $y_1,y_2\not\in V(P^3)$ and let $Q:=y_1y_2$ with color $i_1$. Next, connect $P^3$ with the rainbow path $Q$ by Claim \ref{conn} (i)-(ii), we get a rainbow path $P^4$ with endpoints $y_1^1,v_3$ (if $y_1,y_2\in Y_2$) or $u_3,y_3$ (otherwise). It is routine to check that $|A\backslash V(P^4)|=|B\backslash V(P^4)|$. 
Together with Claim \ref{lemma4.1}, one obtains that  $\mathcal{G}$ contains a transversal Hamilton cycle, a contradiction.  Hence,  $G_i[Y_2\cup B]=\emptyset$ for all  $i\in \mathcal{C}'\cup col(M)$. 
Recall that $\delta(G_i)\geq \lceil\frac{n}{2}-1\rceil$. This implies that $N_{G_i}(w)=A\backslash Y_2$ for all $i\in \mathcal{C}'\cup col(M)$ and all $w\in Y_2\cup B$. 
Move vertices in $Y_2$ from $A$ to $B$, we get $|A|=\lceil\frac{n}{2}-1\rceil$, $|B|=\lfloor\frac{n}{2}+1\rfloor$ and $G_i[B]=\emptyset$ for each $i\in \mathcal{C}'\cup col(M)$. By Lemmas \ref{lemma4.4} and \ref{lemma5.2}, our desired result holds.

{\bf Case 3.} $n$ is even and $|\mathcal{C}''\backslash col(M)|=0$. 

By a similar discussion as Case 2 and the proof of Lemma \ref{lemma5.2}, we obtain that $\mathcal{G}[Y_2\cup B]$ contains no rainbow $P_3$ or $2P_2$. Hence either $G_{i}[Y_2\cup B]=\emptyset$ for all but at most one unique $i\in \mathcal{C}$ or $E(G_i[Y_2\cup B])\subseteq \{uv\}$ for fixed $u,v$ and all $i\in \mathcal{C}$. Move vertices in $Y_2$ to $B$, we get $|A|=\frac{n}{2}-1$, $|B|=\frac{n}{2}+1$. Therefore, either $G_i[B]=\emptyset$ for all but at most one $i\in \mathcal{C}$, or $E(G_i[B])\subseteq \{uv\}$ for all $i\in \mathcal{C}$. Recall that $\delta(\mathcal{G})\geq \frac{n}{2}-1$. In both cases, $\mathcal{G}$ does not contain transversal Hamilton cycles, as desired.
\end{proof}

\begin{lemma}\label{lemma5.3}
Assume $0<\frac{1}{n}\ll \delta\ll 1$.
 Let $\mathcal{C}$ be a set of $n$ colors and $\mathcal{G}=\{G_i:i\in \mathcal{C}\}$ be a collection of  graphs with common vertex set $V$ of size $n$. Let $A\cup B$ be a partition of $V$ with  $|A|=\frac{n}{2}$, and let $\mathcal{C}'\cup \mathcal{C}''$ be a partition of $\mathcal{C}$ with  $|\mathcal{C}''|\leq \delta n$.  Assume that $\delta(\mathcal{G})\geq \frac{n}{2}-1$ and $G_i[B]=\emptyset$ for all $i\in \mathcal{C}'$. If  $\mathcal{G}$ does not contain transversal Hamilton cycles, then  one of the following holds:
    \begin{itemize}
        \item $\mathcal{G}$ is a spanning collection of $\mathcal{H}_{n-t}^t$ for some odd integer $t$,
        \item  there exists a  partition $A'\cup B'$ of $V$ with $|A'|=\frac{n}{2}-1$,  such that either $G_i[B']=\emptyset$ for all but at most one $i\in \mathcal{C}$, or  $E(G_i[B'])\subseteq \{uv\}$ for fixed vertices $u,v\in B'$ and all $i\in \mathcal{C}$.
    \end{itemize}
\end{lemma}
\begin{proof}
Let  $Y=\{v\in A: d_{G_i}(v,B)\leq (1-\delta^{\frac{1}{4}})|B|\ \textrm{for at least}\ \delta^{\frac{1}{4}}|\mathcal{C}'|\ \textrm{colors}\ i\in \mathcal{C}'\}$. Then \allowdisplaybreaks
\begin{align*}
    (\frac{n}{2}-1)|B||\mathcal{C}'|
    \leq&\sum_{i\in \mathcal{C}'}|E(G_i[A,B])|\\
    =&\sum_{v\in Y}\sum_{i\in \mathcal{C}'}d_{G_i}(v,B)+\sum_{v\in A\backslash Y}\sum_{i\in \mathcal{C}'}d_{G_i}(v,B)\\
    \leq& |Y|(1-\delta^{\frac{1}{4}})|B|\delta^{\frac{1}{4}}|\mathcal{C}'|+|Y||B|(1-\delta^{\frac{1}{4}})|\mathcal{C}'|+(\frac{n}{2}-|Y|)|B||\mathcal{C}'|\\
    =&(\frac{n}{2}-\beta|Y|)|B||\mathcal{C}'|.
\end{align*}
It follows that $|Y|\leq \frac{1}{\sqrt{\delta}}<\sqrt{\delta} n$. If $|Y|\geq 1$, then our result holds by Lemma  \ref{Y}. In what follows, we only consider $|Y|=0$. It is straightforward to check that Claim \ref{lemma4.1} and Claim \ref{conn} (i)-(ii) hold by setting  $V_{bad}=Y$ and $\mathcal{C}_2=\mathcal{C}'$.

Choose a maximal rainbow matching inside $\{G_i[A,B]: i\in \mathcal{C}''\}$, say $M$, such that $G_j[A,B]$ contains a $2$-matching for each $j\in col(M)$. 
Furthermore, $G_{j}$ is $(2\sqrt{\delta},K_{\lfloor\frac{n}{2}\rfloor}\cup K_{\lceil\frac{n}{2}\rceil})$-extremal for all $j\in \mathcal{C}''\backslash col(M)$. 
Greedily choose two disjoint rainbow paths $P_A$ and $P_B$ inside $\{G_i[A\backslash V(M)]\cup G_i[B\backslash V(M)]:i\in \mathcal{C}''\backslash col(M)\}$ with lengths $x_A$ and $x_B$ respectively, such that $V(P_A)\subseteq A$ and $V(P_B)\subseteq B$. 
In view of Claim \ref{conn} (i)-(ii), by using colors in $\mathcal{C}'$, one may connect all rainbow edges in $M$, $P_A$ and $P_B$ into a single rainbow path $P$, whose endpoints are denoted by $u_1\in A$ and  $v_1\in B$. 

If $|\mathcal{C}''\backslash col(M)|$ is even, then let 
$x_A=x_B=\frac{|\mathcal{C}''\backslash col(M)|}{2}$. Thus, 
$|B\backslash V(P)|=|A\backslash V(P)|$.  Together with Claim \ref{lemma4.1}, we know $\mathcal{G}$ contains a transversal Hamilton cycle, a contradiction. Hence $|\mathcal{C}''\backslash col(M)|$ is odd. Now, we give the following claim.

\begin{claim}\label{claim7}
    \begin{enumerate}
        \item[{\rm (i)}] $G_i[A]\cup G_i[B]=\emptyset$ for all $i\in \mathcal{C}'\cup col(M)$,
        \item[{\rm (ii)}] $G_i[A,B]=\emptyset$ for all $i\in \mathcal{C}''\backslash col(M)$. 
    \end{enumerate}
\end{claim}
\begin{proof}[Proof of Claim \ref{claim7}]
    We only give the proof of $G_i[A]=\emptyset$ for all $i\in \mathcal{C}'\cup col(M)$, the other two statements can be proved by similar arguments.

    Suppose that $G_{i_1}[A]\neq \emptyset$ for some $i_1\in \mathcal{C}'\cup col(M)$. Choose $w_1w_2\in E(G_{i_1}[A])$. If $i_1\in col(M)$, then delete the edge with color $i_1$ from $M$. Hence we can assume that $i_1\notin col(P)$ and $w_1,w_2\notin V(P)$. Using Claim~\ref{conn}~(ii) to connect $w_1w_2$ and the rainbow path $P$, we get a rainbow path $P'$ with endpoints $w_1$ and $v_1$. Let $x_A=\frac{|\mathcal{C}''\backslash col(M)|-1}{2}$ and $x_B=\frac{|\mathcal{C}''\backslash col(M)|+1}{2}$. Hence $|B\backslash V(P')|=|A\backslash V(P')|$.  It follows from Claim \ref{lemma4.1} that  $\mathcal{G}$ contains a transversal Hamilton cycle, a contradiction.
\end{proof}

Based on Claim \ref{claim7}, we know $\mathcal{G}$ is a spanning collection of $\mathcal{H}_{n-t}^t$ for some odd $t$, as desired. 
\end{proof}

{{\begin{lemma}\label{lemma5.7}
Assume $0<\frac{1}{n}\ll\delta\ll1$ and $0\leq \gamma\leq3{\delta}$. Let $\mathcal{C}$ be a set of $n$ colors, and $\mathcal{G}=\{G_i:i\in \mathcal{C}\}$ be a collection of graphs with common vertex set $V$ of size $n$ such that  $\delta(\mathcal{G})\geq \frac{n}{2}-1$. Let $A\cup B$ be a partition of $V$ with  $|A|=\lfloor\frac{n+1}{2}\rfloor+\gamma n$, and let $\mathcal{C}'\cup \mathcal{C}''$ be a partition of $\mathcal{C}$ with  $|\mathcal{C}''|\leq \delta n$.    Assume that  $\mathcal{G}$ does not contain transversal Hamilton cycles and $G_i[B]=\emptyset$ for all $i\in \mathcal{C}'$. 
\begin{enumerate}
    \item[{\rm (i)}] If $n$ is odd, then $\mathcal{G}$ is the half-split graph collection.
    \item[{\rm (ii)}] If $n$ is even, then one of the following holds:
    \begin{itemize}
        \item $\mathcal{G}$ is a spanning collection of $\mathcal{H}_{n-t}^t$ for some odd integer $t$,
        \item there exists a new partition $A'\cup B'$ of $V$ with $|A'|=\frac{n}{2}-1$ such that either $G_i[B']=\emptyset$ for all but at most one $i\in \mathcal{C}$, or  $E(G_i[B'])\subseteq \{v_1v_2\}$ for two  fixed vertices $v_1,v_2\in B'$ and all $i\in \mathcal{C}$.
    \end{itemize}
\end{enumerate}
\end{lemma}
\begin{proof}
Define $Y=\{v\in A: d_{G_i}(v,B)\leq (1-\delta^{\frac{1}{4}})|B|\ \textrm{for at least}\ \delta^{\frac{1}{4}}|\mathcal{C}'|\ \textrm{colors}\  i\in \mathcal{C}'\}$. In view of  \eqref{eq:Y}, we have $|Y|\leq 4\sqrt{\delta}n$. 
If $|Y|>\gamma n$, then our result holds by Lemma \ref{Y}. In what follows, we only consider $|Y|\leq \gamma n$. Therefore, $|B|\leq |A|-2|Y|-\sigma$ where $\sigma=0$ if $n$ is even and $\sigma=1$ otherwise.  Notice that Claim \ref{lemma4.1} and Claim \ref{conn} (i)-(ii) hold by setting  $V_{bad}=Y$ and $\mathcal{C}_2=\mathcal{C}'$. 

Recall that each vertex in $Y$ is adjacent to at least $\lceil\frac{n}{2}-1\rceil-(1-\delta^{\frac{1}{4}})|B|>5|Y|$ vertices in $A$ for at least $\delta^{\frac{1}{4}}|\mathcal{C}'|>\delta^{\frac{1}{4}}(1-\delta)n>5|Y|$ colors $i\in\mathcal{C}'$. Then there exist  $|Y|$ disjoint rainbow paths $P_3$ with centers in $Y$ and endpoints in $A\backslash Y$ using colors in $\mathcal{C}'$. 
In the graph collection $\mathcal{G}[A\backslash Y]$, we extend those rainbow paths or choose other disjoint rainbow paths into a set of disjoint maximal rainbow paths. Let $\mathbf{P}=\{Q_1,Q_2,\ldots,Q_t\}$ be a set consisting of all disjoint  rainbow paths in the above, each of which has length $s_i\ (1\leq i\leq t)$. 


{\bf Case 1.} $|B|\geq |A|-(s_1+\cdots+s_t)-\sigma$. 

In this case, there exists a set $\mathbf{P}'=\{Q_1',Q_2',\ldots,Q_{\ell}'\}$ such that $|A\backslash V(\mathbf{P}')|-\sigma=|B|-\ell$, where $Y\subseteq V(\mathbf{P}')$ and the endpoints of $Q_i'$ are not in $Y$ for all $i\in [\ell]$. Assume $|E(Q_i')|=s_i'$ for each $i\in [\ell]$. 
By Claim \ref{conn} (ii), we  can connect all rainbow paths of  $\mathbf{P}'$ into a single rainbow path $P^1$, whose endpoints are in different parts. 
Since $|A|-|B|=2\gamma n+\sigma$, we have $\ell\leq s_1'+\cdots+s_{\ell}'= 2\gamma n+\sigma$. Therefore, 
$|E(P^1)|\leq 6\gamma n+3\sigma$ and $|A\backslash V(P^1)|-\sigma=|B\backslash V(P^1)|$. 

Choose a maximal rainbow matching inside $\{G_i[A\backslash V(P^1),B\backslash V(P^1)]: i\in \mathcal{C}''\backslash col(P^1)\}$, say $M$, such that $G_j[A\backslash V(P^1\cup M),B\backslash V(P^1\cup M)]$ contains a $2$-matching for each $j\in col(M)$. 
Furthermore, $G_{j}$ is $(6(\gamma+\delta),K_{\lfloor\frac{n}{2}\rfloor}\cup K_{\lceil\frac{n}{2}\rceil})$-extremal for all $j\in \mathcal{C}''\backslash col(P^1\cup M)$.  In fact, when construct $P^1-B-Y$,  it is possible to use colors in $\mathcal{C}''\backslash col(P^1\cup M)$ before other colors. Hence either  $col(P^1-B-Y)\subseteq \mathcal{C}''\backslash col(M)$ or $\mathcal{C}''\backslash col(M) \subseteq col(P^1-B-Y)$. 

Based on Claim \ref{conn} (i), one may  connect $P^1$, all rainbow edges in $M$ into a single rainbow path $P^2$, whose endpoints are in different parts. Clearly, $|E(P^2)|\leq 6\gamma n+3\sigma+4{\delta} n$. Next, we are to choose two disjoint  rainbow paths $P_A$  and $P_B$ inside $\{G_i[A\backslash V(P^1\cup M)]\cup G_i[B\backslash V(P^1\cup M)]:i\in \mathcal{C}''\backslash col(P^1\cup M)\}$ such that $V(P_A)\subseteq A$ and $V(P_B)\subseteq B$ respectively, whose lengths are determined by the parity of $|\mathcal{C}''\backslash col(P^1\cup M)|$. Clearly, if $n$ and $|\mathcal{C}''\backslash col(P^1\cup M)|$ have the same parity, then by Claim \ref{lemma4.1}, we know $\mathcal{G}$  contains a transversal Hamilton cycle, a contradiction. Hence $n$ and $|\mathcal{C}''\backslash col(P^1\cup M)|$ have different parity. By similar arguments as the proof of Theorem \ref{th3} (Step 4 in Case 1), we can show that the following claim.
\begin{claim}\label{claim8}
    \begin{enumerate}
\item[{\rm (i)}] If 
$|\mathcal{C}''\backslash col(P^1\cup M)|=0$, then $E(G_i[A])=E(G_i[V(\mathbf{P}')\cap A])$ for all $i\in \mathcal{C}$.
\item[{\rm (ii)}] $G_i[A,B]=\emptyset$ for all $i\in \mathcal{C}''\backslash col(P^1\cup M)$.
    \end{enumerate}
\end{claim}

If $|\mathcal{C}''\backslash col(P^1\cup M)|=0$, then $n$ is odd. By  Claim \ref{claim8} (i), one has $|A|\geq|B|\geq \frac{n-1}{2}$. Hence $|A|=\frac{n+1}{2}$ and $|B|=\frac{n-1}{2}$. This implies that $P^1$ is a null graph. Thus, $G_i[A]=\emptyset$ for all $i\in \mathcal{C}$. Together with Lemma \ref{lemma4.4}, our desired result holds. 

If $|\mathcal{C}''\backslash col(P^1\cup M)|\geq 1$, then based on Claim \ref{claim8} (ii), we know $n$ is even, $|A|=|B|=\frac{n}{2}$ and  $G_i[B]=\emptyset$ for all but at most $\delta n$ colors in ${\mathcal{C}}$. 
By Lemma \ref{lemma5.3}, our desired result holds.

{\bf Case 2.} $|B|<|A|-(s_1+\cdots+s_t)-\sigma$. 

In this case, $\sum_{i=1}^t s_i<2\gamma n$ and so $\sum_{i=1}^t \frac{s_i}{2}\leq \gamma n-\frac{1}{2}$. Using  Claim \ref{conn} (ii) and $t$ vertices in $B$, one may connect all rainbow paths of  $\mathbf{P}$ into a single rainbow paths $P^3$, whose endpoint are in different parts. Let $\Tilde{A}=A\backslash V(P^3)$, $\Tilde{B}=B\backslash V(P^3)$ and $\Tilde{\mathcal{C}}=\mathcal{C}\backslash col(P^3)$. Then $|\Tilde{A}|=|A|-(s_1+\cdots+s_t+t)$, $|\Tilde{B}|=|B|-t$ and $|\Tilde{\mathcal{C}}|=|\mathcal{C}|-(s_1+\cdots+s_t+2t-1)$. For convenience, we assume that $Q_i=u^i_0u^i_1\ldots u^i_{s_i}$ for each $i\in [t]$. 

Let $w$ be an arbitrary vertex in $\Tilde{A}$ and $c_1,c_2$ be two colors  in $\mathcal{C}\backslash col(\mathbf{P})$. Without loss of generality, assume that $d_{G_{c_1}}(w,A\backslash \Tilde{A})\leq d_{G_{c_2}}(w,A\backslash \Tilde{A})$.  Clearly, in $G_{c_1}$ and $G_{c_2}$, $w$ is adjacent to at least $\lceil\frac{n}{2}-1\rceil-(\lceil\frac{n-1}{2}\rceil-\gamma n)=\gamma n-1+\sigma$ vertices in $A\backslash \Tilde{A}$ and it cannot adjacent to the pendant vertices of $Q_i$ for all $i\in 
[t]$ (by the maximality of $\mathbf{P}$). Furthermore, if $u^i_j\in N_{G_{c_1}}(w)$ for some $i\in [t]$ and $j\in [s_i]$, then $u^i_{j-1},u^i_{j+1}\not \in N_{G_{c_2}}(w)$.  
Denote by $t_{odd}$ the number of rainbow paths in $\mathbf{P}$ with odd lengths. Thus, 
\begin{align}\label{eq:5.1}
\gamma n-1+\sigma\leq d_{G_{c_1}}(w,A\backslash \Tilde{A})\leq \sum_{i=1}^t \lceil\frac{s_i-1}{2}\rceil =\sum_{i=1}^t \frac{s_i}{2}-\frac{1}{2}t_{odd}\leq \gamma n-\frac{1}{2}-\frac{1}{2}t_{odd}.
\end{align}
This implies that $t_{odd}\leq 1-2\sigma$. Therefore, $n$ is even and $t_{odd}\leq 1$.

$\bullet$ $t_{odd}=0$.  Notice that $\sum_{i=1}^t \frac{s_i}{2}$ and $\gamma n$ are integers, then  $d_{G_{c_1}}(w,A\backslash \Tilde{A})=\sum_{i=1}^t \frac{s_i}{2}=\gamma n-1$. It follows that $N_{G_{c_1}}(w,A\backslash \Tilde{A})=\cup_{i\in [t]}\{u_1^i,u_3^i,\ldots,u_{s_i-1}^i\}$. Hence, if $i\in [t]$ and $j$ is even, then $u_j^i$ cannot adjacent to $w$ in $G_{c_2}$. Recall that $d_{G_{c_2}}(w,A\backslash \Tilde{A})\geq \gamma n-1=\sum_{i=1}^t \frac{s_i}{2}$. We have $N_{G_{c_1}}(w,A\backslash \Tilde{A})=N_{G_{c_2}}(w,A\backslash \Tilde{A})$. Thus, for any $G_i$ with $i\in \mathcal{C}\backslash col(\mathbf{P})$ and any $w\in \Tilde{A}$, we have $N_{G_i}(w)\cap (A\backslash \Tilde{A})=\cup_{i\in [t]}\{u_1^i,u_3^i,\ldots,u_{s_i-1}^i\}$.

Now, move all vertices in $\cup_{i\in [t]}\{u_1^i,u_3^i,\ldots,u_{s_i-1}^i\}$ from $A$ to $B$. Then $|A|=\frac{n}{2}+1$,  $|B|=\frac{n-1}{2}-1$ and $G_i[A]$ is empty for all $i\in \mathcal{C}\backslash col(\mathbf{P})$. Notice that $|col(\mathbf{P})|=\sum_{i=1}^t s_i\leq 2\gamma n$.  
Together with Lemma \ref{lemma5.2}, our desired result holds.

$\bullet$ $t_{odd}=1$. Then all inequalities in  \eqref{eq:5.1} must be equalities. Thus, 
$d_{G_{c_1}}(w,A\backslash \Tilde{A})=\sum_{i=1}^t \frac{s_i}{2}-\frac{1}{2}= \gamma n-1$ and  therefore $d_{G_{c_1}}(w,V(Q_i))=\lceil\frac{s_i-1}{2}\rceil$ for all $i\in [t]$. Without loss of generality, assume that $Q_1$ is the unique rainbow path with odd length. By a similar discussion as the case for $t_{odd}=0$, we know for every $i\in \mathcal{C}\backslash col(\mathbf{P})$ and every $w\in \Tilde{A}$, $N_{G_i}(w,\cup_{i\in [t]\backslash \{1\}}V(Q_i))=\cup_{i\in [t]\backslash\{1\}}\{u_1^i,u_3^i,\ldots,u_{s_i-1}^i\}$. 

Notice that in each $G_i$ with $i\in \mathcal{C}\backslash col(\mathbf{P})$, each vertex in $\Tilde{A}$ has $\frac{s_1-1}{2}$ neighbors in $V(Q_1)\backslash \{u_{0}^1,u_{s_1}^1\}$ and it cannot adjacent to two adjacent vertices in $Q_1$ with two different colors. Hence for each $i\in \mathcal{C}\backslash col(\mathbf{P})$, $N_{G_i}(w)$ must be one of the following sets: 
\begin{align*}
    &\{u_1^1,u_3^1,\ldots,u_{s_1-2}^1\},\ \{u_2^1,u_4^1,\ldots,u_{s_1-1}^1\},\\
    &\{u_1^1,u_3^1,\ldots,u_{j-3}^1,u_{j}^1,u_{j+2}^1,\ldots,u_{s_1-1}^1\}\ \textrm{for some even integer}\ j\ {\textrm{with}}\ 4\leq j\leq s_1-1.
\end{align*}
It is routine to check that $N_{G_{c_1}}(w,V(Q_1))=N_{G_{c_2}}(w,V(Q_1))$. If there are two vertices $v,v'\in A\backslash \Tilde{A}$ and colors $i_1,i_2,i_3,i_4\in \mathcal{C}\backslash col(\mathbf{P})$ such that $u_{j-1}^1,u_{j+1}^1\in N_{G_{i_1}}(v)=N_{G_{i_2}}(v)$ and $u_j^1,u_{j+2}^1\in N_{G_{i_3}}(v')=N_{G_{i_4}}(v')$, then we can find a longer rainbow path $u_0^1u_1^1\ldots u_{j-1}^1vu_{j+1}^1u_j^1v'u_{j+2}^1\ldots u_{s_i}^1$, where $u_{j-1}^1v\in E(G_{i_1}),vu_{j+1}^1\in E(G_{i_2}),u_j^1v'\in E(G_{i_3})$ and $v'u_{j+2}^1\in E(G_{i_4})$, a contradiction. 
Notice that for two different vertices $v$ and $v'$, it is possible that $N_{G_i}(v,V(Q_1))=\{u_1^1,u_3^1,\ldots,u_{s_1-2}^1\}$ and $N_{G_i}(v',V(Q_1))=\{u_1^1,u_3^1,\ldots,u_{s_1-4}^1,u_{s_1-1}^1\}$ hold  for all $i\in \mathcal{C}\backslash col(\mathbf{P})$. Hence  $\bigcup_{i\in \mathcal{C}\backslash col(\mathbf{P})}N_{G_i}(\Tilde{A},V(Q_1))$ is contained in one of the following sets:
\begin{align*}
    &\{u_1^1,u_3^1,\ldots,u_{s_1-2}^1,u_{s_1-1}^1\},\ \ \{u_2^1,u_4^1,\ldots,u_{s_1-1}^1\},\\
    &\{u_1^1,u_3^1,\ldots,u_{j_0-3}^1,u_{j_0}^1,u_{j_0+2}^1,\ldots,u_{s_1-1}^1\}\ \textrm{for a fixed even integer}\ j_0\ {\textrm{with}}\ 4\leq j_0\leq s_1-1.
\end{align*}

Move all vertices in $\bigcup_{i\in \mathcal{C}\backslash col(\mathbf{P})}N_{G_i}(\Tilde{A},A\backslash \Tilde{A})$ from   $A$ to $B$, then $G_i[A]$ is empty for all $i\in \mathcal{C}\backslash col(\mathbf{P})$. What's more, either $|A|=\frac{n}{2}+1$ and $|B|=\frac{n}{2}-1$, or  $|A|=\frac{n}{2}$ and $|B|=\frac{n}{2}$. 
Together with Lemmas \ref{lemma5.2} and \ref{lemma5.3}, our desired result holds.

This completes the proof of Lemma \ref{lemma5.7}.
\end{proof}}}

Combining Lemmas \ref{lemma4.4}, \ref{lemma5.2}, \ref{lemma5.3} and \ref{lemma5.7}, we obtain Theorem  \ref{lemma5.4}.


\section{Concluding remarks}\label{sec6}
In this paper, we studied the Dirac-type conditions for transversal Hamilton paths and transversal  Hamilton cycles in graph collections. As we know, there are many sufficient conditions to guarantee the existence of Hamilton cycles in a graph, such as Ore's condition \cite{1960Ore}, P{\'o}sa's condition \cite{1962PosaCon} and so on. It is natural to consider the Ore-type condition or P{\'o}sa-type condition for transversal Hamilton cycles in a graph collection (which was already proposed in  \cite{2022liluyi}). Furthermore, motivated by the stability results for transversal Hamilton cycles under Dirac-type condition \cite{2024cheng}, one may study the stability result for transversal Hamilton cycles under Ore-type condition or P{\'o}sa-type condition.

The \textit{closure} of an $n$-vertex graph $G$, denoted $C(G)$, is the graph with vertex set $V(G)$ obtained from $G$ by iteratively adding edges joining pairs of nonadjacent vertices whose degree sum is at least $n$, until no such pair remains. Bondy and Chv\'{a}tal \cite{Bondy} showed that an $n$-vertex graph is Hamiltonian if and only if its closure is Hamiltonian. We propose the following question.
\begin{question}
    Let $\mathcal{G}=\{G_1,\ldots,G_n\}$ be a collection of graphs with common vertex set $V$ of size $n$. If $\{C(G_i):i\in [n]\}$ contains a transversal Hamilton cycle, does   $\mathcal{G}$ contain a transversal Hamilton cycle?
\end{question}

Transversal generalizations were recently considered for digraphs. Cheng, Han, Wang and Wang \cite{2023chengspan} established a minimum degree condition to guarantee the existence transversal tournament factors in digraph collections. 
Chakraborti, Kim, Lee and Seo \cite{2023Tournament} proved the existence of transversal Hamilton paths and cycles in tournament collections. 

For a digraph $D$, define $\delta^+(D)=\{d_D^+(v):v\in V(D)\}$ and $\delta^-(D)=\{d_D^-(v):v\in V(D)\}$ to be the minimum out-degree and minimum in-degree of $D$ respectively. The minimum semi-degree of $D$ is $\delta^0(D):=\min\{\delta^+(D),\delta^-(D)\}$. Ghouila-Houri \cite{Houri} proved that any $n$-vertex digraph $D$ with $\delta^0(D)\geq \frac{n}{2}$ contains a directed Hamilton cycle. Chakraborti,  Kim, Lee and Seo \cite{2023Tournament} proposed that it would  be interesting to consider a transversal version of the above  theorem. Woodall \cite{Woodall} showed that if  $D$ is an $n$-vertex digraph  satisfying $d_D^+(u)+d_D^-(v)\geq {n}$ for all pairs of vertices $\{u,v\}$ with $\overrightarrow{uv}\notin A(D)$, then $D$ contains a directed Hamilton cycle. Hence it is natural to consider the transversal  version of this result.



\bibliographystyle{abbrv}
\bibliography{THPC}

\end{document}